\documentclass[final,onefignum,oneeqnum,onetabnum]{siamart171218}

\usepackage{fullpage}

\usepackage{mathtools}
\usepackage{mhsetup}

\def\bcswitch{\left\{\renewcommand{\arraystretch}{1.2}\begin{array}{c@{,~}c}}

\def\ecswitch{\end{array}\right.}

\def\bcswitchs{\left\{\renewcommand{\arraystretch}{1.2}\begin{array}{c}}

\def\ecswitchs{\end{array}\right.}

\newtheorem{remark}{Remark}

\usepackage{lipsum}
\usepackage{amsfonts}
\usepackage{graphicx}
\usepackage{epstopdf}
\usepackage{algorithmic}

\ifpdf
  \DeclareGraphicsExtensions{.eps,.pdf,.png,.jpg}
\else
  \DeclareGraphicsExtensions{.eps}
\fi

\newcommand{\TheTitle}{Inexact cuts in Stochastic Dual Dynamic Programming}
\newcommand{\TheAuthors}{}

\title{{\TheTitle}}

\author{
  Vincent Guigues\thanks{School  of Applied Mathematics, Funda\c{c}\~ao Getulio Vargas,
190 Praia de Botafogo, Rio de Janeiro, Brazil,
    (\email{vguigues@fgv.br}).}
}

\usepackage{amsopn}

\ifpdf
\hypersetup{
  pdftitle={\TheTitle},
  pdfauthor={\TheAuthors}
}
\fi

\begin{document}

\maketitle

\begin{abstract}
We introduce an extension of Stochastic Dual Dynamic
Programming (SDDP) to solve stochastic convex  dynamic programming equations.
This extension applies when some or all primal and dual subproblems to be solved along the
forward and backward passes of the method
are solved with bounded errors (inexactly). This inexact variant of SDDP is described both for linear 
problems (the corresponding variant being denoted by ISDDP-LP) and nonlinear problems (the corresponding variant
being denoted by ISDDP-NLP).
We prove convergence theorems for ISDDP-LP and ISDDP-NLP both for  bounded and asymptotically vanishing errors.
Finally, we present the results of numerical experiments comparing SDDP and ISDDP-LP on a
portfolio problem with direct transaction costs modelled as a multistage stochastic linear optimization problem. 
On these experiments, ISDDP-LP allows us to obtain a good policy faster than SDDP.
\end{abstract}



\section{Introduction}

Stochastic Dual Dynamic Programming (SDDP) is an
extension of
 the nested decomposition method \cite{birgemulti} to solve 
some $T$-stage stochastic programs, pioneered by \cite{pereira}.
Originally, in \cite{pereira}, it was presented to solve Multistage Stochastic
Linear Programs (MSLPs). Since many real-life applications in, e.g., finance and engineering,
can be modelled by such problems, until recently most papers on SDDP and related
decomposition methods, including theory papers, focused on enhancements of the method for MSLPs.
These enhancements include risk-averse SDDP \cite{shapsddp}, \cite{guiguesrom12} \cite{guiguesrom10}, \cite{philpmatos},
\cite{kozmikmorton}, \cite{shaptekaya} and a convergence proof of SDDP in \cite{philpot}
and of variants incorporating cut selection in \cite{guiguesbandarra19}.

However, SDDP can be applied to solve nonlinear stochastic convex dynamic programming equations. For such problems,
the convergence of the method was proved recently in \cite{lecphilgirar12} for risk-neutral problems, in \cite{guiguessiopt2016}
for risk-averse problems, and in \cite{guilejtekregsddp} for a regularized variant.

To the best of our knowledge, all studies on SDDP
rely on the assumption that all primal and
dual subproblems solved in the forward and backward passes of the method are solved exactly. However, when SDDP is applied to nonlinear problems,
only approximate solutions are available for the subproblems solved in the forward and backward passes of the algorithm.
Additionally, it is known (see for instance the numerical experiments
in \cite{guiguesejor2017, guiguesbandarra19, guilejtekregsddp}) that for both linear and nonlinear Multistage Stochastic Programs (MSPs),
for the first iterations of the method and especially for the first stages, the cuts computed can be quite
distant from the corresponding recourse function in the neighborhood of the trial point at which the cut was computed, making this cut quickly
dominated by other "more relevant" cuts in this neighborhood. Therefore, it makes sense, for both nonlinear and linear MSPs,  to try and solve more quickly and less accurately (inexactly)
all subproblems of the forward and backward passes corresponding to the first iterations, especially for the first stages, and to increase the precision of the computed solutions as the algorithm progresses. 

In this context, the objective of this paper is to design inexact variants of SDDP 
that take this fact into account. These inexact variants of SDDP are described both for linear 
problems (the corresponding variant being denoted by ISDDP-LP) and nonlinear problems (the corresponding variant
being denoted by ISDDP-NLP).

While the idea behind these inexact variants of SDDP is simple and the motivations are clear, the description and convergence analysis of ISDDP-NLP
applied to the class of nonlinear programs introduced in \cite{guiguessiopt2016}
require solving the following problems of convex analysis, interesting per se, and which, to the best of our knowledge, had not been discussed so far
in the literature:
\begin{itemize}
\item SDDP applied to the general class of nonlinear programs introduced in \cite{guiguessiopt2016} relies on a formula for the subdifferential of the value function $\mathcal{Q}(x)$ of a convex optimization problem
of form: 
\begin{equation}\label{qdefintro}
\mathcal{Q}(x)=\left\{
\begin{array}{l}
\inf_{y \in \mathbb{R}^{n}} \;f(y,x)\\
y\in Y \;:\;Ay+Bx=b,\;g(y,x)\leq 0,
\end{array}
\right.
\end{equation}
where $Y \subseteq \mathbb{R}^n$ is nonempty and convex, $f:\mathbb{R}^n \small{\times} \mathbb{R}^m \rightarrow \mathbb{R} \cup \{+\infty\}$ is
convex, lower semicontinuous, and proper, and the components of 
$g$ are convex lower semicontinuous functions.
Formulas for the subdifferential $\partial \mathcal{Q}(x)$ are given in \cite{guiguessiopt2016}.
These formulas are based on the assumption that primal and dual solutions to \eqref{qdefintro} are available.
When only approximate $\varepsilon$-optimal primal and dual solutions are available for \eqref{qdefintro}
written with $x=\bar x$, we derive in  Propositions \ref{varprop1} and \ref{varprop2} formulas for 
affine lower bounding functions $\mathcal{C}$ for $\mathcal{Q}$, that we call inexact cuts,
such that the distance $\mathcal{Q}( \bar x ) - \mathcal{C}( \bar x)$ between the 
values of $\mathcal{Q}$ and of the cut at $\bar x$ is bounded from above by a known function $\varepsilon_0$ of the problem
parameters.
Of course, we would like $\epsilon_0$ to be as small as possible
and we have $\varepsilon_0=0$ when $\varepsilon=0$.
\item We provide conditions ensuring that $\varepsilon$-optimal dual solutions to a convex nonlinear optimization problem
are bounded. Proposition \ref{propboundnormepsdualsol} gives an analytic formula for an upper bound on the norm of these $\varepsilon$-optimal dual solutions.
\item We show in Proposition \ref{propvanish1dual}
that if we compute inexact cuts for a sequence $({\underline{\mathcal{Q}}}^k)$  of value functions of  form \eqref{qdefintro}
(with objective functions $f^k$ of special structure)
at a sequence of points $(x^k)$ on the basis of $\varepsilon^k$-optimal primal and dual solutions
with $\lim_{k \rightarrow +\infty} \varepsilon^k = 0$, then
the distance between the inexact cuts and the value functions at these points $x^k$  converges to 0 too.
This result is very natural but some constraint qualifications
are needed (see Proposition \ref{propvanish1dual}).
\end{itemize}
When optimization problem \eqref{qdefintro} is linear, i.e., when $\mathcal{Q}$ is the value function 
of a linear program, inexact cuts can easily be obtained from approximate dual solutions
since the dual objective is linear in this case.
This observation allows us to build inexact cuts for ISDDP-LP and
was used in \cite{philpzakinex} where inexact cuts are combined with
Benders Decomposition \cite{benderscut} to solve two-stage stochastic linear programs.
In this sense, ISDDP-LP can be seen as an extension of \cite{philpzakinex} 
replacing 
two-stage stochastic linear problems by MSLPs.
In integer programming, inexact master solutions are also commonly used in Benders-like methods \cite{danieldevine},
including  SDDiP, a variant of SDDP to solve multistage stochastic linear programs with integer variables
introduced in \cite{shabbirzou}.

The outline of the paper is as follows. 
Section \ref{sec:computeinexactcuts} provides analytic formulas for computing inexact cuts for value function
$\mathcal{Q}$
of optimization problem \eqref{qdefintro}. In Section \ref{sec:boundingmulti}, we provide an explicit bound for the norm of $\varepsilon$-optimal dual solutions.
Section \ref{sec:insddpl} introduces and studies ISDDP-LP method. The class of problems to which this method applies and the algorithm are
described in Subsection \ref{sec:insddpl1}. 
In Section \ref{conv-analsddp}, we provide a convergence theorem (Theorem \ref{convisddplp}) for ISDDP-LP when errors are bounded
and show in Theorem \ref{convisddp1} that ISDDP-LP solves the original MSLP when error terms vanish asymptotically.
Section \ref{sec:isddpnlpger} introduces and studies ISDDP-NLP. The class of problems to which ISDDP-NLP applies is given in Subsection \ref{sec:sddp1}.
A detailed description of ISDDP-NLP is given in Subsection \ref{sec:isddpalgo} and in Subsection \ref{sec:convsddp} the  convergence 
of the method is shown when errors vanish asymptotically.
Finally, in Section \ref{numexp}, we compare the computational bulk of SDDP and  ISDDP-LP
on four instances of a portfolio optimization problem with direct transaction costs.
On these instances, ISDDP-LP allows us to obtain a good policy faster than SDDP (compared to SDDP, 
with ISDDP-LP the CPU time decreases
by a factor of 6.2\%, 6.4\%, 6.5\%, and 11.1\% for the four instances considered).
It is also interesting to notice that once SDDP is implemented on a MSLP, the implementation of the corresponding ISDDP-LP
with given error terms is straightforward.
Therefore, if for a given application, or given classes of problems, we can find suitable choices of error terms
either using the rules from Remark \ref{remchoicepssto}, other rules, or "playing" with these parameters running ISDDP-LP
on instances, ISDDP-LP could allow us to solve similar new instances quicker than SDDP.

\if{
We use the following notation and terminology:
\par - The usual scalar product in $\mathbb{R}^n$ is denoted by $\langle x, y\rangle = x^T y$ for $x, y \in \mathbb{R}^n$.
The corresponding norm is $\|x\|=\|x\|_2=\sqrt{\langle x, x \rangle}$.
\par - $\mbox{ri}(A)$ is the relative interior of set $A$.
\par - $\mathbb{B}_n(x_0, r)=\{x \in \mathbb{R}^n : \|x - x_0\| \leq r\}$ for $x_0 \in \mathbb{R}^n, r \geq 0$.
\par - dom($f$) is the domain of function $f$.
\par - $\mbox{Diam}(X)=\max_{x, y \in X} \|x-y\|$ is the diameter of $X$.
\par - $\mathcal{N}_A(x)$ is the normal cone to $A$ at $x$.
\par - $X^\varepsilon := X + \varepsilon  \mathbb{B}_n(0, 1)$ is  the $\varepsilon$-fattening of the set $X \subset \mathbb{R}^n$.
\par - $\mathcal{C}(\mathcal{X})$ is the set of continuous real-valued functions on $\mathcal{X}$, equipped with the
norm $\|f\|_{\mathcal{X}} = \sup_{x \in \mathcal{X}}|f(x)|$. 
\par - $\mathcal{C}^1(\mathcal{X})$ is the set of real-valued continuously differentiable functions on $\mathcal{X}$. 
\par - span($X$) is the linear span of set of vectors $X$ and Aff($X$) is the affine span of $X$.
}\fi

\section{Computing inexact cuts for the value function of a convex optimization problem}\label{sec:computeinexactcuts}

\subsection{Inexact cuts for the value function of a linear program}\label{sec:computeinexactcutslp}

Let $X \subset \mathbb{R}^m$ and let $\mathcal{Q}:X\rightarrow {\overline{\mathbb{R}}}$ be the value function given by
\begin{equation} \label{vfunctionlp}
\mathcal{Q}(x)=\left\{
\begin{array}{l}
\inf_{y \in \mathbb{R}^n} \;c^T y\\
y \in Y(x):=\{y \in \mathbb{R}^n : Ay + B x = b, C y \leq f\},
\end{array}
\right.
\end{equation}
for matrices and vectors of appropriate sizes. We assume:\\
\par (H) for every $x \in X$, the set $Y(x)$ is nonempty and 
$y \rightarrow c^T y$ is bounded from below on $Y(x)$.\\
\par If Assumption (H) holds then $\mathcal{Q}$ is convex and finite on $X$ and by duality we can write
\begin{equation} \label{dualvfunctionlp}
\mathcal{Q}(x)=\left\{
\begin{array}{l}
\sup_{\lambda, \mu} \;\lambda^T (b-Bx) + \mu^Tf\\
A^T \lambda + C^T  \mu = c, \mu \leq 0,
\end{array}
\right.
\end{equation}
for $x \in X$. We will call an affine lower bounding function for $\mathcal{Q}$ on $X$ a cut for $\mathcal{Q}$ on $X$.
We say that cut $\mathcal{C}$ is inexact at $\bar x$ for convex function $\mathcal{Q}$
if the distance
$\mathcal{Q}( \bar x ) - \mathcal{C}( \bar x)$ between the 
values of $\mathcal{Q}$ and of the cut at $\bar x$ is strictly positive. When $\mathcal{Q}( \bar x )= \mathcal{C}( \bar x)$ we will say that cut 
$\mathcal{C}$ is exact at $\bar x$.

The following simple proposition will be used to derive ISDDP-LP: it provides an inexact cut for $\mathcal{Q}$ at $\bar x \in X$
on the basis of an approximate solution of \eqref{dualvfunctionlp}: 
\begin{proposition} \label{inexactlp}
Let Assumption (H) hold and let $\bar x \in X$.
Let $(\hat \lambda(\varepsilon), \hat \mu(\varepsilon))$ be an $\epsilon$-optimal feasible solution for dual problem \eqref{dualvfunctionlp}
written for $x= \bar x$, i.e., $A^T {\hat \lambda}(\varepsilon) + C^T  {\hat \mu}(\varepsilon) = c$, $\hat \mu(\varepsilon) \leq 0$, and
\begin{equation}\label{dualsoleps}
 {\hat \lambda}(\varepsilon)^T (b-B {\bar x}) + {\hat \mu}(\varepsilon)^T f \geq \mathcal{Q}( \bar x ) - \varepsilon,
\end{equation}
for some $\varepsilon \geq 0$. Then 
the affine function
$$
\mathcal{C}(x):= {\hat \lambda}(\varepsilon)^T (b-Bx) + {\hat \mu}(\varepsilon)^T f
$$
is a cut for $\mathcal{Q}$ at $\bar x$, i.e., for every $x \in X$ we have
$\mathcal{Q}(x) \geq \mathcal{C}(x)$ and the distance 
$\mathcal{Q}( \bar x ) - \mathcal{C}( \bar x)$ between the 
values of $\mathcal{Q}$ and of the cut at $\bar x$ is at most
$\varepsilon$. 
\end{proposition}
\begin{proof} $\mathcal{C}$ is indeed a cut for $\mathcal{Q}$ (an affine lower bounding function for $\mathcal{Q}$) because $(\hat \lambda(\varepsilon), \hat \mu(\varepsilon))$ is feasible for optimization problem \eqref{dualvfunctionlp}.
Relation \eqref{dualsoleps} gives the upper bound $\varepsilon$ for $\mathcal{Q}( \bar x ) - \mathcal{C}( \bar x)$.
\end{proof}

\subsection{Inexact cuts for the value function of a convex nonlinear program}\label{sec:computeinexactcutsnlp}

Let $\mathcal{Q}: X\rightarrow {\overline{\mathbb{R}}}$ be the value function given by
\begin{equation} \label{vfunctionq}
\mathcal{Q}(x)=\left\{
\begin{array}{l}
\inf_{y \in \mathbb{R}^{n}} \;f(y,x)\\
y \in S(x):=\{y\in Y \;:\;Ay+Bx=b,\;g(y,x)\leq 0\}.
\end{array}
\right.
\end{equation}
Here, $X \subseteq \mathbb{R}^m$ is nonempty, compact, and convex; $Y \subseteq \mathbb{R}^n$ is nonempty, closed, and convex;
and $A$ and $B$ are respectively $q \small{\times} n$ and $q \small{\times} m$ real matrices.
We will make the following assumptions which imply, in particular, the convexity of $\mathcal{Q}$ given
by \eqref{vfunctionq}:\\

\par (H1) $f:\mathbb{R}^n \small{\times} \mathbb{R}^m \rightarrow \mathbb{R} \cup \{+\infty\}$ is 
lower semicontinuous, proper, and convex.
\par (H2) For $i=1,\ldots,p$, the $i$-th component of function
$g(y, x)$ is a convex lower semicontinuous function
$g_i:\mathbb{R}^n \small{\times} \mathbb{R}^m \rightarrow \mathbb{R} \cup \{+\infty\}$.\\

As before, we say that $\mathcal{C}$ is a cut for $\mathcal{Q}$ on $X$ if $\mathcal{C}$ is an affine
function of $x$ such that $\mathcal{Q}(x) \geq \mathcal{C}(x)$ for all $x \in X$. We say that the cut
is exact at $\bar x \in X$ if $\mathcal{Q}(\bar x) = \mathcal{C}(\bar x)$.
Otherwise, the cut is said to be inexact at $\bar x$.

In this section,  our basic goal is, given ${\bar x} \in X$ and
$\varepsilon$-optimal primal and dual solutions of \eqref{vfunctionq} written for $x=\bar x$, to derive
an inexact cut $\mathcal{C}(x)$ for $\mathcal{Q}$ at $\bar x$, i.e., an affine lower bounding function
for $\mathcal{Q}$ such that the distance
$\mathcal{Q}( \bar x ) - \mathcal{C}( \bar x)$ between the 
values of $\mathcal{Q}$ and of the cut at $\bar x$ is bounded from above by a known function of the problem
parameters. Of course, when $\varepsilon=0$, we will check that 
$\mathcal{Q}( \bar x )= \mathcal{C}( \bar x)$.

For $x \in X$, let us introduce for problem \eqref{vfunctionq} the Lagrangian
function 
$$
L_{x}(y, \lambda, \mu)=f(y,x) + \lambda^T (Bx+Ay-b) + \mu^T g(y,x)
$$ and the function
$\ell :Y \small{\times} X  \small{\times}  \mathbb{R}^q      \small{\times} \mathbb{R}_{+}^p \rightarrow \mathbb{R}_{+}$ given by
\begin{equation}\label{defrxy}
\ell (\hat y , \bar x , \hat \lambda , \hat \mu ) = - \min_{y \in Y} \langle \nabla_y L_{\bar x} ( \hat y , \hat \lambda , \hat \mu ) , y - \hat y  \rangle = \max_{y \in Y} 
\langle  \nabla_y L_{\bar x} ( \hat y , \hat \lambda , \hat \mu ) , \hat y - y  \rangle,
\end{equation}
where, here and in what follows, scalar product $\langle \cdot , \cdot \rangle $ is given by $\langle x , y \rangle = x^T y$
and induces norm $\|\cdot\|:=\|\cdot\|_2$.
Next, dual function $\theta_x$ for problem \eqref{vfunctionq} can be written
$
\theta_{x}(\lambda, \mu)=\displaystyle \inf_{y \in Y} \;L_{x}(y, \lambda, \mu)
$
while the dual problem is 
\begin{equation}\label{dualpb}
\displaystyle \sup_{(\lambda, \mu) \in \mathbb{R}^q \small{\times} \mathbb{R}_{+}^{p} }\; \theta_{x}(\lambda, \mu).
\end{equation}
We make the following assumption which ensures no duality gap for \eqref{vfunctionq} for any $x \in X$:\\
\begin{itemize}
\item[(H3)] $\forall x \in X$ $\exists y_x \in \mbox{ri}(Y):\;$ $Bx+Ay_x=b$ and $g(y_x , x)< 0$.\\
\end{itemize}
The following proposition provides an inexact cut for $\mathcal{Q}$ given by \eqref{vfunctionq}:
\begin{proposition} \label{varprop1} Let $\bar x \in X$, let $\varepsilon \geq 0$,
let $\hat y(\epsilon)$ be an $\epsilon$-optimal feasible primal solution for problem \eqref{vfunctionq}
written for $x= \bar x$ 
and let $(\hat \lambda(\epsilon), \hat \mu(\epsilon))$ be an $\epsilon$-optimal feasible solution of the
corresponding dual problem, i.e., of problem
\eqref{dualpb} written for $x=\bar x$.
Let Assumptions (H1), (H2), and (H3) hold. 
Assume that $Y$ is nonempty, closed, and convex, that $f(\cdot,x)$ is finite on $S(x)$ for all $x \in X$, and that 
$\eta(\varepsilon)=\ell (\hat y(\epsilon) , \bar x  , \hat \lambda(\epsilon) , \hat \mu(\epsilon) )$
is finite.
If additionally $f$ and $g$ are differentiable on $Y \small{\times} X$ then the affine function
\begin{equation}\label{cutvarprop1}
\mathcal{C}(x):= L_{\bar x} ( \hat y(\epsilon), {\hat \lambda}(\epsilon), \hat \mu(\epsilon) )- \eta(\varepsilon)  + 
\langle \nabla_x L_{\bar x} ( \hat y(\epsilon), {\hat \lambda}(\epsilon), \hat \mu(\epsilon) ) , x - \bar x \rangle
\end{equation}
is a cut for $\mathcal{Q}$ at $\bar x$ and the distance 
$\mathcal{Q}( \bar x ) - \mathcal{C}( \bar x)$ between the 
values of $\mathcal{Q}$ and of the cut at $\bar x$ is at most
$\varepsilon + \ell (\hat y(\epsilon),  \bar x  , \hat \lambda(\epsilon) , \hat \mu(\epsilon) )$.
\end{proposition}
\begin{proof}
To simplify notation, we use 
$\hat y, \hat \lambda, \hat \mu$, for respectively $\hat y(\epsilon), \hat \lambda(\epsilon), \hat \mu(\epsilon)$.
Consider primal problem \eqref{vfunctionq} written for $x=\bar x$.
Due to Assumption (H3) and the fact that $f(\cdot,\bar x)$ is bounded from below on $S(\bar x)$,
the optimal value $\mathcal{Q}(\bar x)$ of this problem
is the optimal value of the corresponding dual problem, i.e., of problem
\eqref{dualpb} written for $x=\bar x$. 
Using the fact that $\hat y$ and $(\hat \lambda, \hat \mu)$ are respectively $\varepsilon$-optimal primal
and dual solutions it follows that
\begin{equation}\label{optprimaldual}
f( \hat y , \bar x) \leq \mathcal{Q}(  \bar x  ) + \varepsilon \mbox{ and }\theta_{\bar x}( \hat \lambda , \hat \mu  ) \geq \mathcal{Q}(  \bar x  ) - \varepsilon. 
\end{equation}
Moreover, since the approximate primal and dual solutions are feasible, we have that
\begin{equation}\label{feasibility}
 \hat y \in Y,\, B {\bar x} + A {\hat y} = b,\,g(\hat y, \bar x) \leq 0, \,\hat \mu \geq 0.
\end{equation}
Using Relation \eqref{optprimaldual}, the definition of dual function $\theta_{{\bar x}}$, and the fact that $\hat y \in Y$, we get
\begin{equation}\label{optdual2}
L_{\bar x} (\hat y, {\hat \lambda}, \hat \mu ) \geq  \theta_{\bar x}( \hat \lambda , \hat \mu  ) \geq \mathcal{Q}( \bar x ) - \varepsilon.  
\end{equation}
Due to Assumptions (H1) and (H2), for any $\lambda$ and $\mu \geq 0$ the function $L_{\cdot}(\cdot,\lambda,\mu)$ which associates 
the value $L_{x}(y, \lambda, \mu)$ to $(x,y)$ is convex. 
Since $\hat \mu \geq 0$, it follows that for every $x \in X, y \in Y$, we have that 
\begin{align*}
L_x (y, {\hat \lambda}, \hat \mu )   \geq  
L_{\bar x} ( \hat y, {\hat \lambda}, \hat \mu ) +\langle \nabla_x L_{\bar x} ( \hat y, {\hat \lambda}, \hat \mu ) , x - \bar x \rangle + \langle \nabla_y L_{\bar x} ( \hat y, {\hat \lambda}, \hat \mu ) , y - \hat y \rangle.
\end{align*}
Since $(\hat \lambda, \hat \mu)$ is feasible for dual problem \eqref{dualpb}, the Weak Duality Theorem gives
$\mathcal{Q}( x  ) \geq  \theta_{x}( \hat \lambda , \hat \mu  )  = \inf_{y \in Y} L_{x}(y, \hat \lambda , \hat \mu )$ 
for every $x \in X$
and minimizing over $y \in Y$ on each side of the above inequality  we obtain
\begin{align*}
\mathcal{Q}( x  )   \geq  
L_{\bar x} ( \hat y, {\hat \lambda}, \hat \mu )-  \ell (\hat y,  \bar x  , \hat \lambda , \hat \mu ) + 
\langle \nabla_x L_{\bar x} ( \hat y, {\hat \lambda}, \hat \mu ) , x - \bar x \rangle.
\end{align*}
Finally, using relation \eqref{optdual2}, we get
\begin{align*}
\mathcal{Q}( \bar x ) - \mathcal{C}( \bar x ) =
\mathcal{Q}( \bar x ) - L_{\bar x} ( \hat y, {\hat \lambda}, \hat \mu )+  \ell (\hat y,  \bar x  , \hat \lambda , \hat \mu ) \leq \varepsilon + \ell (\hat y , \bar x  , \hat \lambda , \hat \mu ).
\end{align*}
\end{proof}
We now refine the bound $\varepsilon + \ell (\hat y(\epsilon),  \bar x  , \hat \lambda(\epsilon) , \hat \mu(\epsilon) )$ on $\mathcal{Q}( \bar x ) - \mathcal{C}( \bar x)$ given by Proposition \ref{varprop1}
making the following assumptions:  
\begin{itemize}
\item[(H4)] $f$ is differentiable on $Y \small{\times} X$ and there exists $M_1>0$ such that for every $x \in X, y_1, y_2 \in Y$, we have
$$
\|\nabla_y f(y_2,x) -  \nabla_y f(y_1, x)  \|  \leq  M_1 \|y_2   - y_1\|. 
$$
\item[(H5)] $g$ is differentiable on $Y \small{\times} X$ and
there exists $M_2>0$ such that for every $i=1,\ldots,p, x \in X, y_1, y_2 \in Y$, we have
$$
\|\nabla_y g_i(y_2,x) -  \nabla_y g_i(y_1, x)  \|  \leq  M_2 \|y_2   - y_1\|. 
$$
\end{itemize}
In what follows we denote the diameter of set $Y$ by $D(Y)$.
\begin{proposition} \label{varprop2}
Assume that $Y$ is nonempty, convex, and compact.
Let $\bar x \in X$, let $\varepsilon \geq 0$,
let $\hat y( \epsilon)$ be an $\epsilon$-optimal feasible primal solution for problem \eqref{vfunctionq}
written for $x= \bar x$ 
and let $(\hat \lambda( \epsilon), \hat \mu( \epsilon))$ be an $\epsilon$-optimal feasible solution of the
corresponding dual problem, i.e., of problem
\eqref{dualpb} written for $x=\bar x$.
Also let $\mathcal{L}_{\bar x}$ be any lower bound on $\mathcal{Q}( \bar x )$.
Let Assumptions (H1), (H2), (H3), (H4), and (H5) hold. 
Then $\mathcal{C}(x)$ given by \eqref{cutvarprop1} is a cut for $\mathcal{Q}$
at $\bar x$  and setting
$
M_3=M_1 + \mathcal{U}_{\bar x} M_2
$
with
$$
\mathcal{U}_{\bar x}= 
\frac{f(y_{\bar x} , \bar x  )  - \mathcal{L}_{\bar x} + \varepsilon }{\min (-g_{i}(y_{\bar x} ,  \bar x), i=1,\ldots,p )},
$$
the distance 
$\mathcal{Q}( \bar x ) - \mathcal{C}( \bar x)$ between the 
values of $\mathcal{Q}$ and of the cut at $\bar x$ is at most
$$
\begin{array}{ll}
\varepsilon + \ell (\hat y( \epsilon),  \bar x  , \hat \lambda( \epsilon) , \hat \mu( \epsilon) ) - \frac{\ell (\hat y( \epsilon),  \bar x  , \hat \lambda( \epsilon) , \hat \mu( \epsilon) )^2}{2 M_3 \emph{D}(Y)^2}&\mbox{if }\ell (\hat y( \epsilon),  \bar x  , \hat \lambda( \epsilon) , \hat \mu( \epsilon) ) \leq M_3 \emph{D}(Y)^2,\\
\varepsilon + \frac{1}{2} \ell (\hat y( \epsilon),  \bar x  , \hat \lambda( \epsilon) , \hat \mu( \epsilon) )& \mbox{otherwise.}
\end{array}
$$
\end{proposition}
\begin{proof}
As before we use the short notation
$\hat y, \hat \lambda, \hat \mu$, for respectively $\hat y(\epsilon), \hat \lambda(\epsilon), \hat \mu(\epsilon)$.
We already know from Proposition \ref{varprop1} that $\mathcal{C}$ is a cut for $\mathcal{Q}$. Let us now prove
the upper bound for $\mathcal{Q}( \bar x ) - \mathcal{C}( \bar x)$ given in the proposition.
We compute 
$$
\nabla_y L_{\bar x}(y, \lambda, \mu) =\nabla_y f(y, \bar x) + A^T \lambda + \sum_{i=1}^p \mu_i \nabla_y g_i(y, \bar x).
$$
Therefore for every $y_1, y_2 \in Y$, using Assumptions (H4) and (H5), we have
\begin{equation} \label{boundgradLag}
\|\nabla_y L_{\bar x}(y_2, {\hat \lambda}, {\hat \mu})  -   \nabla_y L_{\bar x}(y_1, {\hat \lambda}, {\hat \mu}) \| \leq (M_1 +  \|{\hat \mu}\|_1 M_2) \|y_2 - y_1\|.  
\end{equation}
Next observe that 
$$
\begin{array}{lll}
\mathcal{L}_{\bar x} - \varepsilon \leq 
\mathcal{Q}( \bar x )-\varepsilon  \leq \theta_{\bar x}(\hat \lambda, \hat \mu)  & \leq & f(y_{\bar x}, {\bar x}) + {\hat \lambda}^T ( A y_{\bar x} + B {\bar x} - b ) +  {\hat \mu}^T g( y_{\bar x}, \bar x ) \\
& \leq & f(y_{\bar x}, {\bar x}) + \|{\hat \mu}\|_1  \max_{i=1,\ldots,p} g_{i}(y_{\bar x} ,  \bar x).
\end{array}
$$
From the above relation, we get $\| {\hat \mu}\|_1 \leq \mathcal{U}_{\bar x}$, which, plugged into \eqref{boundgradLag}, gives
\begin{equation}\label{boundlag}
\|\nabla_y L_{\bar x}(y_2, {\hat \lambda}, {\hat \mu})  -   \nabla_y L_{\bar x}(y_1, {\hat \lambda}, {\hat \mu}) \| \leq M_3 \|y_2 - y_1\|.
\end{equation}
Now let $y_* \in Y$ such that 
$
\ell (\hat y , \bar x , \hat \lambda , \hat \mu ) = 
\langle  \nabla_y L_{\bar x} ( \hat y , \hat \lambda , \hat \mu ) , \hat y - y_*  \rangle.
$
Using relation \eqref{boundlag}, for every $0 \leq t \leq 1$, we get
$$
L_{\bar x}({\hat y} + t (y_* - {\hat y}), {\hat \lambda}, {\hat \mu}) \leq L_{\bar x}({\hat y}, {\hat \lambda}, {\hat \mu})
+ t \langle \nabla_y L_{\bar x}({\hat y}, {\hat \lambda}, {\hat \mu}), y_* - {\hat y} \rangle + \frac{1}{2} M_3 t^2 \|y_* - {\hat y}\|^2.
$$
Since ${\hat y} + t (y_* - {\hat y}) \in Y$, using the above relation and the definition of $\theta_{\bar x}$, we obtain
$$
\mathcal{Q}( \bar x ) - \varepsilon \leq \theta_{\bar x}(\hat \lambda , \hat \mu ) \leq L_{\bar x}({\hat y}, {\hat \lambda}, {\hat \mu}) -t \ell (\hat y,  \bar x  , \hat \lambda , \hat \mu ) + \frac{1}{2} M_3 t^2 \|y_* - {\hat y}\|^2.
$$
Therefore $\mathcal{Q}( \bar x ) - \mathcal{C}( \bar x ) = \mathcal{Q}( \bar x ) - L_{\bar x}({\hat y}, {\hat \lambda}, {\hat \mu}) + \ell (\hat y,  \bar x  , \hat \lambda , \hat \mu )$
is bounded from above by
$$
\varepsilon +  \ell (\hat y,  \bar x  , \hat \lambda , \hat \mu ) + \min_{0 \leq t \leq 1} \Big(-t \ell (\hat y,  \bar x  , \hat \lambda , \hat \mu ) + \frac{1}{2} M_3 t^2 \mbox{D}(Y)^2 \Big)
$$
and we easily conclude computing $\min_{0 \leq t \leq 1} \Big( -t \ell (\hat y,  \bar x  , \hat \lambda , \hat \mu ) + \frac{1}{2} M_3 t^2 \mbox{D}(Y)^2 \Big)$.
\end{proof}
\begin{remark}\label{remextincutvar} It is possible to extend Proposition \ref{varprop2} when optimization problem 
$\max_{y \in Y}  \langle  \nabla_y L_{\bar x} ( \hat y , \hat \lambda , \hat \mu ) , \hat y - y  \rangle$
with optimal value
$\ell (\hat y , \bar x , \hat \lambda , \hat \mu )$
is  solved approximately.
\end{remark}

\section{Bounding the norm of $\varepsilon$-optimal solutions to the dual of a convex optimization problem} \label{sec:boundingmulti} 

Consider the following convex optimization problem:
\begin{equation} \label{defpbbounddual}
f_* =
\left\{
\begin{array}{l}
\min f(y)\\
Ay= b, \,g(y) \leq 0,\;y \in Y 
\end{array}
\right.
\end{equation}
where 
\begin{itemize}
\item[(i)] $Y \subset \mathbb{R}^n$ is a closed convex set and $A$ is a $q \small{\times}  n$ matrix;
\item[(ii)] $f$ is convex Lipschitz continuous with Lipschitz constant $L( f )$ on $Y$;
\item[(iii)] all components of $g$ are convex Lipschitz continuous functions
with Lipschitz constant $L( g )$ on $Y$;
\item[(iv)] $f$ is bounded from below on the feasible set.
\end{itemize}
We assume the following Slater type constraint qualification:
\begin{equation}\label{slaterboundmulti}
\mbox{SL: There exist }\kappa>0 \mbox{ and }y_0 \in \mbox{ri}( Y ) \mbox{ such that }g( y_0 ) \leq -\kappa {\textbf{e}} \mbox{ and }A y_0 = b   
\end{equation}
where {\textbf{e}} is a vector of ones in $\mathbb{R}^p$. \\

Since SL holds, the optimal value $f_*$ of \eqref{defpbbounddual} can be written as the optimal value of the dual problem:
\begin{equation}\label{dualfirst}
f_* = \displaystyle \max_{\mu \geq 0, \lambda} \left\{ \theta(\lambda, \mu ) := \displaystyle \min_{y \in Y} \{ f(y) + \langle \lambda , Ay -b \rangle + \langle \mu , g(y) \rangle  \} \right\}.
\end{equation}

Consider the vector space $F=A\mbox{Aff}(Y) - b$ where Aff($Y$) is the affine span of $Y$.
Clearly for any $y\in Y$ and every $\lambda \in F^{\perp}$ we have $\lambda^T ( A y - b) = 0$ and therefore
for every $\lambda \in \mathbb{R}^q$, $\theta(\lambda , \mu )=\theta(\Pi_{F}( \lambda ) , \mu )$
where $\Pi_{F}( \lambda )$ is the orthogonal projection of $\lambda$ onto $F$.

It follows that if $F^{\perp} \neq \{ 0 \}$, the set of $\epsilon$-optimal dual solutions of
dual problem \eqref{dualfirst} is not bounded because from any $\epsilon$-optimal dual solution
$(\lambda(\varepsilon), \mu(\varepsilon))$ we can build an $\epsilon$-optimal dual solution
$(\lambda(\varepsilon)+\lambda, \mu(\varepsilon))$ with the same value of the dual function of norm
arbitrarily large taking $\lambda$ in $F^{\perp}$ with norm sufficiently large. 

However, the optimal value of the dual (and primal) problem can be written equivalently as
\begin{equation}\label{dualreform}
f_* = \displaystyle \max_{\lambda, \mu} \left\{  \theta(\lambda , \mu )  : \mu \geq 0, \lambda = A y-b, y \in \mbox{Aff}(Y) \right\}. 
\end{equation}

In this section, our goal is to derive bounds on the norm of $\epsilon$-optimal solutions to the dual of \eqref{defpbbounddual}
written in the form \eqref{dualreform}.

In what follows, we denote the $\|\cdot\|_2$-ball of radius $r$ and center $y_0$ in $\mathbb{R}^n$ by
$\mathbb{B}_n( y_0 , r)$.
From Assumption SL, we deduce that there is $r>0$ such that $\mathbb{B}_n( y_0 , r)  \cap   \mbox{Aff}(Y) \subseteq Y$
and that there is some ball $\mathbb{B}_q( 0, \rho_*)$ of positive radius $\rho_*$ such that
the intersection of this ball and of the set $A \mbox{Aff}(Y) - b$ is contained in the set 
$A\Big( \mathbb{B}_n( y_0 , r)  \cap   \mbox{Aff}(Y)  \Big)-b$. To define such $\rho_*$, let
$\rho: A \mbox{Aff}(Y)-b \rightarrow \mathbb{R}_{+}$ given by
$$
\rho(z)=\max \left\{t \|z\| \;:\;t\geq 0, tz \in A( \mathbb{B}_n( y_0 , r)  \cap   \mbox{Aff}(Y) )-b\right\}.
$$
Since $y_0 \in Y$, we can write $\mbox{Aff}(Y)=y_0 + V_Y$ where $V_Y$ is the vector space
$V_Y = \{x-y,\;x,y \in \mbox{Aff}(Y)\}$. Therefore
$$
A( \mathbb{B}_n( y_0 , r)  \cap   \mbox{Aff}(Y) )-b= A ( \mathbb{B}_n (0, r ) \cap V_Y )
$$
and $\rho$ can be reformulated as
\begin{equation}\label{definitionrho}
\rho(z)=\max \left\{t \|z\| \;:\;t\geq 0, tz \in A ( \mathbb{B}_n (0, r ) \cap V_Y )\right\}.
\end{equation}
Note that $\rho$ is well defined and finite valued (we have $ 0\leq \rho(z) \leq \|A\|r$).
Also, clearly $\rho( 0 ) = 0$ and
$\rho(z)=\rho(\lambda z)$ for every $\lambda > 0$ and $z \neq 0$. 
Therefore if $A=0$ then  $\rho_*$ can be any positive real, for instance $\rho_*=1$,
and if $A \neq 0$ 
we define 
\begin{equation}\label{defrhostar}
\rho_* = \min \{\rho(z)\;:\;z\neq 0, z \in A \mbox{Aff}(Y)-b\} =  \min \{\rho(z)\;:\;\|z\|=1, z \in A V_Y \},
\end{equation}
which is well defined and positive since $\rho(z)>0$ for every $z$ such that  $\|z\|=1, z \in A \mbox{Aff}(Y)-b$
(indeed if $z \in A \mbox{Aff}(Y)-b$ with $\|z\|=1$ then $z = A y-b$ for some $y \in\mbox{Aff}(Y), y \neq y_0$,
and since 
$$
\frac{r}{\|y-y_0\|} z = A\Big( y_0  + r \frac{y- y_0}{\|y-y_0\|} \Big)- b \in A\Big( \mathbb{B}_n( y_0 , r)  \cap   \mbox{Aff}(Y)  \Big)-b,
$$
we have $\rho(z) \geq \frac{r}{\|y-y_0\|}\|z\|=\frac{r}{\|y-y_0\|}>0$). We now claim that parameter $\rho_*$ we have just defined satisfies
our requirement namely
\begin{equation}\label{deductionfromSL}
\mathbb{B}_q( 0, \rho_*) \cap \Big(  A  \mbox{Aff}(Y) - b \Big)  \subseteq A\Big( \mathbb{B}_n( y_0 , r)  \cap   \mbox{Aff}(Y)  \Big)-b.
\end{equation}
This can be rewritten as 
\begin{equation}\label{deductionfromSLrewr}
\mathbb{B}_q( 0, \rho_*) \cap A  V_Y  \subseteq A\Big( \mathbb{B}_n( 0 , r)  \cap  V_Y  \Big).
\end{equation}
Indeed, let $z \in \mathbb{B}_q( 0, \rho_*) \cap \Big(  A \mbox{Aff}(Y) - b \Big)$.
If $A=0$ or $z=0$ then $ z \in A\Big( \mathbb{B}_n( y_0 , r)  \cap   \mbox{Aff}(Y)  \Big)-b$.
Otherwise, by definition of $\rho$, we have $\rho(z ) \geq \rho_* \geq \|z\|$.
Let ${\bar t} \geq 0$ be such that ${\bar t} z \in A( \mathbb{B}_n( y_0 , r)  \cap   \mbox{Aff}(Y) )-b$ and
$\rho(z)={\bar t}\|z\|$. The relations $({\bar t}-1)\|z\| \geq 0$ and $z \neq 0$ imply $\bar t \geq 1$.
By definition of $\bar t$, we can write ${\bar t} z = Ay -b$ where $y \in \mathbb{B}_n( y_0 , r)  \cap   \mbox{Aff}(Y)$.
It follows that $z$ can be written
$$
z=A\Big(y_0 + \frac{y-y_0}{{\bar t}}  \Big) -b = A {\bar y} - b
$$
where $\displaystyle  \bar y = y_0 + \frac{y-y_0}{{\bar t}}  \in \mbox{Aff}(Y)$ and 
$\|\bar y - y_0\|= \displaystyle \frac{\|y-y_0\|}{\bar t} \leq \|y - y_0\| \leq r$ (because $\bar t \geq 1$ and $y \in \mathbb{B}_n( y_0 , r)$).
This  means that $z \in A\Big( \mathbb{B}_n( y_0 , r)  \cap   \mbox{Aff}(Y)  \Big)-b$, which proves inclusion \eqref{deductionfromSL}.

We are now in a position to state the main result of this section:
\begin{proposition}\label{propboundnormepsdualsol} Consider optimization problem \eqref{defpbbounddual} with optimal value $f_*$.
Let Assumptions (i)-(iv) and SL hold and let $(\lambda(\varepsilon), \mu(\varepsilon))$ be an $\varepsilon$-optimal solution
to the dual problem \eqref{dualreform} with optimal value $f_*$.
Let 
\begin{equation}\label{defrkappaLg}
0<r \leq \frac{\kappa}{2 L( g )}, 
\end{equation}
be such that the intersection of the ball $\mathbb{B}_n(y_0, r)$ 
and of Aff($Y$) is contained in $Y$ (this $r$ exists because $y_0 \in \mbox{ri}(Y)$).
If $A=0$ let $\rho_*=1$. Otherwise, let $\rho_*$ given by
\eqref{defrhostar} with  $\rho$ as in  \eqref{definitionrho}.
Let $\mathcal{L}$ be any lower bound on the optimal value
$f_*$ of \eqref{defpbbounddual}. Then we have
$$
\|(\lambda(\varepsilon), \mu(\varepsilon)) \| \leq \frac{f( y_0 ) - \mathcal{L} + \varepsilon + L( f ) r }{\min( \rho_* , \kappa/2)}.
$$
\end{proposition}
\begin{proof} By definition of $(\lambda(\varepsilon), \mu(\varepsilon))$ and  $\mathcal{L}$, and using SL, we have
\begin{equation}\label{firstrelationboundepssol}
\mathcal{L} - \varepsilon \leq f_* - \varepsilon \leq \theta( \lambda(\varepsilon), \mu(\varepsilon)) . 
\end{equation}
Now define $z(\varepsilon)=0$ if $\lambda(\varepsilon)=0$ and
$z( \varepsilon )=-\frac{\rho_*}{\| \lambda( \varepsilon ) \| }  \lambda( \varepsilon )$ otherwise.
Observing that $z( \varepsilon ) \in \mathbb{B}_q( 0, \rho_*) \cap \Big(  A  \mbox{Aff}(Y) - b \Big)$
and using relation \eqref{deductionfromSL} we deduce that
$z( \varepsilon ) \in A\Big( \mathbb{B}_n( y_0 , r)  \cap   \mbox{Aff}(Y)  \Big)-b \subseteq AY -b$.
Therefore, we can write $z( \varepsilon ) = A {\bar y}-b$ for some
${\bar y} \in \mathbb{B}_n(y_0 , r) \cap \mbox{Aff}(Y) \subseteq Y$. Next, using the definition of $\theta$, we get 
$$
\begin{array}{lll}
\theta( \lambda(  \varepsilon ) , \mu ( \varepsilon  )  )   & \leq & 
f( \bar y ) + \lambda( \varepsilon )^T ( A {\bar y} - b ) + \mu( \varepsilon )^T g( \bar y ) \mbox{ since }{\bar y} \in Y,\\
 &\leq & f(y_0 ) + L( f ) r + z( \varepsilon )^T \lambda( \varepsilon ) + \mu( \varepsilon )^T g( y_0 ) + L(g) r  \| \mu( \varepsilon ) \|_1,\\
& \leq & f(y_ 0) + L( f ) r - \rho_*  \| \lambda( \varepsilon ) \| - \frac{\kappa}{2} \| \mu( \varepsilon ) \|_1  \mbox{ using SL and }\eqref{defrkappaLg},
\end{array}
$$
where for the second inequality we have used (ii),(iii), and $\|\bar y - y_0 \| \leq r$.
We obtain for $\|(\lambda(  \varepsilon ) , \mu ( \varepsilon  ))\| = \sqrt{ \|\lambda(  \varepsilon )\|^2  + \|\mu(  \varepsilon )\|^2  } $
the upper bound
\begin{equation}\label{rewritingfinaldualnorm}
\begin{array}{lll}
 \|\lambda(  \varepsilon )\| + \|\mu(  \varepsilon )\|
\leq  \|\lambda(  \varepsilon )\| + \|\mu(  \varepsilon )\|_1 \leq \frac{f(y_0 ) + L( f ) r - \theta(\lambda(  \varepsilon ) , \mu ( \varepsilon  ))}{\min( \rho_* , \kappa/2)}. 
\end{array}
\end{equation}
Combining \eqref{firstrelationboundepssol} with upper bound \eqref{rewritingfinaldualnorm} on $\|(\lambda(  \varepsilon ) , \mu ( \varepsilon  ))\|$, we obtain the desired bound. 
\end{proof}

We also have the following immediate corollary of Proposition \ref{propboundnormepsdualsol}:

\begin{corollary}\label{corpropboundnormepsdualsol} Under the assumptions of Proposition \ref{propboundnormepsdualsol}, let $\bar f$
be an upper bound on $f$ on the feasibility set of \eqref{defpbbounddual} and assume that
$\bar f$ is convex and Lipschitz continuous on $\mathbb{R}^n$ with Lipschitz constant $L( \bar f )$. Then we have for 
$\|(\lambda(\varepsilon), \mu(\varepsilon)) \|$ the bound 
$
\|(\lambda(\varepsilon), \mu(\varepsilon)) \| \leq \frac{ {\bar f} ( y_0 ) - \mathcal{L} + \varepsilon + L( \bar f ) r }{\min( \rho_* , \kappa/2)}.
$
\end{corollary}

\section{Inexact cuts in SDDP applied to multistage stochastic linear programs}\label{sec:insddpl}

\subsection{Problem formulation, assumptions, and algorithm}\label{sec:insddpl1}

We are interested in solution methods for linear Stochastic Dynamic Programming equations:
the first stage problem is 
\begin{equation}\label{firststodplp}
\mathcal{Q}_1( x_0 ) = \left\{
\begin{array}{l}
\min_{x_1 \in \mathbb{R}^n} c_1^T x_1 + \mathcal{Q}_2 ( x_1 )\\
A_{1} x_{1} + B_{1} x_{0} = b_1,
x_1 \geq 0
\end{array}
\right.
\end{equation}
for $x_0$ given and for $t=2,\ldots,T$, $\mathcal{Q}_t( x_{t-1} )= \mathbb{E}_{\xi_t}[ \mathfrak{Q}_t ( x_{t-1}, \xi_{t}  )  ]$ with
\begin{equation}\label{secondstodplp} 
\mathfrak{Q}_t ( x_{t-1}, \xi_{t}  ) = 
\left\{ 
\begin{array}{l}
\min_{x_t \in \mathbb{R}^n} c_t^T x_t + \mathcal{Q}_{t+1} ( x_t )\\
A_{t} x_{t} + B_{t} x_{t-1} = b_t,
x_t \geq 0,
\end{array}
\right.
\end{equation}
with the convention that $\mathcal{Q}_{T+1}$ is null and
where for $t=2,\ldots,T$, random vector $\xi_t$ corresponds to the concatenation of the elements in random matrices $A_t, B_t$ which have a known
finite number of rows and random vectors $b_t, c_t$.
Moreover, it is assumed that $\xi_1$ is not random. For convenience, we will denote 
$$
X_t(x_{t-1}, \xi_t):=\{x_t \in \mathbb{R}^n : A_{t} x_{t} + B_{t} x_{t-1} = b_t, \,x_t \geq 0 \}.
$$
We make the following assumptions:
\begin{itemize}
\item[(A0)] $(\xi_t)$ 
is interstage independent and
for $t=2,\ldots,T$, $\xi_t$ is a random vector taking values in $\mathbb{R}^K$ with a discrete distribution and
a finite support $\Theta_t=\{\xi_{t 1}, \ldots, \xi_{t M}\}$ while $\xi_1$ is deterministic, with 
vector $\xi_{t j}$ being the concatenation
of the elements in $A_{t j}, B_{t j}, b_{t j}, c_{t j}$.\footnote{To simplify notation and without loss of generality, we have assumed that the 
number of realizations $M$ of $\xi_t$, the size $K$ of $\xi_t$ and $n$ of $x_t$ do not depend on $t$.} 
\item[(A1-L)] The set $X_1(x_{0}, \xi_1 )$ is nonempty and bounded and for every $x_1 \in X_1(x_{0}, \xi_1 )$,
for every $t=2,\ldots,T$, for every realization $\tilde \xi_2, \ldots, \tilde\xi_t$ of $\xi_2,\ldots,\xi_t$,
for every $x_{\tau} \in X_{\tau}( x_{\tau-1} , \tilde \xi_{\tau}), \tau=2,\ldots,t-1$, the set $X_t( x_{t-1} , {\tilde \xi}_t )$
is nonempty and bounded.
\end{itemize}

We put $\Theta_1 = \{\xi_1\}$ and for $t\geq 2$
we set
$p_{t i}= \mathbb{P}(\xi_t = \xi_{t i}) >0, i=1,\ldots,M$.

ISDDP-LP applied to linear Stochastic Dynamic Programming equations \eqref{firststodplp}, \eqref{secondstodplp}
is a simple extension of SDDP where the  subproblems of the forward and backward passes
are solved approximately. At iteration $k$, for $t=2,\ldots,T$,
 function  $\mathcal{Q}_t$ is approximated by a  
 piecewise affine lower bounding function $\mathcal{Q}_t^k$ 
which is a maximum of affine lower bounding functions $\mathcal{C}_{t}^i$ called inexact cuts:
$$
\mathcal{Q}_t^k(x_{t-1} ) = \max_{0 \leq i \leq k} \mathcal{C}_{t}^i( x_{t-1}  ) \mbox{ with }\mathcal{C}_{t}^i (x_{t-1})=\theta_{t}^i + \langle \beta_{t}^i , x_{t-1} \rangle
$$
where coefficients $\theta_{t}^i, \beta_{t}^i$ are computed as explained below. The steps of ISDDP-LP are as follows.\\

\par {\textbf{ISDDP-LP, Step 1: Initialization.}} For $t=2,\ldots,T$, take for $\mathcal{C}_t^0 = \mathcal{Q}_t^0$ 
a known lower bounding affine function for $\mathcal{Q}_t$. Set the iteration count $k$ to 1 and $\mathcal{Q}_{T+1}^0 \equiv 0$.\\
\par {\textbf{ISDDP-LP, Step 2: Forward pass.}} We generate sample
${\tilde \xi}^k = (\tilde \xi_1^k, \tilde \xi_2^k,\ldots,\tilde \xi_T^k)$ from the distribution of $\xi^k \sim (\xi_1,\xi_2,\ldots,\xi_T)$,
with the convention that $\tilde \xi_1^k = \xi_1$. Using approximation $\mathcal{Q}_{t+1}^{k-1}$
of $\mathcal{Q}_{t+1}$  (computed at previous iterations), we compute a $\delta_t^k$-optimal feasible solution $x_t^k$ of the problem
\begin{equation}\label{pbforwardpasslp}
\left\{
\begin{array}{l}
\min_{x_t \in \mathbb{R}^n} x_t^T {\tilde c}_t^k + \mathcal{Q}_{t+1}^{k-1} ( x_t )\\
x_t \in X_t(x_{t-1}^k, {\tilde \xi}_{t}^k )
\end{array}
\right.
\end{equation}
for $t=1,\ldots,T$,
where $x_0^k=x_0$ and where $\tilde c_t^k$ is the realization of $c_t$ in $\tilde \xi_t^k$. For $k \geq 1$ and $t=1,\ldots,T$, define the function ${\underline{\mathfrak{Q}}}_t^k : \mathbb{R}^n {\small{\times}} \Theta_t \rightarrow  {\overline{\mathbb{R}}}$ by
\begin{equation}\label{backwardt0lp}
{\underline{\mathfrak{Q}}}_t^k (x_{t-1} , \xi_t   ) =  
\left\{
\begin{array}{l}
\min_{x_t \in \mathbb{R}^n} c_t^T x_t + \mathcal{Q}_{t+1}^k ( x_t )\\
x_t \in X_t(x_{t-1}, \xi_{t} ).
\end{array}
\right.
\end{equation}
With this notation, we have 
\begin{equation}\label{forwdefQlp}
{\underline{\mathfrak{Q}}}_t^{k-1} (x_{t-1}^k , {\tilde \xi}_t^k   )
\leq  \langle {\tilde c}_t^k ,  x_t^k \rangle + \mathcal{Q}_{t+1}^{k-1}( x_t^k ) \leq {\underline{\mathfrak{Q}}}_t^{k-1} (x_{t-1}^k , {\tilde \xi}_t^k   ) + \delta_t^k.
\end{equation}

\par {\textbf{ISDDP-LP, Step 3: Backward pass.}} 
The backward pass builds inexact cuts for $\mathcal{Q}_t$ at $x_{t-1}^k$ computed in the forward pass.
For $t=T+1$,  we have $\mathcal{Q}_{t}^k = \mathcal{Q}_{T+1}^k \equiv 0$, i.e., $\theta_{T+1}^k$ and $\beta_{T+1}^k$ are null. 
For $j=1,\ldots,M$, we solve approximately the problem
\begin{equation}\label{backwardTlp} 
\left\{ 
\begin{array}{l}
\displaystyle \min_{x_T \in \mathbb{R}^n} c_{T j}^T x_T \\
A_{T j} x_{T} + B_{T j} x_{T-1}^k = b_{T j},
x_T \geq 0,
\end{array}
\right.
\mbox{ with dual }
\left\{ 
\begin{array}{l}
\max_{\lambda} \lambda^T ( b_{T j} - B_{T j} x_{T-1}^k )\\
A_{T j}^T \lambda \leq c_{T j},
\end{array}
\right.
\end{equation}
and optimal value $\mathfrak{Q}_T ( x_{T-1}^k, \xi_{T j}  )$.
More precisely, let $\lambda_{T j}^k$ be an $\varepsilon_T^k$-optimal basic feasible solution of the dual problem above
(it is in particular an extreme point of the feasible set).
Therefore
$A_{T j}^T \lambda_{T j}^k \leq c_{T j}$ and 
\begin{equation}\label{epssolbackTlp}
\mathfrak{Q}_T ( x_{T-1}^k, \xi_{T j}  ) - \varepsilon_T^k \leq   \langle \lambda_{T j}^k , b_{T j} - B_{T j} x_{T-1}^k \rangle  \leq \mathfrak{Q}_T ( x_{T-1}^k, \xi_{T j}  ). 
\end{equation}
We 
compute 
\begin{equation}\label{thetaTkbetaTk}
\theta_{T}^k= \sum_{j=1}^{M} p_{T j} \langle   b_{T j}, \lambda_{T j}^k \rangle \mbox{ and }\beta_{T}^k = -\sum_{j=1}^{M} p_{T j} B_{T j}^T \lambda_{T j}^k.
\end{equation}
Using Proposition \ref{inexactlp} we have that $\mathcal{C}_{T}^k (x_{T-1})=\theta_{T}^k + \langle \beta_{T}^k , x_{T-1} \rangle$ is an inexact cut for
$\mathcal{Q}_T$ at $x_{T-1}^k$. Using \eqref{epssolbackTlp}, we also see that 
\begin{equation}\label{qualitycutT}
\mathcal{Q}_T( x_{T-1}^k ) - \mathcal{C}_T^k ( x_{T-1}^k ) \leq \varepsilon_T^k. 
\end{equation}

Then for $t=T-1$ down to $t=2$, knowing $\mathcal{Q}_{t+1}^k \leq \mathcal{Q}_{t+1}$,
for $j=1,\ldots,M$, consider the optimization problem
\begin{equation}\label{backwardtlp}
{\underline{\mathfrak{Q}}}_t^k ( x_{t-1}^k, \xi_{t j}  ) = 
\left\{ 
\begin{array}{l}
\displaystyle \min_{x_t} c_{t j}^T x_t + \mathcal{Q}_{t+1}^k ( x_t ) \\
x_t \in X_t( x_{t-1}^k , \xi_{t j} )
\end{array}
\right.
=
\left\{ 
\begin{array}{l}
\displaystyle \min_{x_t, f} c_{t j}^T x_t + f \\
A_{t j} x_{t} + B_{t j} x_{t-1}^k = b_{t j}, x_t \geq 0,\\
f \geq \theta_{t+1}^i + \langle \beta_{t+1}^i , x_t  \rangle, i=1,\ldots,k,
\end{array}
\right.
\end{equation}
with optimal value ${\underline{\mathfrak{Q}}}_t^k ( x_{t-1}^k, \xi_{t j}  )$.
Observe that due to (A1-L) the above problem is feasible and has a finite optimal value. Therefore ${\underline{\mathfrak{Q}}}_t^k ( x_{t-1}^k, \xi_{t j}  )$
can be expressed as the optimal value of the corresponding dual problem:
\begin{equation}\label{dualpbtbacklp}
{\underline{\mathfrak{Q}}}_t^k ( x_{t-1}^k, \xi_{t j}  ) = 
\left\{
\begin{array}{l}
\displaystyle \max_{\lambda, \mu} \lambda^T( b_{t j} - B_{t j} x_{t-1}^k   ) + \sum_{i=1}^k \mu_{i} \theta_{t+1}^i  \\
A_{t j}^T \lambda +\displaystyle  \sum_{i=1}^k \mu_{i} \beta_{t+1}^i \leq c_{t j},\;\sum_{i=1}^k \mu_{i}=1,\\
\mu_{i} \geq 0,\,i=1,\ldots,k.
\end{array}
\right.
\end{equation}
Let $(\lambda_{t j}^k, \mu_{t j}^k )$ be an  $\varepsilon_t^k$-optimal basic feasible solution of dual problem \eqref{dualpbtbacklp}
(it is in particular an extreme point of the feasible set) and let ${\underline{\mathcal{Q}}}_t^k$ be the function given by ${\underline{\mathcal{Q}}}_t^k ( x_{t-1} ) = \sum_{j=1}^{M} p_{t j} {\underline{\mathfrak{Q}}}_t^k( x_{t-1} , \xi_{t j} )$.
We compute 
\begin{equation}\label{formulathetabetatk}
\theta_{t}^k =\sum_{j=1}^{M} p_{t j}\Big(  \langle  \lambda_{t j}^k ,  b_{t j} \rangle +  \langle    \mu_{t j}^k , \theta_{t+1, k} \rangle \Big) \mbox{ and }
\beta_{t}^k =- \sum_{j=1}^{M} p_{t j}B_{t j}^T \lambda_{t j}^k,
\end{equation}
where $i$-th component $\theta_{t+1, k}(i)$  of vector $\theta_{t+1, k}$  is $\theta_{t+1}^i$ for $i=1,\ldots,k$.
 Setting $\mathcal{C}_t^k ( x_{t-1} ) = \theta_t^k + \langle \beta_t^k , x_{t-1} \rangle$ and using Proposition \ref{inexactlp}, we have 
\begin{equation}\label{cutt}
{\underline{\mathcal{Q}}}_t^k ( x_{t-1} ) \geq \mathcal{C}_t^k ( x_{t-1} ) \;\;\mbox{ and }\;\;{\underline{\mathcal{Q}}}_t^k ( x_{t-1}^k ) - \mathcal{C}_t^k  ( x_{t-1}^k )    \leq \varepsilon_t^k.
\end{equation}
Using the fact that $\mathcal{Q}_{t+1}^k( x_{t-1} )  \leq \mathcal{Q}_{t+1}( x_{t-1} )$, we have 
${\underline{\mathfrak{Q}}}_t^k(x_{t-1}, \xi_{t j}) \leq \mathfrak{Q}_t(x_{t-1}, \xi_{t j})$,
${\underline{\mathcal{Q}}}_t^k(x_{t-1}) \leq \mathcal{Q}_t(x_{t-1})$,
and therefore 
\begin{equation}
\mathcal{Q}_t ( x_{t-1} ) \geq \mathcal{C}_t^k ( x_{t-1} ) 
\end{equation}
which shows that $\mathcal{C}_t^k$ is an inexact cut for $\mathcal{Q}_t$.\\
\par {\textbf{ISDDP-LP, Step 4:} Do $k \leftarrow k+1$ and go to Step 2.\\

Following the proof of Lemma 1 in \cite{philpot}, we obtain that
for all $t=2,\ldots,T+1$,
the collection of distinct values $(\theta_t^k, \beta_t^k)_k$
is finite and therefore cut coefficients $(\theta_t^k, \beta_t^k)_k$
are uniformly bounded. Observe that this proof uses the fact that 
$(\lambda_{t j}^k , \mu_{t j}^k)$ are extreme points of the feasible set of 
\eqref{dualpbtbacklp}. There could however be unbounded sequences 
of approximate optimal feasible solutions to \eqref{dualpbtbacklp}.

\if{
\par The previous algorithm can be modified solving in the backward pass some of the subproblems only. 
This gives rise to Inexact DOASA whose steps are given below in the case when the randomness only appears
in the right-hand side of the constraints, namely when $\xi_t=b_t$.\\

\par Steps 1, 2, and 4 of IDOASA are identical to Steps 1, 2, and 4 of ISDDP-LP.\\
\par  {\textbf{Step 4 of IDOASA is as follows}}: set $\mathcal{D}_T^k = \emptyset$,
take a sample $\Theta_T^k \subset \Theta_T$
and for each $b_{T j} \in \Theta_t^k$ compute an $\varepsilon_T^k$-optimal solution $\lambda_{T j}^k$
of dual problem \eqref{backwardT} and add $\lambda_{T j}^k$ to set $\mathcal{D}_T^k$. For each $b_{T j} \notin \Theta_t^k$, compute 
$$
\lambda_{T j}^k = \mbox{argmax} \{ \lambda^T(b_{T j} - B_{T j} x_{T-1}^k ) \;  : \; \lambda \in \mathcal{D}_T^k\}.
$$
Compute $\theta_T^k, \beta_{T}^k$ by \eqref{thetaTkbetaTk}.
For $t=T-1$ down to $t=2$, set $\mathcal{D}_t^k = \emptyset$, take a sample $\Theta_t^k \subset \Theta_t$
and for each $b_{t j} \in \Theta_t^k$ compute an $\varepsilon_t^k$-optimal solution $(\lambda_{t j}^k, \mu_{t j}^k)$
of dual problem \eqref{dualpbtback} and add $(\lambda_{t j}^k, \mu_{t j}^k )$ to set $\mathcal{D}_t^k$. For each $b_{t j} \notin \Theta_t^k$, compute 
$$
(\lambda_{t j}^k , \mu_{t j}^k )  = \mbox{argmax} \{\lambda^T( b_{t j} - B_{t j} x_{t-1}^k   ) + \sum_{i=1}^k \mu_{i} \theta_{t+1}^i  \;  : \; (\lambda, \mu) \in \mathcal{D}_t^k \}.
$$
Compute $\theta_t^k, \beta_{t}^k$ by \eqref{formulathetabetatk}.\\

\par To put it in a nutshell, IDOASA is a variant of DOASA where all subproblems solved in the backward and 
forward passes are solved approximately, namely $\varepsilon_t^k$-optimal solutions are computed for the subproblems
of stage $t$ and iteration $k$.
Observe that since $\xi_t=b_t$, we have $A_{t j}=A_t$, $c_{t j}=c_T$ and therefore
for fixed $k$ dual problems \eqref{backwardT} (when $b_{T j}$ varies in $\Theta_T^k$)
have the same feasible sets. Similarly, for fixed $k$ dual problems \eqref{dualpbtback} 
(when $b_{t j}$ varies in $\Theta_t^k$) have the same feasible sets. This justifies that 
$\mathcal{C}_t^k$ computed by IDOASA is a valid cut (an affine lower bounding function) for $\mathcal{Q}_t$.
}\fi

\subsection{Convergence analysis}\label{conv-analsddp}

In this section we state a convergence result for ISDDP-LP in Theorem \ref{convisddplp} when errors $\delta_t^k, \varepsilon_t^k$
are bounded and in Theorem \ref{convisddp1} when these errors vanish asymptotically.

We will need the following simple extension of \cite[Lemma A.1]{lecphilgirar12}:
\begin{lemma}\label{limsuptechlemma} Let $X$ be a compact set, let 
$f: X \rightarrow \mathbb{R}$ be Lipschitz continuous, and suppose that the sequence of $L$-Lipschitz continuous functions 
$f^k, k \in \mathbb{N}$ satisfies $f^{k}(x) \leq f^{k+1}(x) \leq f(x) \;\mbox{for all }x \in X,\;k \in \mathbb{N}$.
Let $(x^k)_{k \in \mathbb{N}}$ be a sequence in $X$ and  assume that 
\begin{equation}\label{limsuphyp}
\varlimsup_{k \rightarrow +\infty} f( x^k ) -f^k( x^k ) \leq S 
\end{equation}
for some $S \geq 0$.
Then
\begin{equation}\label{limsuphypkm1}
\varlimsup_{k \rightarrow +\infty} f( x^k ) -f^{k-1}( x^k ) \leq S. 
\end{equation}
\end{lemma}
\begin{proof} Let us show \eqref{limsuphypkm1} by contradiction.
Assume that  \eqref{limsuphypkm1} does not hold. Then there exist $\varepsilon_0>0$ and $\sigma: \mathbb{N} \rightarrow \mathbb{N}$
increasing 
such that for every $k \in \mathbb{N}$ we have
\begin{equation}\label{firstlemma}
f( x^{\sigma(k)} ) -f^{\sigma(k)-1}( x^{\sigma(k)} ) > S + \varepsilon_0.
\end{equation}
Since $x^{\sigma(k)}$ is a sequence of the compact set $X$, it has some convergent subsequence which converges to some $x_* \in X$.
Taking into account \eqref{limsuphyp} and the fact that $f^k$ are $L$-Lipschitz continuous, we can take $\sigma$
such that \eqref{firstlemma} holds and
\begin{eqnarray}
f( x^{\sigma(k)} ) - f^{\sigma(k)}( x^{\sigma(k)} ) &\leq& S + \frac{\varepsilon_0}{4},\label{seclemma}\\
f^{\sigma(k)-1}( x^{\sigma(k)} ) - f^{\sigma(k)-1}( x_* )& >& -\frac{\varepsilon_0}{4},\label{thlemma}\\
f^{\sigma(k)}( x_* ) - f^{\sigma(k)}( x^{\sigma(k)} ) &>&  -\frac{\varepsilon_0}{4}.\label{fourlemma}
\end{eqnarray}
Therefore for every $k \geq 1$ we get
$$
\begin{array}{lll}
f^{\sigma(k)}( x_* ) - f^{\sigma(k-1)}( x_* ) & \geq & f^{\sigma(k)}( x_* ) - f^{\sigma(k)-1}( x_* )\mbox{ since }\sigma(k) \geq \sigma(k-1) + 1,\\
& = & f^{\sigma(k)}( x_* ) - f^{\sigma(k)}( x^{\sigma(k)} )\;\;(>-\varepsilon_0/4\mbox{ by }\eqref{fourlemma}),\\
 & &  + f^{\sigma(k)}( x^{\sigma(k)} ) -   f( x^{\sigma(k)} )\;\;(\geq-S-\varepsilon_0/4\mbox{ by }\eqref{seclemma}),\\
 && + f( x^{\sigma(k)} )  -  f^{\sigma(k)-1}( x^{\sigma(k)} )\;\;(>S+\varepsilon_0\mbox{ by }\eqref{firstlemma}),\\
 &  &+  f^{\sigma(k)-1}( x^{\sigma(k)} )    - f^{\sigma(k)-1}( x_* )\;\;(>-\varepsilon_0/4\mbox{ by }\eqref{thlemma}),\\
 & > & \varepsilon_0/4,
\end{array}
$$
which implies $f^{\sigma(k)}( x_* ) > f^{\sigma(0)}( x_* ) + k \frac{\varepsilon_0}{4}$. This is in contradiction with the fact that the sequence $f^{\sigma(k)}( x_* )$ is bounded from above by $f(x_*)$.
\end{proof}

We will assume that the sampling procedure in ISDDP-LP satisfies the following property:\\

\par (A2) The samples in the backward passes are independent: $(\tilde \xi_2^k, \ldots, \tilde \xi_T^k)$ is a realization of
$\xi^k=(\xi_2^k, \ldots, \xi_T^k) \sim (\xi_2, \ldots,\xi_T)$ 
and $\xi^1, \xi^2,\ldots,$ are independent.\\

Before stating our first convergence theorem, we need more notation. Due to Assumption (A0), the realizations of $(\xi_t)_{t=1}^T$ form a scenario tree of depth $T+1$
where the root node $n_0$ associated to a stage $0$ (with decision $x_0$ taken at that
node) has one child node $n_1$
associated to the first stage (with $\xi_1$ deterministic).
We denote by $\mathcal{N}$ the set of nodes and for a node $n$ of the tree, we define: 
\begin{itemize}
\item $C(n)$: the set of children nodes (the empty set for the leaves);
\item $x_n$: a decision taken at that node;
\item $p_n$: the transition probability from the parent node of $n$ to $n$;
\item $\xi_n$: the realization of process $(\xi_t)$ at node $n$\footnote{The same notation $\xi_{\tt{Index}}$ is used to denote
the realization of the process at node {\tt{Index}} of the scenario tree and the value of the process $(\xi_t)$
for stage {\tt{Index}}. The context will allow us to know which concept is being referred to.
In particular, letters $n$ and $m$ will only be used to refer to nodes while $t$ will be used to refer to stages.}:
for a node $n$ of stage $t$, this realization $\xi_n$ contains in particular the realizations
$c_n$ of $c_t$, $b_n$ of $b_t$, $A_{n}$ of $A_{t}$, and $B_{n}$ of $B_{t}$.
\end{itemize}

Next, we define for iteration $k$ decisions $x_n^k$ for all node $n$ of the scenario tree
simulating the policy obtained in the end of iteration $k-1$ replacing 
cost-to-go function $\mathcal{Q}_t$ by 
$\mathcal{Q}_{t}^{k-1}$ for $t=2,\ldots,T+1$:\\
\rule{\linewidth}{1pt}
{\textbf{Simulation of the policy in the end of iteration $k-1$.}}\\
\hspace*{0.4cm}{\textbf{For }}$t=1,\ldots,T$,\\
\hspace*{0.8cm}{\textbf{For }}every node $n$ of stage $t-1$,\\
\hspace*{1.6cm}{\textbf{For }}every child node $m$ of node $n$, compute a $\delta_t^k$-optimal solution $x_m^k$ of
\begin{equation} \label{defxtkj}
{\underline{\mathfrak{Q}}}_t^{k-1}( x_n^k , \xi_m ) = \left\{
\begin{array}{l}
\displaystyle \inf_{x_m} \; c_m^T x_m  + \mathcal{Q}_{t+1}^{k-1}( x_m ) \\
x_m \in X_t( x_n^k, \xi_m ),
\end{array}
\right.
\end{equation}
\hspace*{2.4cm}where $x_{n_0}^k = x_0$.\\
\hspace*{1.6cm}{\textbf{End For}}\\
\hspace*{0.8cm}{\textbf{End For}}\\
\hspace*{0.5cm}{\textbf{End For}}\\
\rule{\linewidth}{1pt}\\ 
We are now in a position to state our first convergence theorem for ISDDP-LP:
\begin{theorem}[Convergence of ISDDP-LP with bounded errors] \label{convisddplp}
Consider the sequences of decisions $(x_n^k)_{n \in \mathcal{N}}$ and of functions $(\mathcal{Q}_t^k)$ generated by ISDDP-LP.
Assume that (A0), (A1-L), and (A2) hold, and that 
errors $\varepsilon_t^k$ and $\delta_t^k$ are bounded: $0 \leq \varepsilon_t^k \leq {\bar \varepsilon}$,
$0 \leq \delta_t^k \leq {\bar \delta}$ for finite ${\bar \delta}, {\bar \varepsilon}$. Then the following holds:
\begin{itemize}
\item[(i)] for $t=2,\ldots,T+1$, for all node $n$ of stage $t-1$,  almost surely
\begin{equation}\label{lowerbounds}
0 \leq \varliminf_{k \rightarrow +\infty} \mathcal{Q}_t( x_{n}^k ) - \mathcal{Q}_t^k( x_{n}^k ) \leq  
\varlimsup_{k \rightarrow +\infty} \mathcal{Q}_t( x_{n}^k ) - \mathcal{Q}_t^k( x_{n}^k ) \leq ({\bar \delta}  +  {\bar{\varepsilon}})(T-t+1);
\end{equation}
\item[(ii)] for every $t=2,\ldots,T$, for all node $n$ of stage $t-1$,
the limit superior and limit inferior of the sequence of upper bounds  $\Big( \displaystyle  \sum_{m \in C(n)} p_m (  c_m^T x_m^k   + \mathcal{Q}_{t+1}( x_m^k ))  \Big)_k$ satisfy almost surely
\begin{equation}\label{uppbounds}
\begin{array}{l}
0 \leq 
\varliminf_{k \rightarrow +\infty} \displaystyle \sum_{m \in C(n)} p_m \Big[ c_m^T x_m^k   + \mathcal{Q}_{t+1}( x_m^k ) \Big] - \mathcal{Q}_t( x_n^k ),  \\
\varlimsup_{k \rightarrow +\infty} \displaystyle  \sum_{m \in C(n)} p_m \Big[ c_m^T x_m^k   + \mathcal{Q}_{t+1}( x_m^k ) \Big] - \mathcal{Q}_t( x_n^k ) \leq ({\bar \delta} + {\bar{\varepsilon}})(T-t+1);
\end{array}
\end{equation}
\item[(iii)]
the limit superior and limit inferior of the sequence   ${\underline{\mathfrak{Q}}}_1^{k-1}( x_{0} , \xi_1 )$ of lower bounds on the optimal value 
$\mathcal{Q}_1( x_0 )$ of \eqref{firststodplp} satisfy almost surely
\begin{equation}\label{lbounds}
\mathcal{Q}_1( x_{0} )- {\bar{\delta}} T   - {\bar{\varepsilon}} (T-1) \leq \varliminf_{k \rightarrow +\infty} {\underline{\mathfrak{Q}}}_1^{k-1}( x_{0} , \xi_1 ) \leq  
\varlimsup_{k \rightarrow +\infty} {\underline{\mathfrak{Q}}}_1^{k-1}( x_{0} , \xi_1 ) \leq \mathcal{Q}_1( x_{0} ).
\end{equation}
\end{itemize}
\end{theorem}
\begin{proof} The proof is provided in the appendix.
\end{proof}

Theorem \ref{convisddp1} below shows the convergence of ISDDP-LP in a finite number of iterations when errors $\varepsilon_t^k, \delta_t^k$ vanish asymptotically.

\begin{theorem} [Convergence of ISDDP-LP with asymptotically vanishing errors] \label{convisddp1}
Consider the sequences of decisions $(x_n^k)_{n \in \mathcal{N}}$ and of functions $(\mathcal{Q}_t^k)$ generated by ISDDP-LP.
Let Assumptions (A0), (A1-L), and (A2) hold.
If for all $t=1,\ldots,T$,
$\displaystyle \lim_{k \rightarrow +\infty} \delta_t^k = 0 $ and
for all $t=1,\ldots,T-1$, $\lim_{k \rightarrow +\infty} \varepsilon_t^k = 0$,
then ISDDP-LP converges with probability one in a finite number of iterations to an optimal solution to
\eqref{firststodplp}, \eqref{secondstodplp}.
\end{theorem}
\begin{proof}
Due to Assumptions (A0), (A1-L), ISDDP-LP 
generates almost surely a finite number of 
trial points $x_1^k, x_2^k,\ldots, x_T^k$.
Similarly, almost surely only a finite number of different functions $\mathcal{Q}_t^k, t=2,\ldots,T,$ can be generated.
Therefore, after some iteration $k_1$, every optimization subproblem solved in the forward and backward passes is a copy of 
an optimization problem solved previously. It follows that after some iteration $k_0$ all subproblems are solved exactly (optimal solutions
are computed for all subproblems) and functions $\mathcal{Q}_t^k$ do not change any more.
Consequently, from iteration $k_0$ on, we can apply the arguments of the proof of convergence of (exact) SDDP applied to linear programs (see Theorem 5 in \cite{philpot}).
\end{proof}

\begin{remark}\label{remchoicepssto} [Choice of parameters $\delta_t^k$ and $\varepsilon_t^k$]
Recalling our convergence analysis and what motivates inexact variants of SDDP, it makes sense to choose for 
$\delta_t^k$ and $\varepsilon_t^k$ sequences which decrease with $k$ and which, for fixed $k$,
decrease with $t$. A simple rule consists in defining relative errors, as long as a solver
handling such errors is used to solve the problems of the forward and backward passes.
Let the relative error for stage $
t$ and iteration $k$ be ${\tt{Rel}}\_{\tt{Err}}_t^k$. 
We propose to use the relative error
\begin{equation}\label{rel_errors2}
{\tt{Rel}}\_{\tt{Err}}_t^k = \frac{1}{k}\Big[ \overline{\varepsilon} - \left(\frac{\overline{\varepsilon} - \varepsilon_0}{T-2} \right) (t-2) \Big],
\end{equation}
for stage $t \geq 2$ and iteration $k \geq 1$ (in both the forward and backward passes) for some 
small $\varepsilon_0$, $0<\varepsilon_0 < \overline{\varepsilon}$,
and ${\tt{Rel}}\_{\tt{Err}}_1^k=0$, which induces corresponding $\delta_t^k$ and $\varepsilon_t^k$.
The relative error ${\tt{Rel}}\_{\tt{Err}}_1^k$ at the first stage needs to be null to
define a valid lower bound at each iteration, see also Remark \ref{stoppingisddp}.
However, it seems more difficult to define sound absolute errors. One possible sequence of absolute error terms in the backward pass could be
$
\varepsilon_t^k = \max\Big(1, \left|{\underline{\mathfrak{Q}}}_t^{k-1} (x_{t-1}^k ,  \tilde \xi_t^k ) \right|   \Big) {\tt{Rel}}\_{\tt{Err}}_t^k
$
with ${\tt{Rel}}\_{\tt{Err}}_t^k$ still given by \eqref{rel_errors2}.
\end{remark}

\section{Inexact cuts in SDDP applied to a class of nonlinear multistage stochastic programs}\label{sec:isddpnlpger}

In this section we introduce ISDDP-NLP, an inexact variant of SDDP 
which combines the tools developed in Sections \ref{sec:computeinexactcuts} and \ref{sec:boundingmulti} with SDDP.

\subsection{Problem formulation and assumptions}\label{sec:sddp1}

ISDDP-NLP applies to the class of multistage stochastic nonlinear optimization problems introduced in \cite{guiguessiopt2016}
of form
\begin{equation}\label{pbtosolve}
\begin{array}{l} 
\displaystyle{\inf_{x_1,\ldots,x_T}} \; \mathbb{E}_{\xi_2,\ldots,\xi_T}[ \displaystyle{\sum_{t=1}^{T}}\;f_t(x_t(\xi_1,\xi_2,\ldots,\xi_t ), x_{t-1}(\xi_1,\xi_2,\ldots,\xi_{t-1}), \xi_t )]\\
x_t(\xi_1,\xi_2,\ldots,\xi_t) \in X_t( x_{t-1}(\xi_1,\xi_2,\ldots,\xi_{t-1}), \xi_t )\;\mbox{a.s.}, \;x_{t} \;\mathcal{F}_t\mbox{-measurable, }t \leq T,
\end{array}
\end{equation}
where $x_0$ is given,  $(\xi_t)_{t=2}^T$ is a stochastic process, $\mathcal{F}_t$ is the sigma-algebra
$\mathcal{F}_t:=\sigma(\xi_j, j\leq t)$, and where $X_t( x_{t-1} , \xi_t)$ is now given by
$$
X_t( x_{t-1} , \xi_t)= \{x_t \in \mathbb{R}^n : x_t \in \mathcal{X}_t,\;g_t(x_t, x_{t-1}, \xi_t) \leq 0,\;\;\displaystyle A_{t} x_{t} + B_{t} x_{t-1} = b_t\},
$$
with $\xi_t$ containing in particular the random elements in matrices $A_t, B_t$, and vector $b_t$.

For this problem, we can write Dynamic Programming equations: assuming that $\xi_1$ is deterministic,
the first stage problem is 
\begin{equation}\label{firststodp}
\mathcal{Q}_1( x_0 ) = \left\{
\begin{array}{l}
\inf_{x_1 \in \mathbb{R}^n} F_1(x_1, x_0, \xi_1) := f_1(x_1, x_0, \xi_1)  + \mathcal{Q}_2 ( x_1 )\\
x_1 \in X_1( x_{0}, \xi_1 )\\
\end{array}
\right.
\end{equation}
for $x_0$ given and for $t=2,\ldots,T$, $\mathcal{Q}_t( x_{t-1} )= \mathbb{E}_{\xi_t}[ \mathfrak{Q}_t ( x_{t-1},  \xi_{t}  )  ]$ with
\begin{equation}\label{secondstodp} 
\mathfrak{Q}_t ( x_{t-1}, \xi_{t}  ) = 
\left\{ 
\begin{array}{l}
\inf_{x_t \in \mathbb{R}^n}  F_t(x_t, x_{t-1}, \xi_t ) :=  f_t ( x_t , x_{t-1}, \xi_t ) + \mathcal{Q}_{t+1} ( x_t )\\
x_t \in X_t ( x_{t-1}, \xi_t ),
\end{array}
\right.
\end{equation}
with the convention that $\mathcal{Q}_{T+1}$ is null.

We make assumption (A0) on $(\xi_t)$ (see Section \ref{sec:insddpl1}) and will denote by $A_{t j}, B_{t j},$ and $b_{t j}$ the realizations of respectively $A_t, B_t,$ and $b_t$
in $\xi_{t j}$. 

We set $\mathcal{X}_0=\{x_0\}$ and make the following assumptions (A1-NL) on the problem data: there exists $\varepsilon_t>0$ (without loss of generality, we will assume in the sequel that $\varepsilon_t=\varepsilon$) 
such that for $t=1,\ldots,T$,\\
\par (A1-NL)-(a) $\mathcal{X}_t$ is nonempty, convex, and compact.
\par (A1-NL)-(b) For every $j=1,\ldots,M$, the function
$f_t(\cdot, \cdot,\xi_{t j})$ is convex on  $\mathcal{X}_t \small{\times} \mathcal{X}_{t-1}$
and belongs to  $\mathcal{C}^{1}(\mathcal{X}_t \small{\times} \mathcal{X}_{t-1})$, the set of real-valued continuously differentiable functions on $\mathcal{X}_t \small{\times} \mathcal{X}_{t-1}$.
\par (A1-NL)-(c) For every $j=1,\ldots,M$, each component $g_{t i}(\cdot, \cdot, \xi_{t j}), i=1,\ldots,p$, of function $g_t(\cdot, \cdot, \xi_{t j})$ is 
convex on $\mathcal{X}_t \small{\times} \mathcal{X}_{t-1}^{\varepsilon_t}$
and belongs to $\mathcal{C}^{1}( \mathcal{X}_t \small{\times} \mathcal{X}_{t-1} )$
where $\mathcal{X}_{t-1}^{\varepsilon_t}=\mathcal{X}_{t-1}+\varepsilon_t  \{x \in \mathbb{R}^n : \|x\|_2 \leq 1\}$.
\par (A1-NL)-(d) For every $j=1,\ldots,M$, for every
$x_{t-1} \in \mathcal{X}_{t-1}^{\varepsilon_t}$,
the set $X_t(x_{t-1}, \xi_{t j}) \cap \mbox{ri}( \mathcal{X}_t)$ is nonempty.
\par  (A1-NL)-(e) If $t \geq 2$, for every $j=1,\ldots,M$, there exists
${\bar x}_{t j}=({\bar x}_{t j t}, {\bar x}_{t j t-1}  ) \in  \mbox{ri}(\mathcal{X}_t) \small{\times} \mathcal{X}_{t-1}$
such that $g_t(\bar x_{t j t}, \bar x_{t j t-1}, \xi_{t j}) < 0$ and  
$A_{t j} \bar x_{t j t} + B_{t j} \bar x_{t j t-1} = b_{t j}$.\\

Assumptions (A0) and (A1-NL) ensure that functions $\mathcal{Q}_t$ are convex and Lipschitz continuous on $\mathcal{X}_{t-1}$: 
\begin{lemma} \label{convrecfuncQtS} Let Assumptions (A0) and (A1-NL) hold. Then for $t=2,\ldots,T+1$, function $\mathcal{Q}_t$ is convex and Lipschitz continuous on 
$\mathcal{X}_{t-1}$.
\end{lemma}
\begin{proof} See the proof of Proposition 3.1 in \cite{guiguessiopt2016}.
\end{proof}

Assumption (A1-NL)-(d) is used to bound the cut coefficients (see Proposition \ref{propboundeddualsto}). Differentiability and Assumption (A1-NL)-(e) are  useful to derive inexact cuts.

As for MSLPs from Section \ref{sec:insddpl}, due to Assumption (A0), the $M^{T-1}$ realizations of $(\xi_t)_{t=1}^T$ form a scenario tree of depth $T+1$
and we define parameters $n_0, n_1, \mathcal{N}$, $C(n), x_n, p_n, \xi_n$ 
which have the same meaning as in Section \ref{sec:insddpl}.
Additionally, we denote by {\tt{Nodes}}$(t)$ the set of nodes for stage $t$ and 
for a node $n$ of the tree, we define vector $\xi_{[n]}$, the history of the realizations of process $(\xi_t)$ from the first stage node $n_1$ to node $n$.
More precisely, for a node $n$ of stage $t$, the $i$-th component of $\xi_{[n]}$ is $\xi_{\mathcal{P}^{t-i}(n)}$ for $i=1,\ldots, t$,
 where $\mathcal{P}:\mathcal{N} \rightarrow \mathcal{N}$ is the function 
 associating to a node its parent node (the empty set for the root node).

\subsection{ISDDP-NLP algorithm} \label{sec:isddpalgo}

Similarly to SDDP, to solve \eqref{pbtosolve}, ISDDP-NLP approximates for each $t=2,\ldots,T+1$, function
$\mathcal{Q}_t$ by a polyhedral lower approximation $\mathcal{Q}_t^k$ at iteration $k$.
To describe ISDDP-NLP, it is convenient to introduce for $t=1,\ldots,T$, and $k \geq 0$ functions 
$F_t^k(x_t,x_{t-1},\xi_t )=f_t(x_t,x_{t-1},\xi_t)+\mathcal{Q}_{t+1}^k(x_t )$ and ${\underline{\mathfrak{Q}}}_t^k( x_{t-1} , \xi_t ): \mathcal{X}_{t-1} \small{\times} \Theta_t \rightarrow \mathbb{R}$ 
given by 
$$
{\underline{\mathfrak{Q}}}_t^k( x_{t-1} , \xi_t )= \left\{
\begin{array}{l}
\displaystyle \inf_{x_t} \;F_t^k(x_t, x_{t-1} , \xi_t)\\ 
x_t \in X_t( x_{t-1} , \xi_t).
\end{array}
\right.
$$
We start the first iteration with known lower approximations $\mathcal{Q}_t^0= \mathcal{C}_t^0$ for $\mathcal{Q}_t, t=2,\ldots,T$.
Iteration $k \geq 1$ starts with a forward pass
which computes trial points $x_n^k$
for all nodes $n$ of the scenario tree replacing recourse functions 
$\mathcal{Q}_{t+1}$ by approximations $\mathcal{Q}_{t+1}^{k-1}$ available at the beginning of this iteration:
\par {\textbf{Forward pass:}}
\par {\textbf{For }}$t=1,\ldots,T$,
\par \hspace*{0.8cm}{\textbf{For }}every node $n$ of stage $t-1$,
\par\hspace*{1.6cm}{\textbf{For }}every child node $m$ of node $n$, compute a $\delta_t^k$-optimal solution $x_m^k$ of
\begin{equation} \label{defxtkjnlp}
{\underline{\mathfrak{Q}}}_t^{k-1}( x_n^k , \xi_m ) = \left\{
\begin{array}{l}
\displaystyle \inf_{x_m} \; F_t^{k-1}(x_m , x_n^k, \xi_m):= f_t( x_m , x_n^k , \xi_m ) + \mathcal{Q}_{t+1}^{k-1}( x_m ) \\
x_m \in X_t( x_n^k, \xi_m ),
\end{array}
\right.
\end{equation}
\par \hspace*{2.4cm}where $x_{n_0}^k = x_0$ and $\mathcal{Q}_{T+1}^{k-1} = \mathcal{Q}_{T+1} \equiv 0$.
\par \hspace*{1.6cm}{\textbf{End For}}
\par \hspace*{0.8cm}{\textbf{End For}}
\par {\textbf{End For}}
\par Therefore trial points satisfy 
\begin{equation}\label{epssolforward}
{\underline{\mathfrak{Q}}}_t^{k-1}( x_n^k , \xi_m )   \leq F_t^{k-1}(x_m^k , x_n^k, \xi_m)  \leq {\underline{\mathfrak{Q}}}_t^{k-1}( x_n^k , \xi_m ) + \delta_t^k. 
\end{equation}
The forward pass is followed by a backward pass which selects a set of nodes~$n_t^k$, $t=1,\ldots,T$ 
(with $n_1^k=n_1$, and for $t \geq 2$, $n_t^k$ a node of stage $t$, child of node $n_{t-1}^k$) 
corresponding to a sample $({\tilde \xi}_1^k, {\tilde \xi}_2^k,\ldots, {\tilde \xi}_T^k)$
of $(\xi_1, \xi_2,\ldots, \xi_T)$. For $t=2,\ldots,T$, an inexact cut  
\begin{equation}\label{eqcutctknlp}
\mathcal{C}_t^k( x_{t-1} ) = \theta_t^{k} - \eta_t^k ( \varepsilon_t^k ) +  \langle \beta_t^{k}, x_{t-1}-x_{n_{t-1}^{k}}^{k} \rangle
\end{equation}
is computed for $\mathcal{Q}_t$ at $x_{n_{t-1}^{k}}^{k}$ for some coefficients $\theta_t^{k}, \eta_t^k ( \varepsilon_t^k ), \beta_t^k$ whose computations are detailed below. 
At the end of iteration $k$, we obtain the polyhedral lower approximations $\mathcal{Q}_{t}^{k}$ of $\mathcal{Q}_t,\;t=2,\ldots,T+1$, 
given by
$
\mathcal{Q}_{t}^{k}(x_{t-1})  = \displaystyle \max_{0 \leq \ell \leq k}\;\mathcal{C}_t^{\ell}( x_{t-1} ).
$ 
Cuts are computed backward, starting from $t=T+1$, down to $t=2$. For $t=T+1$, the cut is exact:
$\mathcal{C}_{T+1}^k, \theta_{T+1}^k, \eta_{T+1}^k,$ and $\beta_{T+1}^k$ are null. For stage $t<T+1$,
we compute for every child node $m$ of $n:=n_{t-1}^k$ 
an $\varepsilon_t^k$-optimal solution $x_{m}^{B k}$ of
\begin{equation}\label{primalpbisddpnlp}
{\underline{\mathfrak{Q}}}_t^k(x_n^k , \xi_m ) = \left\{
\begin{array}{l}
\displaystyle \inf_{x_m} \;F_t^k(x_m , x_n^k ,  \xi_m):=f_t( x_m , x_n^k , \xi_m) + \mathcal{Q}_{t+1}^k ( x_m )\\ 
x_m \in X_t(x_n^k , \xi_m   )
\end{array}
\right.
\end{equation}
and an $\varepsilon_t^k$-optimal solution $(\lambda_m^k, \mu_m^k)$ of the dual problem
\begin{equation}\label{dualpbisddpnlp}
\begin{array}{l}
\displaystyle \max_{\lambda, \mu, x_m}  h_{t, x_n^k}^{k m}( \lambda , \mu )\\
\lambda = A_m x_m + B_m x_n^k  - b_m,\;x_m \in \mbox{Aff}( \mathcal{X}_t ), \; \mu \geq 0,
\end{array}
\end{equation}
where $h_{t, x_n^k}^{k m}$ is the dual function with $h_{t, x_n^k}^{k m}(\lambda, \mu)$ given by the optimal value of
{\small{
\begin{equation}\label{defdualisddp}
\left\{ 
\begin{array}{l}
\displaystyle \inf_{x_m} \mathcal{L}_{t m}^k  (x_m,\lambda,\mu) :=  F_t^k ( x_m, x_n^k , \xi_m ) + \langle  \lambda , A_m x_m + B_m x_n^k -b_m \rangle + \langle \mu , g_t ( x_m , x_n^k , \xi_m ) \rangle \\
x_m \in \mathcal{X}_t.
\end{array}
\right.
\end{equation}
}}
We now check that Assumption (A1-NL) implies that the following Slater type constraint qualification holds for problem \eqref{primalpbisddpnlp}
(i.e., for all problems solved in the backward passes):
\begin{equation}\label{slater}
\mbox{there exists }{\tilde x}_m^{B k} \in \mbox{ri}(\mathcal{X}_{t}) \mbox { such that }A_m {\tilde x}_m^{B k} + B_m x_{n}^k = b_m \mbox{ and }g_t({\tilde x}_m^{B k},  x_{n}^k , \xi_m)<0.
\end{equation}
The above constraint qualification is the analogue of \eqref{slaterboundmulti} for problem \eqref{primalpbisddpnlp}.
\begin{lemma}\label{boundcutcoeff} 
Let Assumption (A1-NL) hold. Then for every $k \in \mathbb{N}^*$, \eqref{slater} holds. 
\end{lemma}
\begin{proof}
Let $j=j(m)$ such that $\xi_{t j}=\xi_{m}$.
If $x_{n}^k = {\bar x}_{t j t-1}$ then recalling (A1-NL)-(e), \eqref{slater} holds with ${\tilde x}_m^{B k}  = {\bar x}_{t j t}$.
Otherwise, we define
$$
x_{n}^{k \varepsilon}  = x_{n}^k + \varepsilon \frac{x_{n}^k - {\bar x}_{t j t-1}}{\|x_{n}^k - {\bar x}_{t j t-1}\|}.
$$
Observe that since $x_{n}^k \in \mathcal{X}_{t-1}$, we have $x_{n}^{k \varepsilon} \in \mathcal{X}_{t-1}^{\varepsilon}$. Setting
$$
X_{t m}=\{(x_t, x_{t-1}) \in \mbox{ri}(\mathcal{X}_{t}) \small{\times} \mathcal{X}_{t-1}^{\varepsilon}: A_m x_t + B_m x_{t-1} = b_m,\;g_t(x_t, x_{t-1},\xi_m) \leq 0   \},
$$
since $x_{n}^{k \varepsilon} \in \mathcal{X}_{t-1}^{\varepsilon}$, using (A1-NL)-(d), there exists
$x_{m}^{k \varepsilon} \in \mbox{ri}(\mathcal{X}_t )$ such that
$(x_{m}^{k \varepsilon}, x_{n}^{k \varepsilon}) \in X_{t m}$.
Now clearly, since $\mathcal{X}_t$ and $\mathcal{X}_{t-1}$ are convex, the set $\mbox{ri}(\mathcal{X}_{t}) \small{\times} \mathcal{X}_{t-1}^{\varepsilon}$
is convex too and using (A1-NL)-(c), we obtain that $X_{t m}$ is convex.
Since $({\bar x}_{t j t}, {\bar x}_{t j t-1}) \in X_{t m}$ (due to Assumption (A1-NL)-(e)) 
and recalling that $(x_{m}^{k \varepsilon}, x_{n}^{k \varepsilon}) \in X_{t m}$, we obtain
that for every $0<\theta<1$, the point 
\begin{equation}\label{defxttheta}
(x_t(\theta), x_{t-1}(\theta))=(1-\theta) ({\bar x}_{t j t}, {\bar x}_{t j t-1}) + \theta (x_{m}^{k \varepsilon}, x_{n}^{k \varepsilon}) \in X_{t m}. 
\end{equation}
For 
\begin{equation}\label{deftheta0}
0<\theta=\theta_0=\frac{1}{1+\frac{\varepsilon}{ \|x_{n}^k - {\bar x}_{t j t-1}\|}}<1,
\end{equation}
we get $x_{t-1}(\theta_0)=x_{n}^k$, $x_t(\theta_0 ) \in \mbox{ri}(\mathcal{X}_{t}), 
A_m x_t(\theta_0 ) + B_m x_{t-1}(\theta_0 ) = A_m x_t(\theta_0 ) + B_m x_{n}^k = b_m$, 
and since $g_{t i}, i=1,\ldots,p$, are convex on $\mathcal{X}_t \small{\times} \mathcal{X}_{t-1}^{\varepsilon}$  (see Assumption (A1-NL)-(c)) and therefore on $X_{t m}$, we get 
$$
\begin{array}{lll}
g_t(x_t(\theta_0), x_{t-1}(\theta_0),\xi_m) & = & g_t(x_t(\theta_0), x_{n}^k, \xi_{t j}) \\  
&\leq & \underbrace{(1-\theta_0)}_{>0} \underbrace{g_t( {\bar x}_{t j t}, {\bar x}_{t j t-1}, \xi_{t j} )}_{<0} + 
\underbrace{\theta_0}_{>0} \underbrace{g_t( x_{m}^{k \varepsilon}, x_{n}^{k \varepsilon }, \xi_{t j} )}_{\leq 0} <0.
\end{array}
$$
Therefore, we have justified that \eqref{slater} holds with ${\tilde x}_m^{B k} = x_t(\theta_0 )$.
\end{proof}
From \eqref{slater}, we deduce that the optimal value ${\underline{\mathfrak{Q}}}_t^k(x_n^k , \xi_m )$ of primal problem \eqref{primalpbisddpnlp}
is the optimal value of dual problem \eqref{dualpbisddpnlp} and therefore $\varepsilon_t^k$-optimal dual solution
$(\lambda_m^k , \mu_m^k)$ satisfies:
\begin{equation}\label{defepssoldual}
{\underline{\mathfrak{Q}}}_t^k(x_n^k , \xi_m ) - \varepsilon_t^k  \leq  h_{t, x_n^k}^{k m}( \lambda_m^k  , \mu_m^k ) \leq {\underline{\mathfrak{Q}}}_t^k(x_n^k , \xi_m ). 
\end{equation}
We now use the results of Section \ref{sec:computeinexactcutsnlp} to derive an inexact cut $\mathcal{C}_t^k$ for $\mathcal{Q}_t$ at $x_n^k$ (recall that $n=n_{t-1}^k$).
Problem \eqref{primalpbisddpnlp} can be rewritten as
\begin{equation}\label{backwardpbrewritten}
\left\{
\begin{array}{l}
\displaystyle \inf_{x_m, y_m} \;f_t( x_m , x_n^k , \xi_m) + y_m \\ 
x_m \in X_t(x_n^k , \xi_m  ), y_m \geq \theta_{t+1}^j -  \eta_{t+1}^j (  \varepsilon_{t+1}^j ) + \langle \beta_{t+1}^{j}, x_m -  x_{n_t^j}^j \rangle, j=1,\ldots,k,
\end{array}
\right.
\end{equation}
which is of form \eqref{vfunctionq} with 
$
y=[x_m;y_m], x=x_{n}^k, f(y,x)=f_t(x_m,x_n^k,\xi_m)+y_m, 
A=[A_m \;0_{q \small{\times} 1} ], B=B_m, b=b_m, g(y,x)=g_t(x_m,x_n^k,\xi_m),
Y=\{y=[x_m ; y_m ] :x_m \in \mathcal{X}_t, B_{t+1}^k y \leq b_{t+1}^k \},
$
where the $j$-th line of matrix $B_{t+1}^k$ is $[(\beta_{t+1}^j)^T, -1]$
and where the $j$-th component of $b_{t+1}^k$ is $-\theta_{t+1}^j +  \eta_{t+1}^j (  \varepsilon_{t+1}^j ) + \langle \beta_{t+1}^{j} x_{n_t^j}^j \rangle$.

Therefore denoting by $(x_m^{B k}, y_m^{B k})$ an optimal solution of optimization problem \eqref{backwardpbrewritten}, by 
$\ell_{t}^{k m}(x_m^{B k}, x_n^k$, $\lambda_m^k$, $\mu_m^k$,$\xi_m)$ the optimal value of the optimization problem\footnote{Observe that this is a linear program if $\mathcal{X}_t$ is polyhedral.}
\begin{equation}\label{l2isddpnlp}
{\footnotesize{
\begin{array}{l}
\begin{array}{l}
\displaystyle \max \; \displaystyle  \langle \nabla_{x_t} f_t ( x_m^{B k}, x_n^k, \xi_m )  +  A_m^T \lambda_m^k + \sum_{i=1}^p \mu_m^{k}(i) \nabla_{x_t} g_{t i}(x_m^{B k}, x_n^k, \xi_m ) , x_m^{B k} - x_m  \rangle 
+ y_m^{B k} - y_m,\\
x_m \in \mathcal{X}_t, B_{t+1}^k [x_m;y_m] \leq b_{t+1}^k,
\end{array}
\end{array}
}}
\end{equation}
and introducing  coefficients
{\small{
\begin{equation}\label{defcoeffsisddp}
\begin{array}{lcl}
\theta_{t}^{k m}& =& \mathcal{L}_{t m}^k  (x_m^{B k},\lambda_m^k,\mu_m^k) = f_t ( x_m^{B k}, x_n^k, \xi_m ) +  \mathcal{Q}_{t+1}^k ( x_m^{B k} )  + \langle \mu_m^k ,  g_{t} ( x_m^{B k} , x_n^k , \xi_m ) \rangle,\\
\eta_{t}^{k m}(\varepsilon_t^k )& =& \ell_{t}^{k m}(x_m^{B k}, x_n^k , \lambda_m^k , \mu_m^k, \xi_m ) ,\\
\beta_t^{k m}&=&\nabla_{x_{t-1}} f_t ( x_m^{B k}, x_n^k, \xi_m )+ B_m^T \lambda_m^k +  \sum_{i=1}^p \mu_m^k(i) \nabla_{x_{t-1}} g_{t i} ( x_m^{B k} , x_n^k , \xi_m ),
\end{array}
\end{equation}
}}then using Proposition \ref{varprop1} we obtain that $\theta_{t}^{k m}-\eta_{t}^{k m}(\varepsilon_t^k ) + \langle \beta_t^{k m} , \cdot - x_n^k \rangle $
is an inexact cut for ${\underline{\mathfrak{Q}}}_t^k(\cdot , \xi_m )$ at $x_n^k$.\footnote{Note that the assumptions
of Proposition \ref{varprop1} are satisfied. In particular, 
$f_t(\cdot,x_n^k,\xi_m) + \mathcal{Q}_{t+1}^k(\cdot)$ is bounded from below on the feasible set  of \eqref{primalpbisddpnlp}  and the optimal value of $y_m$ in \eqref{backwardpbrewritten} and \eqref{l2isddpnlp} is finite.
In fact, problems \eqref{backwardpbrewritten} and \eqref{l2isddpnlp} can be equivalently rewritten as an optimization problem over a compact set adding the
constraints  $\min_{x_t \in \mathcal{X}_t} \mathcal{Q}_{t+1}^1(x_t) \leq   y_m \leq \max_{x_t \in \mathcal{X}_t} \mathcal{Q}_{t+1}( x_t)$ on $y_m$ and with such reformulation Proposition \ref{varprop2} applies too.}
It follows that setting 
\begin{equation}\label{formulathetaknlp}
\theta_t^k=\sum_{m \in C(n)} p_m \theta_{t}^{k m},\;\;\eta_t^{k}(\varepsilon_t^k ) = \sum_{m \in C(n)} p_m \eta_{t}^{k m}(\varepsilon_t^k ) ,\;\;\beta_t^k=\sum_{m \in C(n)} p_m  \beta^{k m},
\end{equation}
the 
affine function $\mathcal{C}_t^k(\cdot)=\theta_t^k-\eta_t^{k}(\varepsilon_t^k )+\langle \beta_t^k ,\cdot - x_n^k \rangle $
is an inexact cut for $\mathbb{E}_{\xi_t}[{\underline{\mathfrak{Q}}}_t^k(\cdot , \xi_t )]$ and therefore for $\mathcal{Q}_t$.

The computation of coefficients \eqref{formulathetaknlp} ends the backward pass and iteration $k$.

\begin{remark}\label{stoppingisddp} Since $\mathcal{Q}_t^k$ is a lower bound on $\mathcal{Q}_t$,
a stopping criterion similar to the one used with SDDP can be used. 
For that, we need to compute a valid lower bound in the forward passes solving
exactly the first stage problems in the forward passes taking $\delta_1^k=0$.
\end{remark}

\begin{remark}
We assumed that for ISDDP-NLP 
nonlinear optimization problems  are solved approximately
whereas linear optimization problems are solved exactly. 
Since in ISDDP-NLP we compute the optimal value
$\ell_{t}^{k m}(x_m^{B k}, x_n^k , \lambda_m^k , \mu_m^k, \xi_m )$
of optimization problem \eqref{l2isddpnlp}, it is assumed that these problems are linear. Since these optimization problems have a linear objective function, they are linear programs if and only if 
$\mathcal{X}_t$ is polyhedral. If this is not the case then
(a) either we add components  to $g$ 
pushing
the nonlinear constraints in the representation of $\mathcal{X}_t$ in $g$ or
(b) we also solve \eqref{l2isddpnlp} approximately.  
In Case (b), we can still build an inexact cut $\mathcal{C}_t^k$ (see Remark \ref{remextincutvar}) and study the 
convergence of the corresponding variant of ISDDP-NLP along the lines of Section \ref{sec:convsddp}.
\end{remark}

\subsection{Convergence analysis}\label{sec:convsddp}

In Proposition \ref{propboundeddualsto}, we show that the cut coefficients and approximate dual solutions computed in the backward passes are almost surely bounded
with the following additional assumption:\\

\par (SL-NL) For $t=2, \ldots, T$, there exists $\kappa_t>0, r_t>0$ such that for every $x_{t-1} \in \mathcal{X}_{t-1}$,
for every $j=1,\ldots,M$,
there exists $x_t \in \mathcal{X}_t$ such that $\mathbb{B}(x_t, r_t) \cap \mbox{Aff}( \mathcal{X}_t ) \subseteq  \mathcal{X}_t$,
$A_{t j} x_t + B_{t j} x_{t-1}=b_{t j}$, and for every $i=1,\ldots,p$, $g_{t i}( x_t, x_{t-1}, \xi_{t j}) \leq -\kappa_t$. \\
\begin{proposition}\label{propboundeddualsto} Assume that errors $(\varepsilon_t^k)_{k \geq 1}$ are bounded: for 
$t=1,\ldots,T$, we have $0 \leq \varepsilon_t^k \leq {\bar \varepsilon}_t< +\infty$.
If Assumptions (A0), (A1-NL), and (SL-NL) hold then the sequences 
$(\theta_{t}^k)_{t ,k}$, $(\eta_{t}^k (  \varepsilon_{t}^k ))_{t, k}$, $( \beta_t^k )_{t, k}$, 
$(\lambda_{m}^k )_{m, k}$, $( \mu_{m}^k )_{m, k}$
 generated by the ISDDP-NLP algorithm are almost surely bounded: for $t=2,\ldots,T+1$, there exists
 a compact set $C_t$ such that the sequence  
 $(\theta_{t}^k , \eta_{t}^k (  \varepsilon_{t}^k ), \beta_t^k )_{k \geq 1}$ almost surely belongs to 
 $C_t$ and for every $t=1,\ldots,T-1$, for every node $n$ of stage $t$, for every $m \in C(n)$,
 there exists
 a compact set $\mathcal{D}_m$ such that the sequence  
 $( \lambda_{m}^k , \mu_m^k )_{k : n_t^k = n}$ almost surely belongs to 
 $\mathcal{D}_m$.
\end{proposition}
\begin{proof} The proof is by backward induction on $t$. Our induction hypothesis 
$\mathcal{H}(t)$
for 
$t \in \{2,\ldots,T+1\}$ is that the sequence $(\theta_{t}^k , \eta_{t}^k (  \varepsilon_{t}^k ), \beta_t^k )_{k \geq 1}$ belongs to 
a compact set $C_t$. $\mathcal{H}(T+1)$ holds because for $t=T+1$ the corresponding coefficients are null.
Now assume that $\mathcal{H}(t+1)$ holds for some $t \in \{2,\ldots,T\}$ and take an arbitrary $n \in {\tt{Nodes}}(t-1)$ and $m \in C(n)$. We want to show that $\mathcal{H}(t)$
holds and 
that the sequence  $( \lambda_{m}^k , \mu_m^k )_{k: n_{t-1}^k =n}$ belongs to 
some compact set $\mathcal{D}_m$. Since $f_t(\cdot,\cdot,\xi_m), g_t(\cdot,\cdot,\xi_m) \in \mathcal{C}^1( \mathcal{X}_t \small{\times} \mathcal{X}_{t-1} )$
we can find finite $m_t, M_{t 1}, M_{t 2}, M_{t 3}, M_{t 4}$ such that for every $x_t \in \mathcal{X}_t, x_{t-1} \in \mathcal{X}_{t-1}$, 
for every $i=1,\ldots,p$, for every $m \in C(n)$,  we have
$\| \nabla_{x_t, x_{t-1}} f_t ( x_t , x_{t-1}, \xi_m)  \| \leq M_{t 2}$, 
$\| \nabla_{x_t, x_{t-1}} g_{t i} ( x_t , x_{t-1},\xi_m ) \| \leq M_{t 3}$,
$m_t \leq f_t( x_t , x_{t-1},\xi_m) \leq M_{t 1}$,
and  $\| g_t ( x_t , x_{t-1},\xi_m) \| \leq M_{t 4}$.
Also since $\mathcal{H}(t+1)$ holds, the sequence $(\|\beta_{t+1}^k\|)_{k \geq 1}$ is bounded
from above by, say, $L_{t+1}$, which is a Lipschitz constant for all
functions $(\mathcal{Q}_{t+1}^k )_{k \geq 1}$.
We now    
derive a bound on $\|(\lambda_m^k , \mu_m^k )\|$ using Proposition \ref{propboundnormepsdualsol}  
and Corollary \ref{corpropboundnormepsdualsol}. 
We will denote by $L(\mathcal{Q}_{t+1})$ a Lipschitz constant of $\mathcal{Q}_{t+1}$ on $\mathcal{X}_t$
(see Lemma \ref{convrecfuncQtS}). Let us check that the assumptions of this corollary
are satisfied for problem \eqref{primalpbisddpnlp}: 
\begin{itemize}
\item[(i)] $\mathcal{X}_t$ is a closed convex set;
\item[(ii)] $F_t^k(\cdot, x_{n}^k , \xi_m)$ is bounded from above by
${\bar f}_m(\cdot)=f_t(\cdot, x_{n}^k ,\xi_m) + \mathcal{Q}_{t+1}( \cdot )$. Since 
$f_t(\cdot,\cdot,\xi_m)$ is convex and finite in a neighborhood of $\mathcal{X}_t \small{\times} \mathcal{X}_{t-1}$,
it is Lipschitz continuous on $\mathcal{X}_t \small{\times} \mathcal{X}_{t-1}$ with Lipschitz constant,
say, $L_m(f_t)$. Therefore ${\bar f}_m$ is Lipschitz continuous with Lipschitz constant
$L_m(f_t)+L(\mathcal{Q}_{t+1})$ on $\mathcal{X}_t$.
\item[(iii)] Since all components of $g_t(\cdot,\cdot,\xi_m)$ are convex and finite in a neighborhood of $\mathcal{X}_t \small{\times} \mathcal{X}_{t-1}$,
they are  Lipschitz continuous on $\mathcal{X}_t \small{\times} \mathcal{X}_{t-1}$.
\item[(iv)]  $\mathcal{L}_m= \displaystyle \min_{x_{t-1} \in \mathcal{X}_{t-1}} {\underline{\mathfrak{Q}}}_t^1(x_{t-1} , \xi_m )$ 
is a (finite) lower bound for the objective function on the feasible set
(the minimum is well defined due to (A1-NL) and $\mathcal{H}(t)$).
\end{itemize}
Due to Assumption (SL-NL) we can find $\hat x_m^k$ such that 
$\mathbb{B}_n( \hat x_m^k , r_t ) \cap \mbox{Aff}( \mathcal{X}_t ) \subseteq \mathcal{X}_t$ and
$\hat x_m^k \in X_t( x_{n}^k , \xi_m)$.
Therefore, reproducing the reasoning of Section \ref{sec:boundingmulti}, we can find
$\rho_m>0$ such that 
$
\mathbb{B}_q( 0, \rho_m ) \cap A_m V_{\mathcal{X}_t}  \subseteq A_m \Big( \mathbb{B}_n( 0 , r_t )  \cap  V_{\mathcal{X}_t} \Big)
$
where $V_{\mathcal{X}_t}$ is the vector space $V_{\mathcal{X}_t}=\{x-y,\;x,y \in \mbox{Aff}( \mathcal{X}_t )\}$
(this is relation \eqref{deductionfromSLrewr} for problem \eqref{primalpbisddpnlp}).
Applying Corollary \ref{corpropboundnormepsdualsol} to problem \eqref{primalpbisddpnlp} we deduce that
$\|(\lambda_m^k , \mu_m^k )\| \leq U_t:= \max_{m \in C(n)} U_{t m}$ 
where\footnote{Observe that $U_{t m}$ does not depend on $k$. In particular, 
the only relation radius $\rho_m$ (involved in the formula giving $U_{t m}$)
has to satisfy is
$
\mathbb{B}_q( 0, \rho_m ) \cap A_m V_{\mathcal{X}_t}  \subseteq A_m \Big( \mathbb{B}_n( 0 , r_t )  \cap  V_{\mathcal{X}_t} \Big)
$
and this relation does not depend on $k$.}
{\small{
$$
U_{t m}
= \frac{
(  L_m(f_t)   + L( \mathcal{Q}_{t+1} )  ) r_t + {\bar{\varepsilon}}_t +
\displaystyle  \max_{x_t \in \mathcal{X}_t, x_{t-1} \in \mathcal{X}_{t-1}}(f_t( x_t ,x_{t-1},\xi_m ) + \mathcal{Q}_{t+1}(x_t)) - \mathcal{L}_m         }  {\min(\rho_m, \frac{\kappa_t}{2})}.
$$
}}
Now let $n=n_{t-1}^k$. For 
$\theta_t^k = \sum_{m \in C(n)} p_m \theta_t^{k m}$,
we get the bound 
$
m_t -U_t M_{t 4} +  \min_{x_t \in \mathcal{X}_t} \mathcal{Q}_{t+1}^1 ( x_t ) \leq \theta_t^k \leq M_{t 1}+   \max_{x_t \in \mathcal{X}_t} \mathcal{Q}_{t+1} ( x_t ).
$
Note that $\eta_{t}^k (  \varepsilon_{t}^k ) \geq 0$ and 
the objective function of problem \eqref{l2isddpnlp}  with optimal value 
$\eta_{t}^{k m} (  \varepsilon_{t}^k )$ is bounded from above on the feasible set by
$
{\bar \eta}_t = \Big( M_{t 2} + \sqrt{2} \max( \max_{m \in C(n)} \|A_m^T\| , M_{t 3} \sqrt{p} ) U_t   + L( \mathcal{Q}_{t+1} ) \Big) D( \mathcal{X}_t  )
$
and therefore the same upper bound holds for $\eta_{t}^k (  \varepsilon_{t}^k )$.
Finally, recalling definition \eqref{formulathetaknlp} of $\beta_t^k$ we have
$
\| \beta_t^k \| \leq
L_t := M_{t 2} + \sqrt{2} \max(\displaystyle \max_{m \in C(n)} \|B_m^T\| , M_{t 3} \sqrt{p} ) U_t, 
$
which completes the proof and provides a Lipschitz constant
 $L_t$ valid for functions $(\mathcal{Q}_t^k)_k$.
\end{proof}

We will assume that the sampling procedure in ISDDP-NLP satisfies (A2) (see Section \ref{conv-analsddp}).

To show that the sequence of error terms $(\eta_t^k (  \varepsilon_t^k ))_k$
converges to 0 when $\displaystyle \lim_{k \rightarrow +\infty} \varepsilon_t^k $~$= 0$,
we will make use of Proposition \ref{propvanish1dual} which follows:
\begin{proposition}\label{propvanish1dual} Let $Y \subset \mathbb{R}^n, X \subset \mathbb{R}^m$,  be two nonempty compact convex sets.
Let $f \in \mathcal{C}^1 (  Y\small{\times}X )$ be convex on $Y\small{\times}X$.
Let $(\mathcal{Q}^k)_{k \geq 1}$ be a sequence of convex $L$-Lipschitz continuous functions
on $Y$ satisfying $\underline{\mathcal{Q}}  \leq \mathcal{Q}^k \leq {\bar{\mathcal{Q}}}$ on $Y$
where $\underline{\mathcal{Q}}, {\bar{\mathcal{Q}}}$ are continuous on $Y$.
Let $g \in \mathcal{C}^1( Y \small{\times} X  )$ with components 
$g_i, i=1,\ldots,p$, 
convex on $Y\small{\times}X^{\varepsilon}$ for some $\varepsilon>0$. We also assume 
{\small{
$$
(H):\;\exists r, \kappa >0\;:\;\forall x \in X \;\exists y \in Y :\,\mathbb{B}_n( y, r ) \cap \emph{Aff}( Y) \subseteq Y,\;A y + B x = b,\;g(y, x) \leq  -\kappa {\textbf{e}},
$$
}}where ${\textbf{e}}$ is a vector of ones of size $p$.
Let $(x^k)_{k \geq 1}$ be a sequence in $X$, let $(\varepsilon^k)_{k \geq 1}$ be a sequence of nonnegative real numbers, 
and let $y^k ( \varepsilon^k )$ be an $\varepsilon^k$-optimal and feasible solution to 
\begin{equation}\label{defpblemmavanish4}
\inf \;\{f(y, x^k ) + \mathcal{Q}^k ( y) \;\;: \;\; y \in Y, \;Ay + B x^k = b, \;g(y, x^k) \leq 0 \} .
\end{equation}
Let $(\lambda^k(  \varepsilon^k ), \mu^k ( \varepsilon^k ) )$ be an $\varepsilon^k$-optimal solution to the dual
problem 
\begin{equation}\label{dualpropconv0etak}
\begin{array}{l}
\sup_{\lambda, \mu} \;h_{x^k}^k ( \lambda, \mu ) \\
\lambda = A y + B x^k -b ,\;y \in \emph{Aff}(Y),\;\mu \geq 0,
\end{array}
\end{equation}
where 
$
h_{x^k}^k ( \lambda, \mu )=\displaystyle \inf_{y \in Y} \{f(y, x^k ) + \mathcal{Q}^k ( y) + \langle \lambda , A y + B x^k -b \rangle +  \langle \mu , g(y, x^k ) \rangle \}.
$
Define $\eta^k (  \varepsilon^k )$ as the optimal value of the following optimization problem:
{\small{
\begin{equation}\label{defpblemmavanish5}
\begin{array}{ll}
\displaystyle \max_{y \in Y} & \left \langle \nabla_y f(y^k ( \varepsilon^k ), x^k ) +A^T \lambda^k (  \varepsilon^k  )  +\displaystyle \sum_{i=1}^p \mu^k(\varepsilon^k )( i ) \nabla_{y} g_i( y^k ( \varepsilon^k )  , x^k ) , y^k ( \varepsilon^k )  - y \right \rangle  \\
&+ \mathcal{Q}^k ( y^k ( \varepsilon^k  ) ) - \mathcal{Q}^k ( y ).
\end{array}
\end{equation}
}}
Then if $\lim_{k \rightarrow +\infty} \varepsilon^k = 0$ we have
\begin{equation}\label{vanishnoises5}
\lim_{k \rightarrow +\infty} \eta^k (  \varepsilon^k ) = 0.
\end{equation}
\end{proposition}
\begin{proof}
For simplicity, we write $\lambda^k, \mu^k, y^k $ instead of 
$\lambda^k(  \varepsilon^k ), \mu^k ( \varepsilon^k ) ), y^k ( \varepsilon^k )$, and 
put $\mathcal{Y}(x)=\{y \in Y: \;Ay + B x = b,\;g(y,x) \leq 0\}$.
Denoting by $y_*^{k} \in \mathcal{Y}( x^k ) $ an optimal solution of \eqref{defpblemmavanish4}, we get
\begin{equation}\label{defepstkredef}
f( y_*^{k}, x^k ) + \mathcal{Q}^k ( y_*^{k} ) \leq 
f( y^{k}, x^k ) + \mathcal{Q}^k ( y^{k} ) \leq f ( y_*^{k} , x^k ) + \mathcal{Q}^k ( y_*^{k} ) + \varepsilon^k.
\end{equation}
We prove \eqref{vanishnoises5} by contradiction. 
Let ${\tilde y}^k$ be an optimal solution of \eqref{defpblemmavanish5}:
$$
\eta^k ( \varepsilon^k )=\langle \nabla_{y} f( y^{k}, x^k ) + 
A^T \lambda^k + \sum_{i=1}^p \mu^k ( i) \nabla_{y} g_{i}( y^{k}, x^k ),y^{k} - {\tilde y}^k   \rangle -\mathcal{Q}^k( {\tilde y}^k ) + \mathcal{Q}^k( y^{k} ).
$$
Assume that \eqref{vanishnoises5} does not hold. Then 
there exists $\varepsilon_0 >0$ and $\sigma_1: \mathbb{N} \rightarrow \mathbb{N}$ increasing such that
for every $k \in \mathbb{N}$ we have
{\small{
\begin{equation}\label{contradeta22}
\begin{array}{l}
\left \langle \nabla_{y} f({y}^{\sigma_1(k)}, x^{ \sigma_1(k) } )  +  A^T \lambda^{\sigma_1(k)} + 
\sum_{i=1}^p \mu^{\sigma_1(k)} ( i) \nabla_{y} g_{i}( y^{\sigma_1(k)}, x^{\sigma_1(k)} ),  -{\tilde y}^{ \sigma_1(k) }  + y^{\sigma_1(k)} \right  \rangle \\ 
+ \mathcal{Q}^{\sigma_1(k)}( y^{\sigma_1(k)}  ) - \mathcal{Q}^{\sigma_1(k)}( {\tilde y}^{\sigma_1(k)}  ) \geq \varepsilon_0. 
\end{array}
\end{equation}
}}

Now denoting by $\mathcal{C}(Y)$ the set of continuous real-valued functions on $Y$, equipped with 
norm $\|f\|_{Y} = \sup_{y \in Y}|f(y)|$, observe that the sequence $( \mathcal{Q}^{\sigma_1(k)} )_k$ in $\mathcal{C}( Y )$
\begin{itemize}
\item[(i)] is bounded: for every $k \geq 1$, for every $y \in Y$, we have:
$
-\infty < \min_{y \in Y} \underline{\mathcal{Q}}( y ) \leq \mathcal{Q}^{\sigma_1(k)}( y ) \leq \max_{y \in Y}   {\bar{\mathcal{Q}}}( y ) < +\infty;
$
\item[(ii)] is equicontinuous since functions $(\mathcal{Q}^{\sigma_1(k)})_k$ are Lipschitz continuous with Lipschitz constant $L$.
\end{itemize}
Therefore using the Arzel\`a-Ascoli theorem, this sequence has a uniformly convergent subsequence: there exists
$\mathcal{Q}^* \in \mathcal{C}( Y )$ and $\sigma_2: \mathbb{N} \rightarrow \mathbb{N}$ increasing 
such that setting $\sigma =\sigma_1 \circ \sigma_2$, we have
$\lim_{k \rightarrow +\infty} \|\mathcal{Q}^{\sigma(k)} - \mathcal{Q}^*  \|_{Y} = 0$.
Using Assumption (H) and Proposition \ref{propboundnormepsdualsol}, we obtain that
the sequence $(\lambda^{\sigma(k)}, \mu^{\sigma(k)})$ is a sequence of a compact set, say $\mathcal{D}$.
Since $({y}^{\sigma(k)}, y_*^{\sigma(k)}, {\tilde y}^{ \sigma(k) }, x^{ \sigma(k) } )_{k \geq 1}$ is a sequence of the compact set 
$Y \small{\times} Y \small{\times} Y \small{\times} X$, taking further a subsequence if needed,
we can assume that
$({y}^{\sigma(k)}, y_*^{\sigma(k)}, {\tilde y}^{ \sigma(k) }, x^{ \sigma(k) }, \lambda^{ \sigma(k) }, \mu^{ \sigma(k) } )$ converges
to some $({\bar y} , y_* , {\tilde y} , x_*, \lambda_*, \mu_* ) \in Y \small{\times} Y \small{\times} Y \small{\times} X \small{\times} \mathcal{D}$.
It follows that there is $k_0 \in \mathbb{N}$ such that for every $k \geq k_0$:
\begin{equation}\label{contradeta23}
\begin{array}{l}
\left| \left \langle \nabla_{y} f({y}^{\sigma(k)}, x^{ \sigma(k) } )  + A^T \lambda^{\sigma(k)} +  \sum_{i=1}^p \mu^{\sigma(k)} ( i) \nabla_{y} g_{i}( y^{\sigma(k)}, x^{\sigma(k)} ),  -{\tilde y}^{ \sigma(k) }  + y^{\sigma(k)} \right  \rangle \right. \\
\;\; \left. - \left \langle \nabla_{y} f({\bar y}, x_* )  + A^T \lambda_{*} + \sum_{i=1}^p \mu_{*} ( i) \nabla_{y} g_{i}( {\bar y},  x_{*} ), -{\tilde y}^{ \sigma(k) }  + {\bar y} \right  \rangle  \right| \leq \varepsilon_0/4, \\
\|{y}^{\sigma(k)} - {\bar y}\| \leq \frac{\varepsilon_0}{8 L},\;\|\mathcal{Q}^{\sigma(k)} - \mathcal{Q}^*  \|_{Y} \leq \varepsilon_0/16.
\end{array}
\end{equation}
We deduce from \eqref{contradeta22}, \eqref{contradeta23} that
{\small{
\begin{equation}\label{caseBposcontrad}
\left \langle \nabla_{y} f({\bar y}, x_{* } )  + A^T \lambda^{*} + 
\sum_{i=1}^p \mu^{*} ( i) \nabla_{y} g_{i}( {\bar y}, x_{*} ), -{\tilde y}^{ \sigma(k_0 ) }  + {\bar y} \right  \rangle +  \mathcal{Q}^{*}( {\bar y}  ) - \mathcal{Q}^{*}( {\tilde y}^{\sigma(k_0 )}  ) \geq \varepsilon_0/2 >0.
\end{equation}
}}
Due to Assumption (H), primal problem \eqref{defpblemmavanish4} and dual problem \eqref{dualpropconv0etak} have the same
optimal value and for every $y \in Y$ and $k \geq 1$ we have:
{\small{
\begin{equation}\label{proofconvqtkcruc0}
\begin{array}{l}
 f({y}^{\sigma(k)}, x^{{\sigma(k)}} ) + \mathcal{Q}^{\sigma(k)}( y^{{\sigma(k)}}  ) + 
 \langle A y^{\sigma(k)} + B x^{\sigma(k)} - b ,   \lambda^{\sigma(k)} \rangle + \langle \mu^{\sigma(k)} ,  g( y^{\sigma(k)} , x^{\sigma(k)}) \rangle  \\
\stackrel{(a)}{\leq}  f ( y_*^{\sigma(k)} , x^{\sigma(k)} ) + \mathcal{Q}^{\sigma(k)} ( y_*^{\sigma(k)} )  + \varepsilon^{\sigma(k)},\\
\stackrel{(b)}{\leq} h_{x^{\sigma(k)}}^{\sigma(k)} (\lambda^{\sigma(k)} , \mu^{\sigma(k)} ) + 2 \varepsilon^{\sigma(k)},\\
\stackrel{(c)}{\leq} f(y, x^{{\sigma(k)}} )  + \langle A y + B x^{\sigma(k)} - b , \lambda^{\sigma(k)} \rangle 
+ \langle \mu^{\sigma(k)} ,  g( y , x^{\sigma(k)} ) \rangle  +  \mathcal{Q}^{\sigma(k)}( y ) + 2 \varepsilon^{\sigma(k)}.
\end{array}
\end{equation}
}}where we have used in \eqref{proofconvqtkcruc0}-(a) the definition of $y_*^{\sigma(k)}, y^{\sigma(k)}$ and the fact that $\mu^{\sigma(k)} \geq 0, y^{\sigma(k)} \in \mathcal{Y}( x^{\sigma(k)} )$, 
in \eqref{proofconvqtkcruc0}-(b) the fact that $(\lambda^{\sigma(k)} , \mu^{\sigma(k)})$ is an $\epsilon^{\sigma(k)}$-optimal dual solution and there is no duality gap, and
in \eqref{proofconvqtkcruc0}-(c) the definition of $h_{x^{\sigma(k)}}^{\sigma(k)}$.

Taking the limit in the above relation as $k \rightarrow +\infty$, we get  for every $y \in Y$:
$$
\begin{array}{l}
f({\bar y}, x_* )  + \langle A {\bar y} + B x_* - b , \lambda_{*} \rangle + \langle \mu_{*} , g( {\bar y} , x_* ) \rangle  + \mathcal{Q}^{*}( \bar y )\\
\leq  f(y, x_* )  + \langle A y + B x_* - b ,  \lambda_{*}  \rangle + \langle \mu_{*} ,  g( y , x_{*} ) \rangle + \mathcal{Q}^{*}( y ).
\end{array}
$$
Recalling that $\bar y \in Y$ this shows that $\bar y$ is an optimal solution of
\begin{equation}
\left\{
\begin{array}{l}
\min f( y, x_* ) + \mathcal{Q}^{*}( y ) + \langle A y + B x_*  - b ,  \lambda_* \rangle + \langle \mu_* , g( y, x_* ) \rangle \\
y \in Y.
\end{array}
\right.
\end{equation}
Now recall that all functions $(\mathcal{Q}^{\sigma(k)})_k$ are convex on $Y$ and therefore the function 
$\mathcal{Q}^*$ is convex on $Y$ too. It follows that the first order optimality conditions for $\bar y$ can be written
\begin{equation}\label{caseBposcontradfinal}
\left \langle \nabla_{y} f(\bar y, x_{* } )  + A^T \lambda_{*} + \sum_{i=1}^p \mu_{*} ( i) \nabla_{y} g_{i}( \bar y , x_{*} ), y  - \bar y  \right  \rangle
+  \mathcal{Q}^{*}( y  ) - \mathcal{Q}^{*}( \bar y  ) \geq 0
\end{equation}
for all $y \in Y$. Specializing the above relation for $y = {\tilde y}^{ \sigma(k_0 ) } $, we get 
$$
\left \langle \nabla_{y} f(\bar y , x_{* } )  + A^T \lambda_{*} + \sum_{i=1}^p \mu_{*} ( i) \nabla_{y} g_{i}( \bar y , x_{*} ), {\tilde y}^{ \sigma(k_0 ) }  - \bar y  \right  \rangle
+  \mathcal{Q}^{*}( {\tilde y}^{ \sigma(k_0 ) }  ) - \mathcal{Q}^{*}( \bar y  ) \geq 0,
$$
but the left-hand side of the above inequality is $\leq -\varepsilon_0/2<0$ due to \eqref{caseBposcontrad} which yields the desired contradiction.
\end{proof}

We can now study the convergence of ISDDP-NLP:
\begin{theorem}[Convergence of ISDDP-NLP]\label{convanisddp}
Consider the sequences of stochastic decisions $x_n^k$ and of recourse functions $\mathcal{Q}_ t^k$
generated by ISDDP-NLP.
Let Assumptions (A0), (A1-NL), (SL-NL), and (A2) hold and assume that for $t=2,\ldots,T$, we have
$\lim_{k \rightarrow +\infty} \varepsilon_{t}^k =0$
and for $t=1,\ldots,T$,
$\lim_{k \rightarrow +\infty} \delta_{t}^k =0$. Then
\begin{itemize}
\item[(i)] almost surely, for $t=2,\ldots,T+1$, the following holds:
$$
\mathcal{H}(t): \;\;\;\forall n \in {\tt{Nodes}}(t-1), \;\; \displaystyle \lim_{k \rightarrow +\infty} \mathcal{Q}_{t}(x_{n}^{k})-\mathcal{Q}_{t}^{k}(x_{n}^{k} )=0.
$$
\item[(ii)]
Almost surely, the limit of the sequence
$( {F}_1^{k-1}(x_{n_1}^k , x_0 ,  \xi_1) )_k$ of the approximate first stage optimal values
and of the sequence
$({\underline{\mathfrak{Q}}}_1^{k}(x_{0}, \xi_1))_k$
is the optimal value 
$\mathcal{Q}_{1}(x_0)$ of \eqref{pbtosolve}.
Let $\Omega=(\Theta_2 \small{\times} \ldots \small{\times} \Theta_T)^{\infty}$ be the sample space
of all possible sequences of scenarios equipped with the product $\mathbb{P}$ of the corresponding 
probability measures. Define on $\Omega$ the random variable 
$x^* = (x_1^*, \ldots, x_T^*)$ as follows. For $\omega \in \Omega$, consider the 
corresponding sequence of decisions $( (x_n^k( \omega ))_{n \in \mathcal{N}} )_{k \geq 1}$
computed by ISDDP-NLP. Take any accumulation point
$(x_n^* (\omega) )_{n \in \mathcal{N}}$ of this sequence. 
If $\mathcal{Z}_t$ is the set of $\mathcal{F}_t$-measurable functions,
define $x_1^*(\omega),\ldots,x_T^*(\omega)$ taking $x_t^{*}(\omega): \mathcal{Z}_t \rightarrow \mathbb{R}^n$ given by
$x_t^{*}(\omega)( \xi_1, \ldots, \xi_t  )=x_{m}^{*}(\omega)$ where $m$ is given by $\xi_{[m]}=(\xi_1,\ldots,\xi_t)$ for $t=1,\ldots,T$.
Then
$$
\mathbb{P}((x_1^*,\ldots,x_T^*) \mbox{ is an optimal solution to \eqref{pbtosolve}})  =1.
$$
\end{itemize}
\end{theorem}
\begin{proof} Let $\Omega_1$ be the event on the sample space $\Omega$  of sequences
of scenarios such that every scenario is sampled an infinite number of times.
Due to (A2), this event has probability one.
Take an arbitrary  realization $\omega$ of ISDDP-NLP in $\Omega_1$. 
To simplify notation we will use $x_n^k, \mathcal{Q}_t^k, \theta_t^k, \eta_t^k(\varepsilon_t^k), \beta_t^k, \lambda_m^k, \mu_m^k$ instead 
of $x_n^k(\omega), \mathcal{Q}_t^k(\omega), \theta_t^k(\omega)$, $\eta_t^k(\varepsilon_t^k)(\omega)$, $\beta_t^k( \omega ), \lambda_m^k( \omega ), \mu_m^k ( \omega )$.
\par Let us prove (i). 
We want to show that $\mathcal{H}(t), t=2,\ldots,T+1$, hold for that realization.
The proof is by backward induction on $t$. For $t=T+1$, $\mathcal{H}(t)$ holds
by definition of $\mathcal{Q}_{T+1}$, $\mathcal{Q}_{T+1}^k$. Now assume that $\mathcal{H}(t+1)$ holds
for some $t \in \{2,\ldots,T\}$. We want to show that $\mathcal{H}(t)$ holds.
Take an arbitrary node $n \in {\tt{Nodes}}(t-1)$. For this node we define 
$\mathcal{S}_n=\{k \geq 1 : n_{t-1}^k = n\}$ the set of iterations such that the sampled scenario passes through node $n$.
Observe that $\mathcal{S}_n$ is infinite because the realization of ISDDP-NLP is in $\Omega_1$.
We first show that 
$
\displaystyle \lim_{k \rightarrow +\infty, k \in \mathcal{S}_n } \mathcal{Q}_{t}(x_{n}^{k})-\mathcal{Q}_{t}^{k}(x_{n}^{k} )=0. 
$
For $k \in \mathcal{S}_n$, we have $n_{t-1}^k =n$, i.e., $x_n^k = x_{n_{t-1}^k}^k$, which implies
\begin{equation}\label{firsteqstosddp}
\mathcal{Q}_t ( x_n^k ) \geq \mathcal{Q}_t^k ( x_n^k ) \geq \mathcal{C}_t^k ( x_n^k ) = \theta_t^k - \eta_t^{k}(\varepsilon_t^k ) =  \sum_{m \in C(n)} p_m ( \theta_{t}^{k m} - \eta_{t}^{k m}(\varepsilon_t^k ) ). 
\end{equation}
Let us now bound $\theta_{t}^{k m}$ from below:
$$
\theta_t^{k m} \stackrel{\eqref{defcoeffsisddp}}{=}
\mathcal{L}_{t m}^k (x_m^{B k}, \lambda_m^k, \mu_m^k ) \geq 
h_{t, x_n^k}^{k m}( \lambda_m^k , \mu_m^k )  \stackrel{\eqref{defepssoldual}}{\geq} 
{\underline{\mathfrak{Q}}}_t^k(x_n^k , \xi_m ) - \varepsilon_t^k
$$
where for the first inequality we have used the definition of $h_{t, x_n^k}^{k m}$ and the fact that  
$x_m^{B k} \in \mathcal{X}_t$.
Next, we have the following lower bound on ${\underline{\mathfrak{Q}}}_t^k(x_n^k , \xi_m )$ for all $k \in \mathcal{S}_n$:
\begin{equation} \label{eqconv1bis0s}
\begin{array}{lcl}
{\underline{\mathfrak{Q}}}_t^k(x_n^k , \xi_m ) & \geq & {\underline{\mathfrak{Q}}}_t^{k-1}(x_n^k , \xi_m )  \mbox{ by monotonicity,}\\
&  \stackrel{\eqref{epssolforward}}{\geq}  &  F_t^{k-1}(x_m^{k}, x_{n}^k, \xi_m ) -  \delta_t^k,\\
& = &  F_t(x_m^{k}, x_{n}^k , \xi_m) + \mathcal{Q}_{t+1}^{k-1}(x_m^{k}) - \mathcal{Q}_{t+1}( x_m^{k})    -  \delta_t^k,\\
& \geq &  \mathfrak{Q}_t(x_{n}^k , \xi_m )+ \mathcal{Q}_{t+1}^{k-1}(x_m^{k}) - \mathcal{Q}_{t+1}( x_m^{k})  -  \delta_t^k, 
\end{array}
\end{equation}
where for the last inequality we have used the definition of $\mathfrak{Q}_t$  and the fact that $x_m^k \in X_t ( x_n^k , \xi_m )$.
Combining \eqref{firsteqstosddp} with \eqref{eqconv1bis0s} and using our lower bound on $\theta_t^{k m}$, we obtain
\begin{equation} \label{eqconv1bisfutures}
0 \leq \mathcal{Q}_{t}(x_{n}^k) - \mathcal{Q}_{t}^k(x_{n}^k)  \leq  \delta_t^k +  \varepsilon_t^k + \displaystyle \sum_{m \in C(n)} p_m \eta_t^{k m}( \varepsilon_t^ k ) + \displaystyle \sum_{m \in C(n)} p_m \Big(  \mathcal{Q}_{t+1}( x_m^k ) - \mathcal{Q}_{t+1}^{k-1}( x_m^k ) \Big).
\end{equation}
We now show that for every $m \in C(n)$, we have 
\begin{equation}\label{noisesvanishisddp}
\lim_{k \rightarrow +\infty, k \in \mathcal{S}_n} \eta_t^{k m}( \varepsilon_t^ k ) = 0. 
\end{equation}
Let us fix $m \in C(n)$. Decision $x_m^{B k}$ is an $\varepsilon_t^k$-optimal solution of
\begin{equation} \label{defxtkjisddp2}
\begin{array}{l}
\left\{
\begin{array}{l}
\displaystyle \inf_{x_m} \;f_t( x_m , x_n^k , \xi_m) + \mathcal{Q}_{t+1}^k ( x_m )\\ 
x_m \in X_t(x_n^k , \xi_m   ),\\
\end{array}
\right.
\end{array}
\end{equation}
and $\eta_t^{k m}( \varepsilon_t^ k )$ is the optimal value of the following optimization problem:
\begin{equation} \label{probetakmsto2}
\begin{array}{l}
{\small{
\begin{array}{ll}
\displaystyle \max_{x_m \in \mathcal{X}_t} & \displaystyle  \langle \nabla_{x_t} f_t ( x_m^{B k}, x_n^k, \xi_m )  +  A_m^T \lambda_m^k + \sum_{i=1}^p \mu_m^{k}(i) \nabla_{x_t} g_{t i}(x_m^{B k}, x_n^k, \xi_m ) , x_m^{B k} - x_m  \rangle  \\
& +\mathcal{Q}_{t+1}^{k}( x_m^{B k} ) - \mathcal{Q}_{t+1}^{k}( x_m ).
\end{array}
}}
\end{array}
\end{equation}
We now check that Proposition \ref{propvanish1dual} can be applied to problems \eqref{defxtkjisddp2}, \eqref{probetakmsto2} setting:
\begin{itemize}
\item $Y=\mathcal{X}_t, X=\mathcal{X}_{t-1}$ which are nonempty compact, and convex;
\item $f(y, x)=f_t(y,x,\xi_m)$ which is convex and continuously differentiable on $Y \small{\times} X$;
\item $g(y,x)=g_t(y,x,\xi_m) \in \mathcal{C}^1( Y \small{\times} X)$ with components $g_i,i=1,\ldots,p$, convex
on $Y \small{\times}X^{\varepsilon}$;
\item $\mathcal{Q}^k=\mathcal{Q}_{t+1}^k$ which is convex Lipschitz continuous on $Y$ with Lipschitz constant $L_{t+1}$
($L_{t+1}$ is an upper bound on 
$(\|\beta_{t+1}^k \|)_{k \in \mathcal{S}_n}$, see Proposition  \ref{propboundeddualsto}) and satisfies
$$
{\underline{Q}} := \mathcal{Q}_{t+1}^1 \leq \mathcal{Q}^k \leq {\bar{\mathcal{Q}}}:=\mathcal{Q}_{t+1}
$$
on $Y$ with ${\underline{Q}}, {\bar{\mathcal{Q}}}$ continuous on $Y$;
\item $(x^k ) = (x_{n}^k)_{k \in \mathcal{S}_n}$ sequence in $X$, $(y^k)_{k \in \mathcal{S}_n} = (x_m^{B k})_{k \in \mathcal{S}_n}$ sequence in $Y$, and
$(\lambda^k, \mu^k )_{k \in \mathcal{S}_n} = (\lambda_m^k , \mu_m^k )_{k \in \mathcal{S}_n}$. 
\end{itemize}
With this notation Assumption (H) is satisfied with $\kappa=\kappa_t$, since Assumption (SL-NL) holds.
Therefore we can apply Proposition \ref{propvanish1dual} to obtain \eqref{noisesvanishisddp}.

Next, recall that $\mathcal{Q}_{t+1}$ is convex; functions $(\mathcal{Q}_{t+1}^k)_k$ are $L_{t+1}$-Lipschitz;
and for all $k \geq 1$ we have $\mathcal{Q}_{t+1}^k \leq \mathcal{Q}_{t+1}^{k+1} \leq\mathcal{Q}_{t+1}$ on compact set $\mathcal{X}_t$.
Therefore, the induction hypothesis
$
\lim_{k \rightarrow +\infty} \mathcal{Q}_{t+1}( x_m^k ) - \mathcal{Q}_{t+1}^k ( x_m^k )=0 
$
implies, using Lemma A.1 in \cite{lecphilgirar12}, that
\begin{equation}\label{indchypreformulatedsto}
\lim_{k \rightarrow +\infty} \mathcal{Q}_{t+1}( x_m^k ) - \mathcal{Q}_{t+1}^{k-1} ( x_m^k )=0 . 
\end{equation}

Plugging \eqref{noisesvanishisddp} and \eqref{indchypreformulatedsto} into
\eqref{eqconv1bisfutures} we obtain 
\begin{equation}\label{mainisddp}
\displaystyle \lim_{k \rightarrow +\infty, k \in \mathcal{S}_n } \mathcal{Q}_{t}(x_{n}^{k})-\mathcal{Q}_{t}^{k}(x_{n}^{k} )=0. 
\end{equation}
It remains to show that
$
\displaystyle \lim_{k \rightarrow +\infty, k \notin \mathcal{S}_n } \mathcal{Q}_{t}(x_{n}^{k})-\mathcal{Q}_{t}^{k}(x_{n}^{k} )=0. 
$
This relation can be proved using Lemma 5.4 in \cite{guilejtekregsddp} which can be applied since 
(A) relation \eqref{mainisddp} holds (convergence was shown for the iterations in $\mathcal{S}_n$),
(B) the sequence $(\mathcal{Q}_t^k)_k$ is monotone, i.e., 
$\mathcal{Q}_t^k \geq \mathcal{Q}_t^{k-1}$ for all $k \geq 1$, (C) Assumption (A2) holds, and
(D) $\xi_{t-1}^k$ is independent on $( (x_{n}^j,j=1,\ldots,k), (\mathcal{Q}_{t}^j,j=1,\ldots,k-1))$.\footnote{Lemma 5.4 in \cite{guilejtekregsddp} is similar to the end of the proof of Theorem 4.1 in \cite{guiguessiopt2016} and uses
the Strong Law of Large Numbers. This lemma itself applies the ideas of the end of the convergence proof of SDDP given in \cite{lecphilgirar12}, which
was given with a different (more general) sampling scheme in the backward pass.} Therefore, we have shown (i).

\par (ii) The proof is similar to the proof of \cite[Theorem 4.1-(ii)]{guiguessiopt2016}. 
\end{proof}

\begin{remark}\label{importantremarknlp}
In ISDDP-NLP algorithm presented in Section \ref{sec:isddpalgo}, decisions are computed at every iteration for all the nodes of the scenario tree
in the forward pass.
However, in practice, at iteration $k$  decisions will only be computed for the nodes $(n_1^k,\ldots,n_T^k)$
and their children nodes. For this variant of ISDDP-NLP, the backward pass is exactly the same as the backward of ISDDP-NLP presented in Section \ref{sec:isddpalgo}
while the forward pass reads as follows: we select a set of nodes $(n_1^k, n_2^k, \ldots, n_T^k)$ 
with $n_t^k$ a node of stage $t$ ($n_1^k=n_1$ and for $t \geq 2$, $n_t^k$
is a child node of $n_{t-1}^k$)
corresponding to a sample $({\tilde \xi}_1^k, {\tilde \xi}_2^k,\ldots, {\tilde \xi}_T^k)$
of $(\xi_1, \xi_2,\ldots, \xi_T)$. More precisely, for $t=1,\ldots,T$,
setting $m=n_{t}^k$ and $n=n_{t-1}^k$, we compute a $\delta_t^k$-optimal solution $x_m^k$ of
\begin{equation} \label{defxtkjnlpbis}
{\underline{\mathfrak{Q}}}_t^{k-1}( x_n^k , \xi_m ) = \left\{
\begin{array}{l}
\displaystyle \inf_{y} \; F_t^{k-1}(y , x_n^k, \xi_m):= f_t( y , x_n^k , \xi_m ) + \mathcal{Q}_{t+1}^{k-1}( y ) \\
y \in X_t( x_n^k, \xi_m ),
\end{array}
\right.
\end{equation}
This variant of ISDDP-NLP will build the same cuts and compute the same decisions for the nodes of the
sampled scenarios as ISDDP-NLP described in Section \ref{sec:isddpalgo}. For this variant, for a node $n$, the decision variables $(x_n^k)_k$ are defined for
an infinite subset ${\tilde{\mathcal{S}}}_{n}$ of iterations where the sampled scenario passes through the parent node of node $n$, i.e., 
${\tilde{\mathcal{S}}}_{n}=\mathcal{S}_{\mathcal{P}(n)}$.
With this notation, for this variant, applying Theorem \ref{convanisddp}-(i), we get for $t=2,\ldots,T+1$,
$
\mbox{for all }n \in {\tt{Nodes}}(t-1), \lim_{k \rightarrow +\infty, k \in \mathcal{S}_{\mathcal{P}(n)}} \mathcal{Q}_{t}(x_{n}^{k})-\mathcal{Q}_{t}^{k}(x_{n}^{k} )=0
$
almost surely. Also a.s., the limit of the sequence
$({F}_1^{k-1}(x_{n_1}^k , x_0 , \xi_1) )_k$ of the approximate first stage optimal values
is the optimal value 
$\mathcal{Q}_{1}(x_0)$ of \eqref{pbtosolve}. The variant of ISDDP-NLP without sampling in the forward pass was presented  first, to allow for the application
of Lemma 5.4 from \cite{guilejtekregsddp}. More specifically, item (D): $\xi_{t-1}^k$ is independent on $( (x_{n}^j,j=1,\ldots,k), (\mathcal{Q}_{t}^j,j=1,\ldots,k-1))$, 
given in the end of the proof of Theorem \ref{convanisddp}-(i) does not apply for ISDDP-NLP with sampling in the forward pass.
\end{remark}

\section{Numerical experiments} \label{numexp}

Our goal in this section is to compare SDDP and ISDDP-LP (denoted for short ISDDP in what follows) on the risk-neutral portfolio problem with direct transaction costs presented in
Section 5.1 of \cite{guilejtekregsddp} (see \cite{guilejtekregsddp} for details).
For this application, $\xi_t$ is the vector of asset returns: if 
$n$ is the number of
risky assets, $\xi_t$ has size $n+1$,
$\xi_t(1:n)$ is the vector of risky asset returns for stage $t$ 
while $\xi_t(n+1)$ is the return of the risk-free asset. We generate four instances of this portfolio problem as follows.

For fixed $T$ (number of stages) and $n$ (number of risky assets),
the distributions of $\xi_t(1:n),t=2,\ldots,T$,  have $M$ realizations 
with $p_{t i}=\mathbb{P}(\xi_t = \xi_{t i})=1/M$, and
$\xi_1(1:n), \xi_{t 1}(1:n), \ldots, \xi_{t M}(1:n)$ obtained sampling
from a normal distribution with mean and standard deviation chosen 
randomly in respectively the intervals $[0.9,1.4]$ and $[0.1,0.2]$.
The monthly return $\xi_t(n+1)$ of the risk-free asset is $1.01$ for all $t$.
The initial portfolio $x_0$ has components uniformly distributed in $[0,10]$ (vector of initial wealth in each asset).
The largest possible position in any security is set to $u_i=20\%$.  
Transaction costs are known with
$\nu_t(i)=\mu_t(i)$
obtained sampling from the distribution of the random variable 
$0.08+0.06\cos(\frac{2\pi}{T} U_T )$ where $U_T$ is a random variable
with a discrete distribution over the set of integers $\{1,2,\ldots,T\}$.
Our four instances of the portfolio problem are obtained taking
for $(M,T,n)$ the combinations of values $(100,10,50)$, $(100,30,50)$, $(50,20,50)$, and 
$(50,40,10)$.
All linear subproblems of the forward and backward passes
are solved numerically using Mosek solver \cite{mosek} and
for ISDDP, we solve approximately these subproblems limiting the number of iterations 
of Mosek solver as indicated in Table \ref{tablenumberiter0} in the Appendix.
The strategy given in this table is (as indicated in Remark \ref{remchoicepssto}) to increase the accuracy
(or, equivalently, increase the maximal number of iterations allowed for Mosek solver)
of the solutions to subproblems as ISDDP iteration increases and 
for a given iteration of ISDDP, to increase the accuracy 
(or, equivalently, increase the maximal number of iterations allowed for Mosek solver)
of the solutions to subproblems as the number of stages increases from $t=2$ to $t=T$, knowing that we
solve exactly the subproblems for the last stage $T$ and for the first stage $t=1$.

SDDP and ISDDP were implemented in Matlab and the code was run on a Xeon E5-2670 processor with 
384 GB of RAM. 
For a given instance, SDDP and ISDDP were run using the same set of sampled scenarios 
along iterations. We stopped  SDDP algorithm when the gap is $<10\%$
and run ISDDP for the same number of iterations.\footnote{The gap is defined as $\frac{Ub-Lb}{Ub}$ where $Ub$ and $Lb$ correspond to upper and lower bounds, respectively. 
Though the portfolio problem is a maximization 
problem (of the mean income), we have rewritten it as a minimization problem (of the mean loss), of form \eqref{firststodp}, \eqref{secondstodp}.
The lower bound $Lb$ is the optimal value of the first stage problem and the 
upper  bound $Ub$ is the upper end of a 97.5\%-one-sided confidence interval on the optimal value for $N=100$ policy realizations, see 
\cite{shapsddp} for a detailed discussion on this stopping criterion.}

On our four instances, we then simulate the policies obtained with SDDP and ISDDP
on a set of 500 scenarios of returns. The gap between the two policies on these scenarios and the CPU time reduction
using ISDDP are given in 
Table \ref{tablecpuisddp}. In this table, the gap is defined by 
$100 \frac{\tt{Cost ISDDP}-\tt{Cost SDDP}}{\tt{Cost SDDP}}$
where 
${\tt{Cost ISDDP}}$ and $\tt{Cost SDDP}$ are respectively the mean cost for ISDDP and
SDDP policies on the 500 simulated scenarios and the CPU time reduction is given by
$100 \frac{\tt{Time SDDP}-\tt{Time ISDDP}}{\tt{Time SDDP}}$
where $\tt{Time SDDP}$ and $\tt{Time ISDDP}$ correspond to the time needed to 
compute SDDP and ISDDP policies (before running the Monte Carlo simulation), respectively. 

On all instances the gap is relatively small and ISDDP policy is computed faster than SDDP policy.

\begin{table}[H]
\centering
\begin{tabular}{|c|c|c|c|c|}
\hline
$M$ & $T$ & $n$ & Gap (\%) & CPU time reduction (\%) \\
\hline
50 & 20 &   50  & 0.1 & 6.2 \\
\hline
50 & 40 &   10  & 4.2 & 11.1 \\
\hline 
100 & 10 &   50  & 0.8 & 6.5 \\
\hline 
100 & 30 &   50  & 3.4 & 6.4 \\
\hline 
\end{tabular}
\caption{Empirical gap between SDDP and ISDDP policies and CPU time reduction for ISDDP over SDDP.}\label{tablecpuisddp}
\end{table}

\par More precisely, we report in Figure \ref{fig1sddp} (for instances with $(M,T,n)=(100,10,50)$ and $(M,T,n)=(100,30,50)$)
and Figure \ref{fig2sddp} (for instances with $(M,T,n)=(50,20,50)$
and $(M,T,n)=(50,40,10)$) three outputs along the iterations
of SDDP and ISDDP: the cumulative CPU time (in seconds), the number of iterations needed for Mosek LP solver to solve
all backward and forward subproblems, and the upper and lower bounds on the optimal value computed
by the methods (note that the upper bounds are only computed from iteration 100 on, because the past
$N=100$ iterations are used to compute them). 

\par These experiments (i) show that it is possible to 
obtain a near optimal policy quicker than SDDP solving approximately some subproblems in SDDP and (ii) confirm that ISDDP computes a valid lower bound since
first stage subproblems are solved exactly. For the first iterations, this lower bound can however be
distant from SDDP lower bound (see for instance the bottom left plots of Figures 
\ref{fig1sddp} and \ref{fig2sddp}). However, both  SDDP and ISDDP lower and upper bounds are 
quite close after 200 iterations, even if Mosek LP solver uses much less iterations to solve
the subproblems with ISDDP (see the middle plots of Figures \ref{fig1sddp}, \ref{fig2sddp}). 
The total CPU time needed by ISDDP is significantly inferior but this CPU time reduction decreases
when the number of iterations increases. If many iterations are required to solve the problem,
after a few hundreds iterations backward and forward subproblems are solved in similar CPU time for SDDP and ISDDP and the total CPU time
reduction starts to stabilize.

\begin{figure}
\centering
\begin{tabular}{cc}
\includegraphics[scale=0.6]{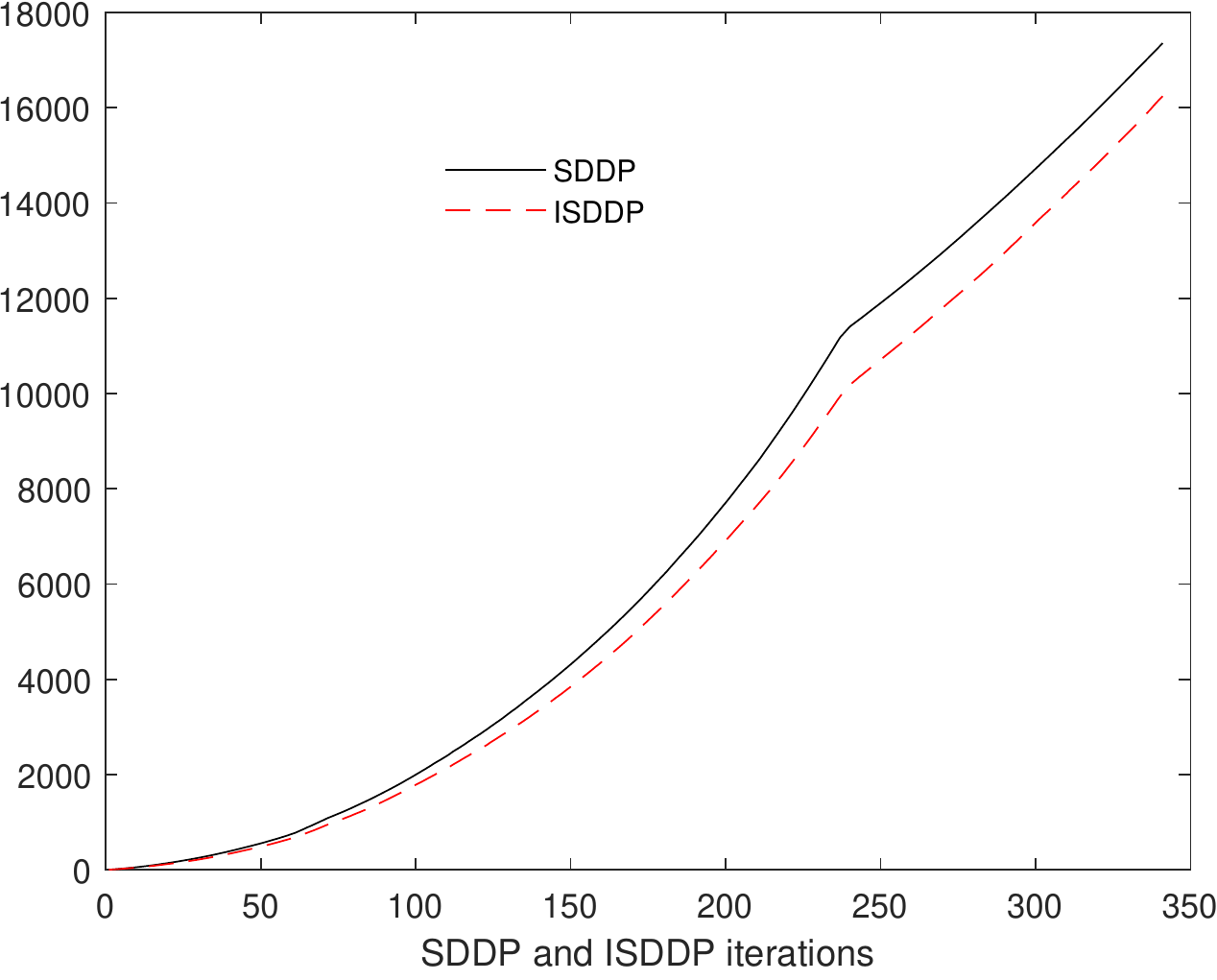}
&
\includegraphics[scale=0.6]{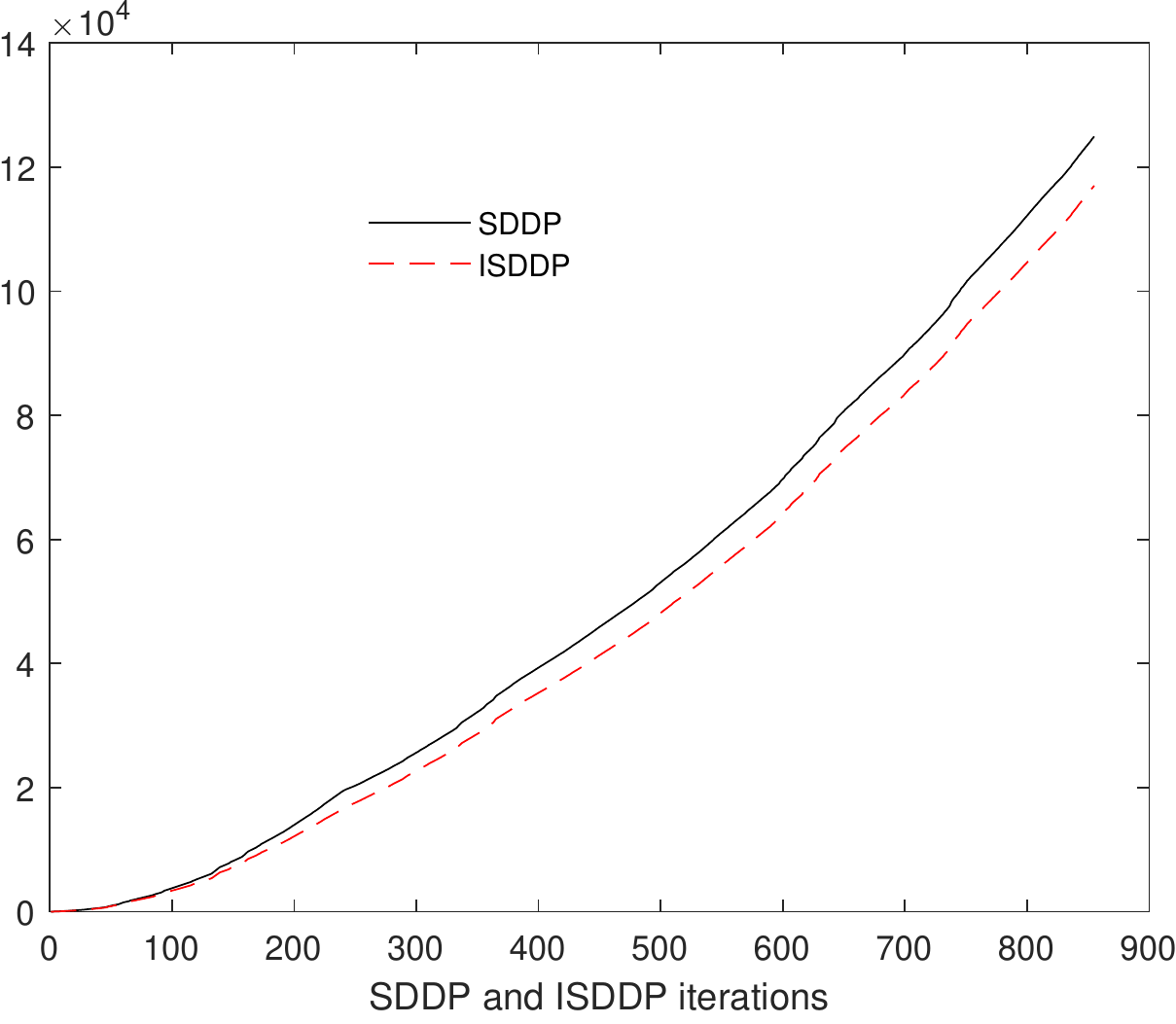}
\\
\includegraphics[scale=0.6]{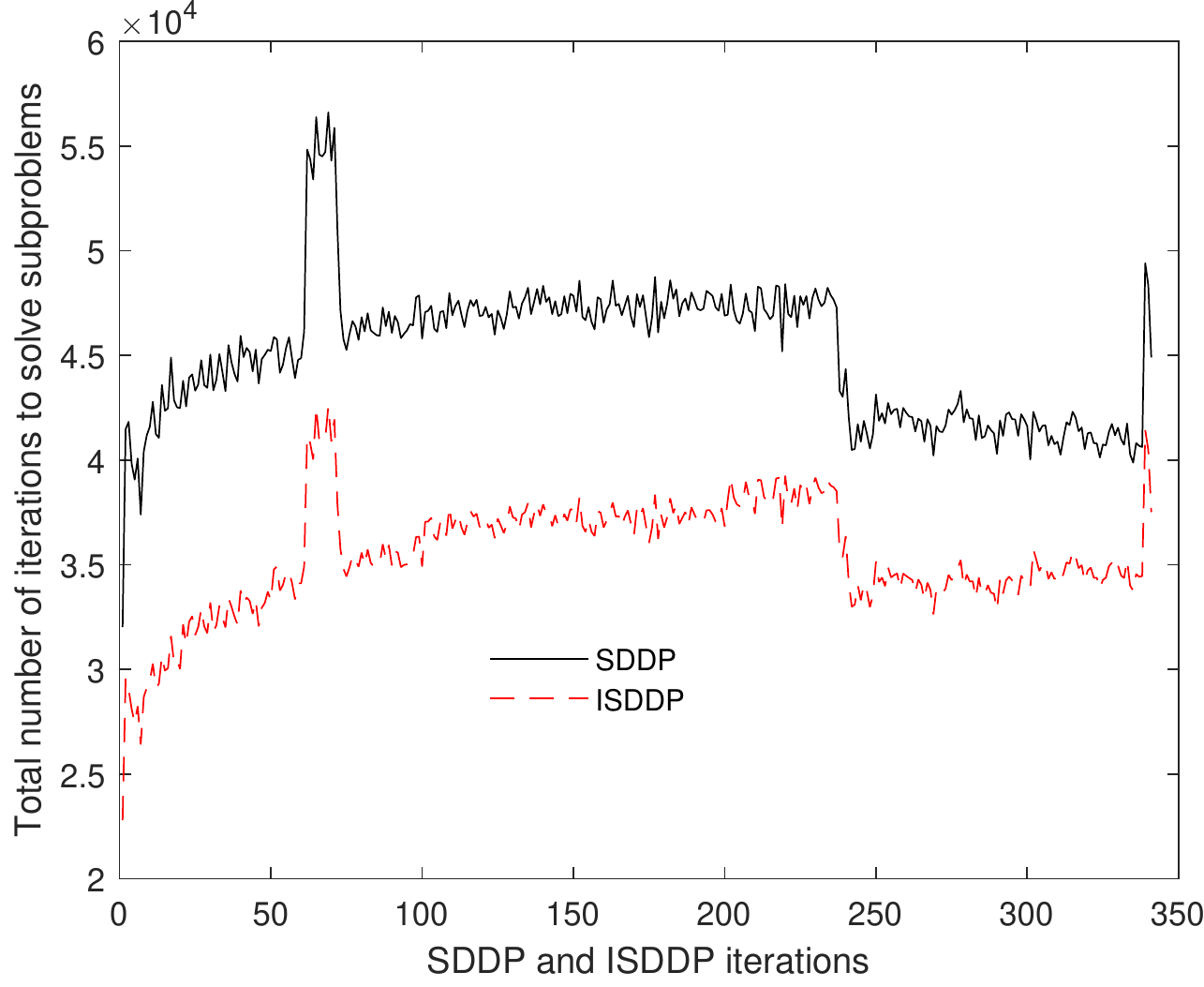}
&
\includegraphics[scale=0.6]{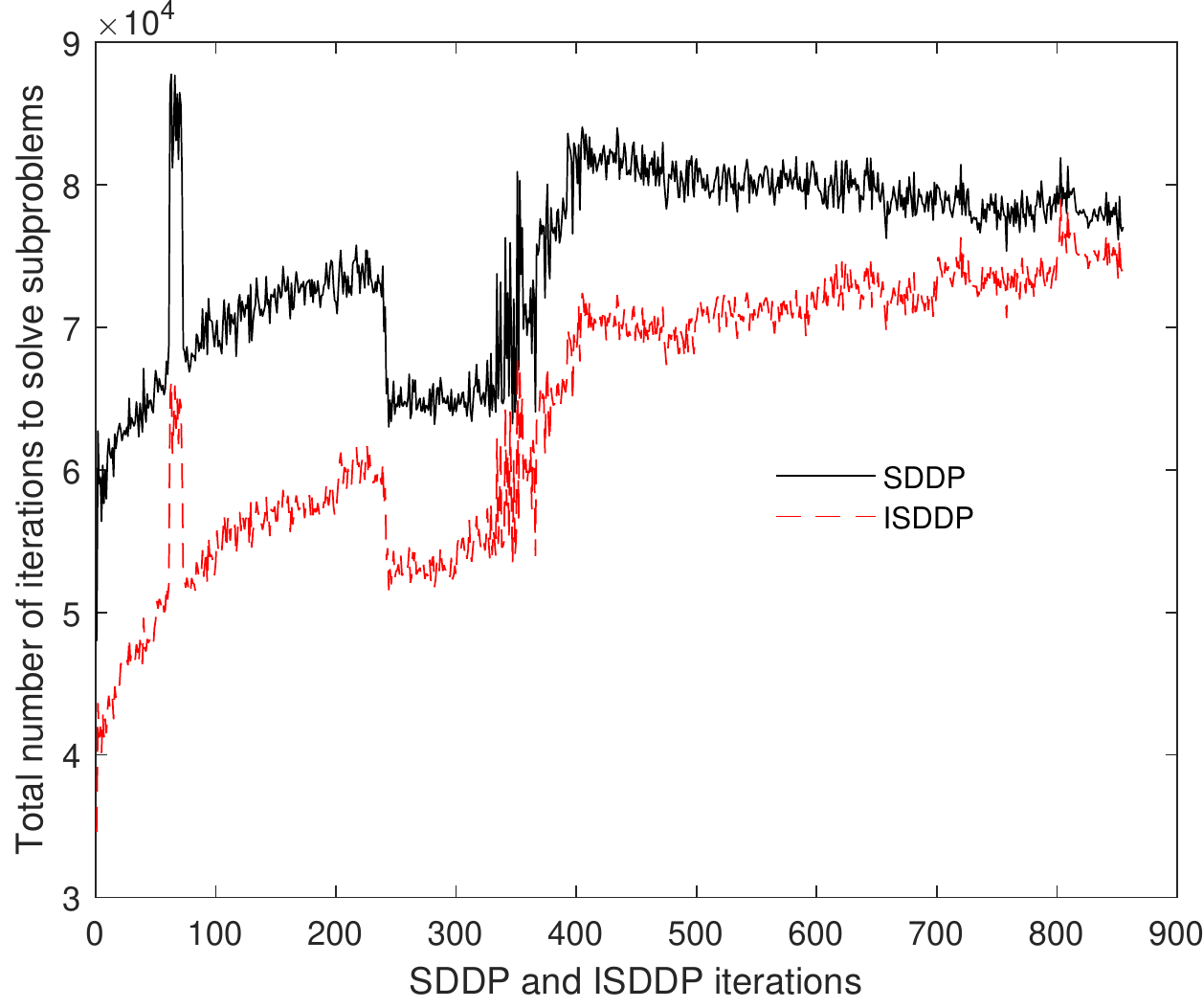}
\\
\includegraphics[scale=0.6]{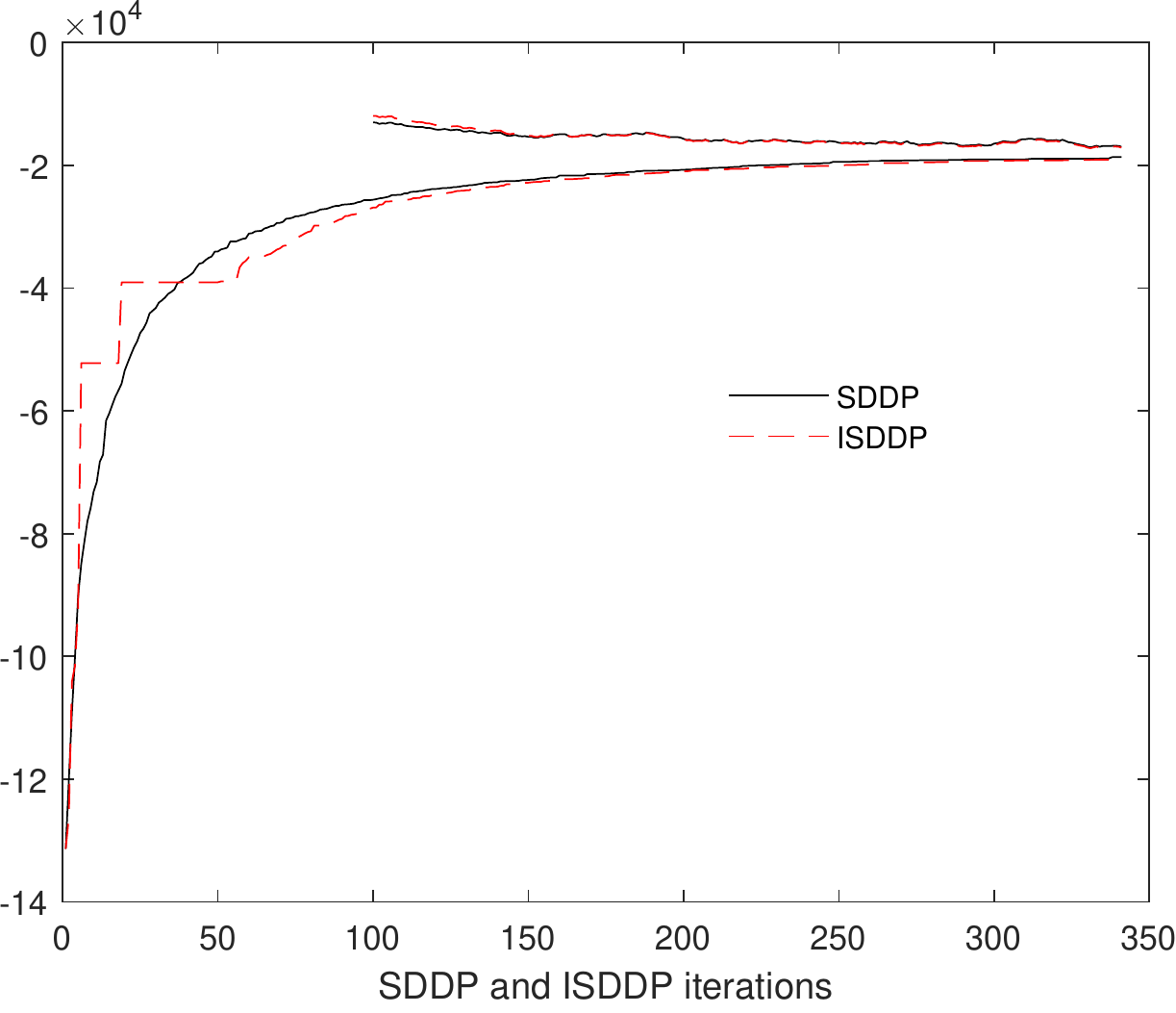}
&
\includegraphics[scale=0.6]{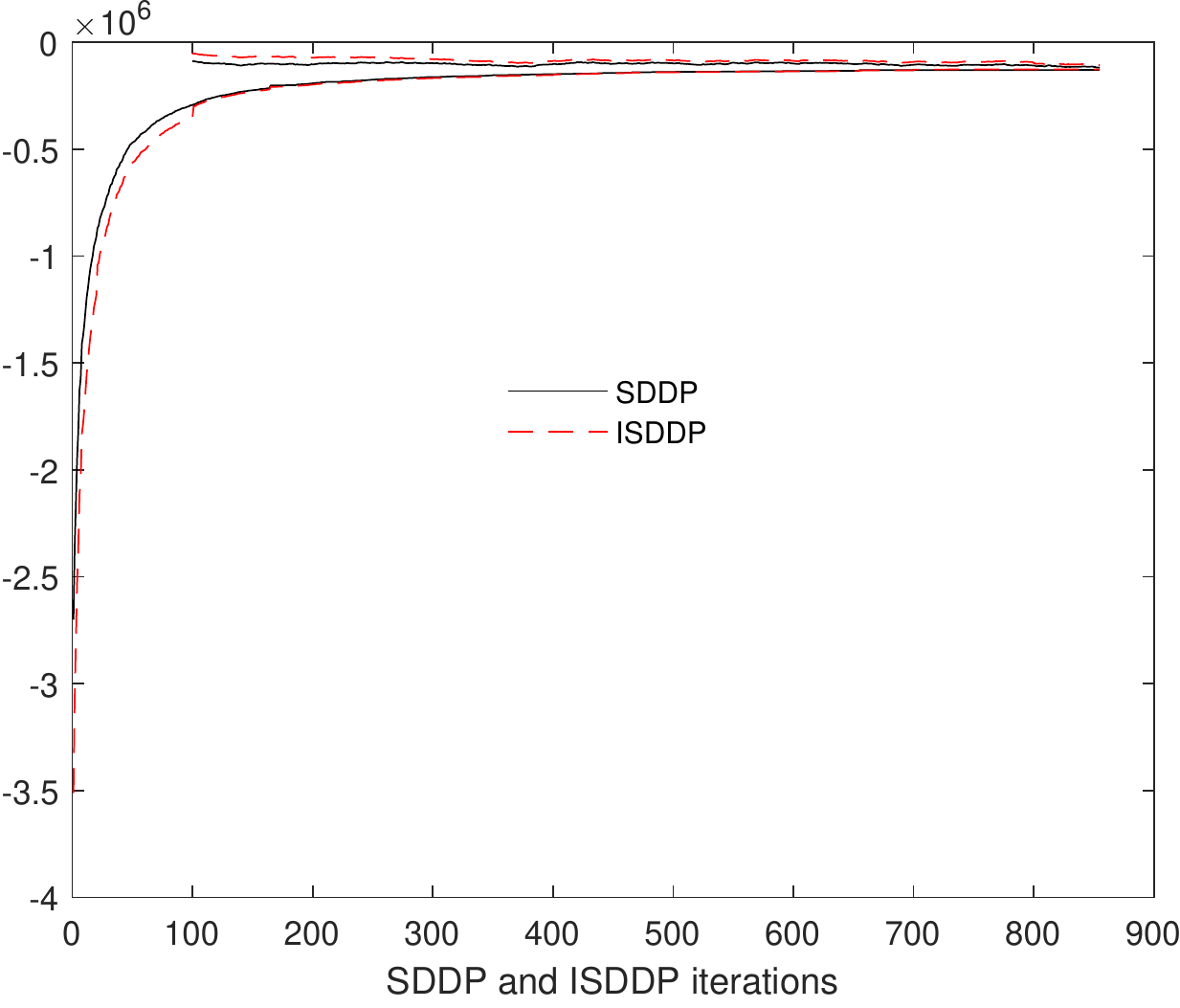}
\end{tabular}
\caption{Top plots: cumulative CPU time (in seconds), middle plots: total number of iterations to solve subproblems, bottom 
plots: upper and lower bounds. Left plots: $M=100,$ $T=10$, $n=50$, right plots: $M=100,$ $T=30$, and $n=50$.}
\label{fig1sddp}
\end{figure}

\begin{figure}
\centering
\begin{tabular}{cc}
\includegraphics[scale=0.6]{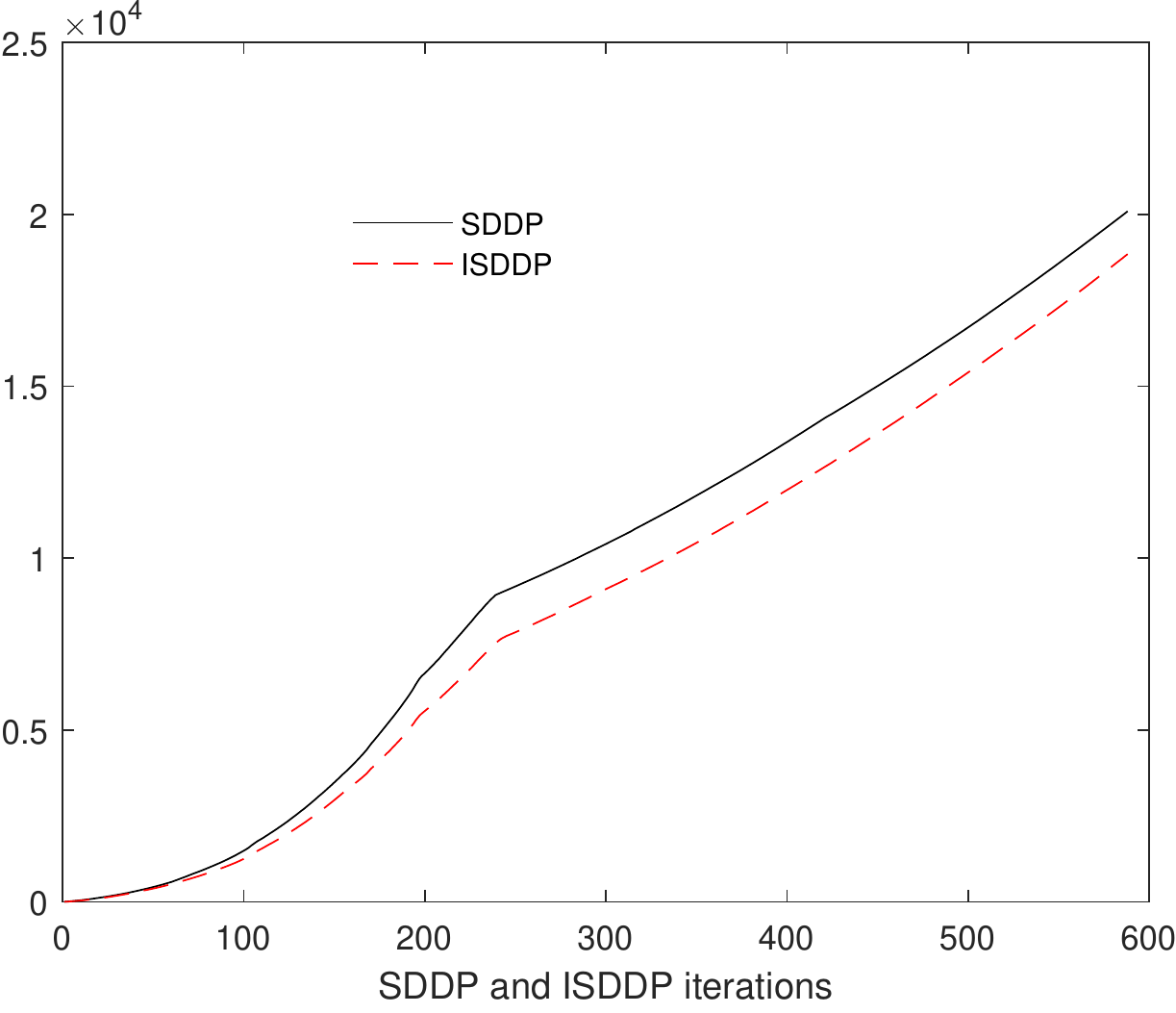}
&
\includegraphics[scale=0.6]{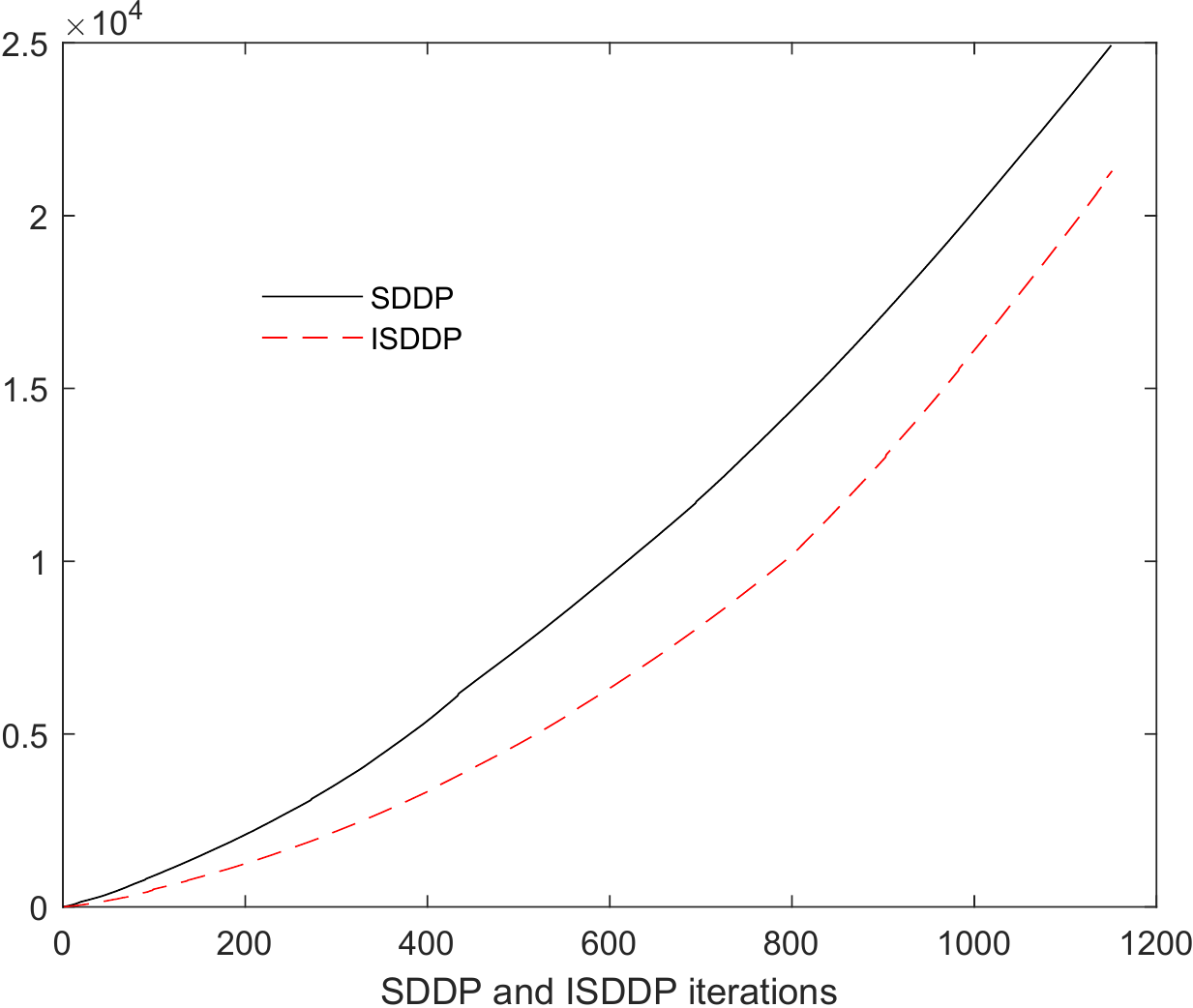}
\\
\includegraphics[scale=0.6]{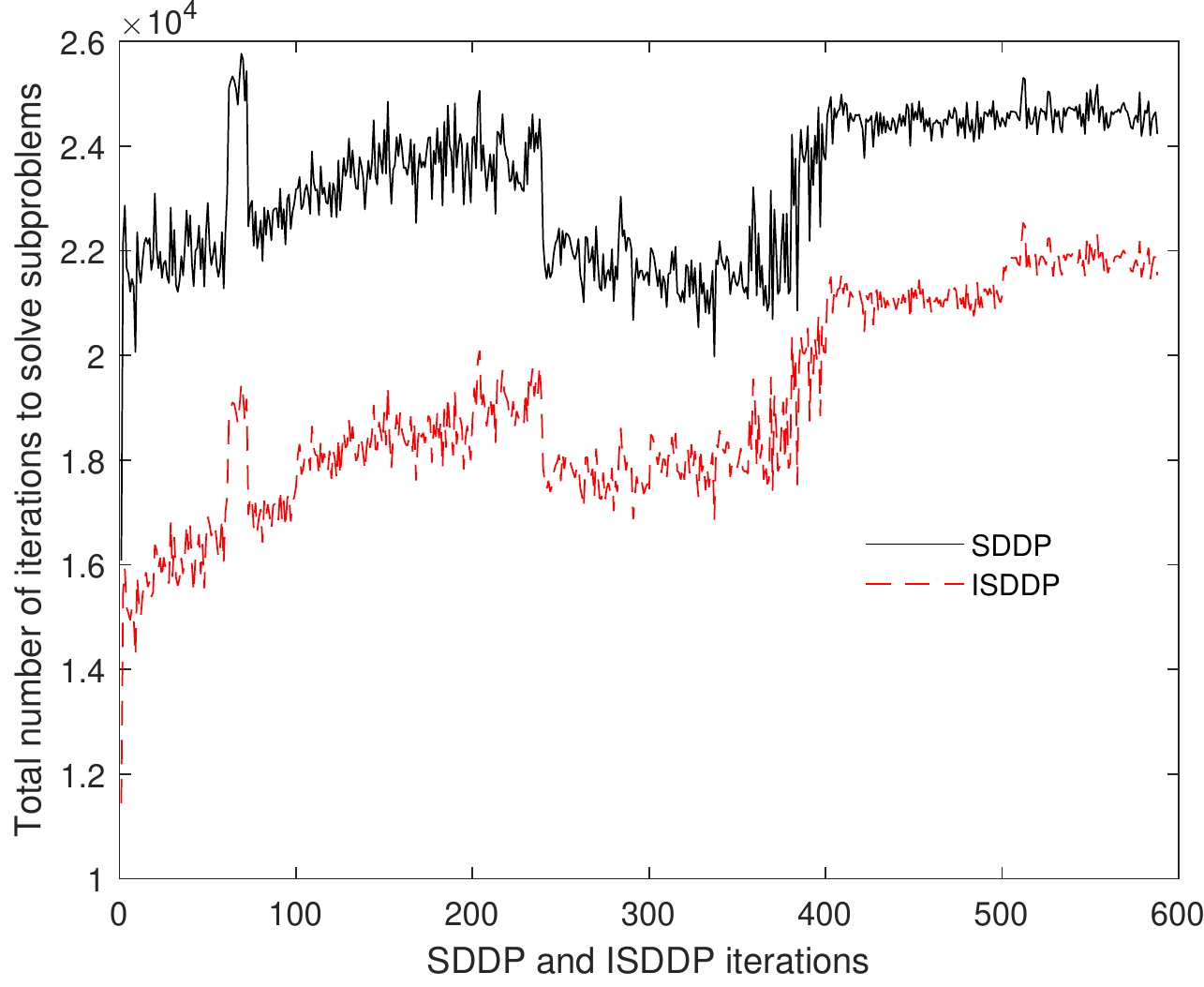}
&
\includegraphics[scale=0.6]{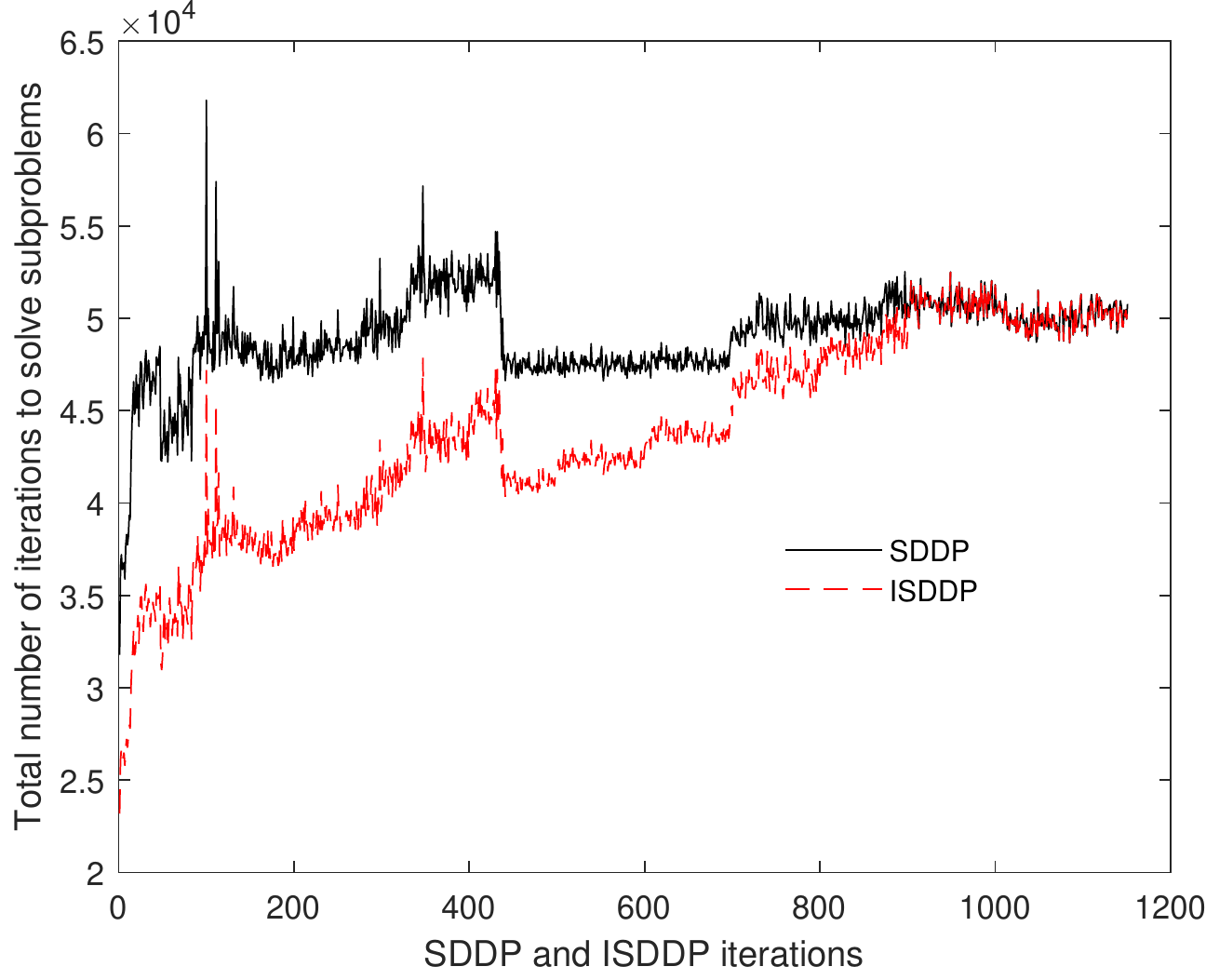}
\\
\includegraphics[scale=0.6]{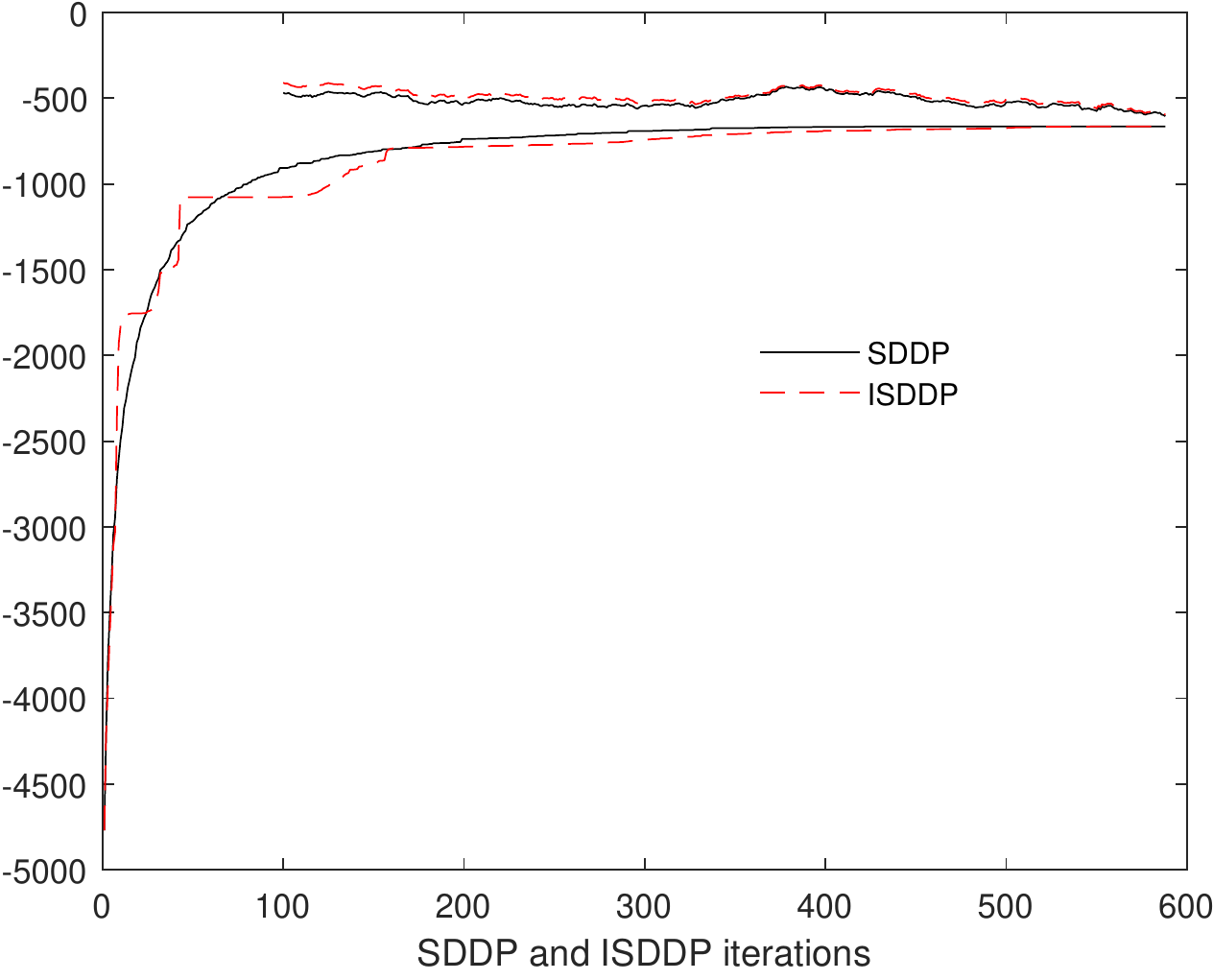}
&
\includegraphics[scale=0.6]{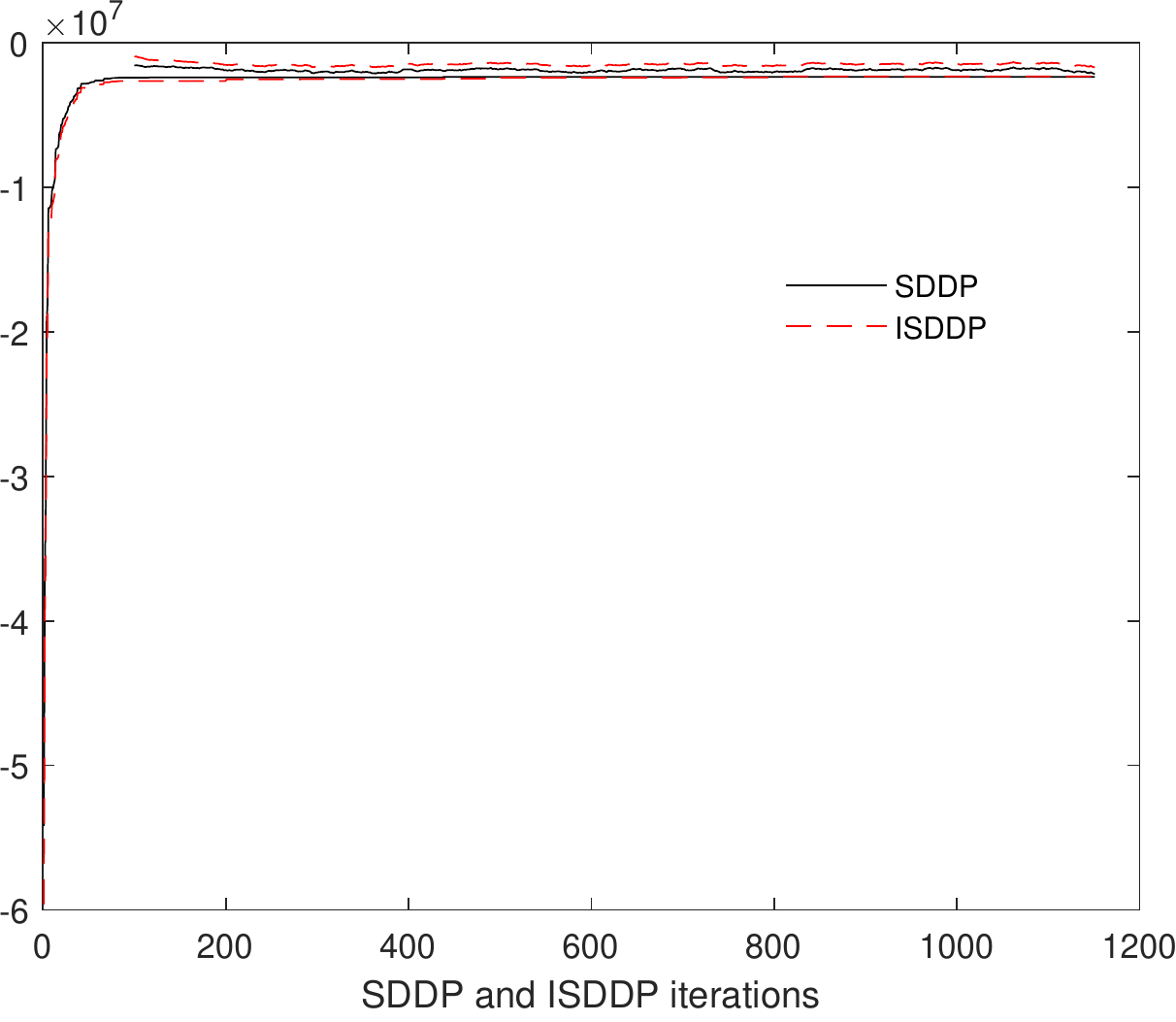}
\end{tabular}
\caption{Top plots: cumulative CPU time (in seconds), middle plots: total number of iterations to solve subproblems, bottom 
plots: upper and lower bounds. Left plots: $M=50,$ $T=20$, $n=50$, right plots: $M=50,$ $T=40$, and $n=10$.}\label{fig2sddp}
\end{figure}

\section{Conclusion}

We have introduced the first inexact variant of SDDP to solve stochastic convex  dynamic programming
equations. We have shown that the method solves these equations for vanishing noises.

It would be interesting to consider the following extensions of this work:
\begin{itemize}
 \item[(i)] derive inexact cuts for problems with nondifferentiable cost and constraint functions;
 \item[(ii)] build cuts in the backward pass on the basis of approximate solutions which are not 
necessarily feasible;
\item[(iii)] apply ISDDP to other real-life applications, testing several strategies for the 
sequence of error terms $(\delta_t^k, \varepsilon_t^k)$ or the maximal number of iterations 
for the LP solver used to solve the subproblems along the iterations of ISDDP.
\end{itemize}

%
%
%
%
%

\section*{Appendix}

\par {\textbf{Proof of Theorem \ref{convisddplp}.}} 

\par (i) We show \eqref{lowerbounds} for $t=2,\ldots,T+1$, and all node $n$ of stage $t-1$ by backward induction on $t$. The relation
holds for $t=T+1$. Now assume that it holds for $t+1$ for some $t \in \{2,\ldots,T\}$. Let us show that it holds for $t$.
Take a node $n$ of stage $t-1$.
Observe that the sequence $\mathcal{Q}_t( x_n^k ) - \mathcal{Q}_t^k ( x_n^k )$ is almost surely bounded and nonnegative.
Therefore it has almost surely a nonnegative limit inferior and a finite limit superior.
Let $\mathcal{S}_n=\{k : n_t^k =n\}$ be the iterations where the sampled scenario passes through node $n$.
For $k \in \mathcal{S}_n$ we have $0 \leq \mathcal{Q}_t( x_n^k ) - \mathcal{Q}_t^k ( x_n^k )$ and 
{\small{
\begin{equation}\label{eqproofconvisddp1}
\begin{array}{l}
\mathcal{Q}_t( x_n^k ) - \mathcal{Q}_t^k ( x_n^k )  \leq   \mathcal{Q}_t( x_n^k ) - \mathcal{C}_t^k ( x_n^k ) \\
 \leq  \displaystyle {\bar \varepsilon} + \sum_{m \in C(n)} p_m  \Big[ \mathfrak{Q}_t(x_n^k , \xi_m ) -  {\underline{\mathfrak{Q}}}_t^{k}(x_n^k , \xi_m )\Big] \\
 \leq  \displaystyle {\bar \varepsilon} + \sum_{m \in C(n)} p_m  \Big[ \mathfrak{Q}_t(x_n^k , \xi_m ) -  {\underline{\mathfrak{Q}}}_t^{k-1}(x_n^k , \xi_m )\Big] \\
 \leq  \displaystyle {\bar \varepsilon} + \delta_t^k + \sum_{m \in C(n)} p_m  \Big[ \mathfrak{Q}_t(x_n^k , \xi_m ) -  \langle c_m , x_m^k \rangle - \mathcal{Q}_{t+1}^{k-1}( x_m^k ) \Big] \\
 \leq  \displaystyle {\bar \varepsilon} + {\bar {\delta}} + \sum_{m \in C(n)} p_m  \Big[ \underbrace{\mathfrak{Q}_t(x_n^k , \xi_m ) -  \langle c_m , x_m^k \rangle - \mathcal{Q}_{t+1}( x_m^k ) }_{\leq 0\mbox{ by definition of }\mathfrak{Q}_t\mbox{ and }x_m^k} + \mathcal{Q}_{t+1}( x_m^k )  -    \mathcal{Q}_{t+1}^{k-1}( x_m^k ) \Big] \\
 \leq  \displaystyle {\bar \varepsilon} + {\bar {\delta}} + \sum_{m \in C(n)} p_m  \Big[ \mathcal{Q}_{t+1}( x_m^k )  -    \mathcal{Q}_{t+1}^{k-1}( x_m^k ) \Big].
\end{array}
\end{equation}
}}

Using the induction hypothesis, we have for every $m \in C(n)$ that
$$
\varlimsup_{k \rightarrow +\infty} \mathcal{Q}_{t+1}( x_{m}^k ) - \mathcal{Q}_{t+1}^k( x_{m}^k ) \leq ({\bar \delta}  +  {\bar{\varepsilon}})(T-t).
$$
In virtue of Lemma \ref{limsuptechlemma}, this implies
\begin{equation}\label{convsupkm1}
\varlimsup_{k \rightarrow +\infty} \mathcal{Q}_{t+1}( x_{m}^k ) - \mathcal{Q}_{t+1}^{k-1}( x_{m}^k ) \leq ({\bar \delta}  +  {\bar{\varepsilon}})(T-t),
\end{equation}
which, plugged into \eqref{eqproofconvisddp1}, gives
\begin{equation}\label{eqconvinsn}
\varlimsup_{k \rightarrow +\infty, k \in \mathcal{S}_n} \mathcal{Q}_t( x_n^k ) - \mathcal{Q}_t^k ( x_n^k ) \leq ({\bar \delta}  +  {\bar{\varepsilon}})(T-t+1).
\end{equation}

Now let us show by contradiction that 
$
\varlimsup_{k \rightarrow +\infty} \mathcal{Q}_t( x_n^k ) - \mathcal{Q}_t^k ( x_n^k ) \leq ({\bar \delta}  +  {\bar{\varepsilon}})(T-t+1)$.
If this relation  does not hold
then there exists $\varepsilon_0>0$ such that
there is an infinite set of iterations $k$ satisfying 
$\mathcal{Q}_t( x_n^k ) - \mathcal{Q}_t^k ( x_n^k ) > ({\bar \delta}  +  {\bar{\varepsilon}})(T-t+1) + \varepsilon_0$
 and by monotonicity, there is also an infinite set of iterations $k$ in the set
 $K=\{k \geq 1 : \mathcal{Q}_t( x_n^k ) - \mathcal{Q}_t^{k-1} ( x_n^k ) > ({\bar \delta}  +  {\bar{\varepsilon}})(T-t+1) + \varepsilon_0\}$.
Let $k_1<k_2<...$ be these iterations: $K=\{k_1,k_2,\ldots,\}$.
Let $y_n^k$ be the random variable which takes the value 1 if $k \in \mathcal{S}_n$ and $0$ otherwise.
Due to Assumptions (A0)-(A2), random variables $y_n^{k_1},y_n^{k_2},\ldots,$ are i.i.d. and have the distribution of $y_n^1$. Therefore by the Strong Law of Large Numbers we get 
$
\frac{1}{N} \displaystyle \sum_{j=1}^N y_n^{k_j}  \xrightarrow{N \rightarrow +\infty} \mathbb{E}[y_n^1]>0 \mbox{ a.s.}
$
Now let $z_1<z_2<\ldots$ be the iterations in $\mathcal{S}_n$: $\mathcal{S}_n=\{z_1,z_2,\ldots\}$.
Relation \eqref{eqconvinsn} can be written 
$
\varlimsup_{k \rightarrow +\infty} \mathcal{Q}_t( x_n^{z_k} ) - \mathcal{Q}_t^{z_k} ( x_n^{z_k} ) \leq ({\bar \delta}  +  {\bar{\varepsilon}})(T-t+1),
$
which, using Lemma \ref{limsuptechlemma}, implies
$
\varlimsup_{k \rightarrow +\infty} \mathcal{Q}_t( x_n^{z_k} ) - \mathcal{Q}_t^{z_{k-1}} ( x_n^{z_k} ) \leq ({\bar \delta}  +  {\bar{\varepsilon}})(T-t+1).
$
Using the fact that $z_k \geq z_{k-1}+1$, we deduce that
$
\varlimsup_{k \rightarrow +\infty, k \in \mathcal{S}_n} \mathcal{Q}_t( x_n^k ) - \mathcal{Q}_t^{k-1} ( x_n^k )  = 
\varlimsup_{k \rightarrow +\infty} \mathcal{Q}_t( x_n^{z_k} ) - \mathcal{Q}_t^{z_{k}-1} ( x_n^{z_k} ) 
 \leq  
\varlimsup_{k \rightarrow +\infty} \mathcal{Q}_t( x_n^{z_k} ) - \mathcal{Q}_t^{z_{k-1}} ( x_n^{z_k} ) \leq ({\bar \delta}  +  {\bar{\varepsilon}})(T-t+1).
$
Therefore, there can only be a finite number of iterations that are both in $K$ and in $\mathcal{S}_n$. This gives
$
\frac{1}{N} \displaystyle \sum_{j=1}^N y_n^{k_j}  \xrightarrow{N \rightarrow +\infty} 0 \mbox{ a.s.}
$
and we obtain the desired contradiction.

\par (ii) Using \eqref{eqproofconvisddp1}, we obtain for every $t=2,\ldots,T$, and every node $n$ of stage $t-1$, that
{\small{
\begin{equation}
0 \leq   \sum_{m \in C(n)} p_m \Big[ c_m^T x_m^k   + \mathcal{Q}_{t+1}( x_m^k )\Big]  - \mathcal{Q}_t( x_n^k ) 
\leq {\bar \delta}  +  {\bar{\varepsilon}} + \sum_{m \in C(n)} p_m \Big[ \mathcal{Q}_{t+1}( x_m^k ) - \mathcal{Q}_{t+1}^{k-1}( x_m^k ) \Big].
\end{equation}
}}
Therefore
$
\varliminf_{k \rightarrow +\infty} \sum_{m \in C(n)} p_m \Big[ c_m^T x_m^k   + \mathcal{Q}_{t+1}( x_m^k ) \Big] -  \mathcal{Q}_t( x_n^k )  \geq 0 
$
and using \eqref{convsupkm1} we get
$$
\varlimsup_{k \rightarrow +\infty} \sum_{m \in C(n)} p_m \Big[ c_m^T x_m^k   + \mathcal{Q}_{t+1}( x_m^k ) \Big]  -  \mathcal{Q}_t(x_n^k ) \leq  ({\bar \delta}  +  {\bar{\varepsilon}})(T-t+1).
$$

\par (iii) We have 
\begin{equation}\label{q1iii}
\begin{array}{lll}
\mathcal{Q}_1 ( x_0 ) \geq {\underline{\mathfrak{Q}}}_1^{k-1} ( x_0 , \xi_1 )  & \geq &
c_1^T x_1^k + \mathcal{Q}_2^{k-1}( x_1^k ) - \delta_1^k \\
& \geq & -{\bar{\delta}} + \mathcal{Q}_1 ( x_0 ) + \mathcal{Q}_2^{k-1}( x_1^k ) - \mathcal{Q}_2 ( x_1^k ). 
\end{array}
\end{equation}
Using \eqref{q1iii} and  \eqref{convsupkm1} with $t=1$, we obtain (iii).\\

\par {\textbf{Additional parameters for ISDDP.}} 
For ISDDP, the maximal
number of iterations allowed for Mosek LP solver to solve subproblems along the iterations of 
ISDDP is given in Table \ref{tablenumberiter0}.

\begin{table}
\centering
\begin{tabular}{c}
\begin{tabular}{|c|c|c|c|}
\hline
\begin{tabular}{c}ISDDP\\iteration\end{tabular} &  $[1,20]$ & $[21,50]$ &  $[51,100]$    \\
\hline
\begin{tabular}{l}
LP solver\\maximal\\
number of\\
iterations at $t$
\end{tabular}
& $\left \lceil (0.4+ 0.6 \frac{(t-2)}{T-2})I_{\max} \right \rceil$ &
$\left \lceil (0.45 + 0.55\frac{(t-2)}{T-2})I_{\max} \right \rceil$  & 
$\left \lceil (0.5 + 0.5\frac{(t-2)}{T-2})I_{\max} \right \rceil$ \\
\hline
\end{tabular}
\\
\begin{tabular}{|c|c|c|c|}
\hline
\begin{tabular}{c}ISDDP\\iteration\end{tabular}  & $[101,200]$  &  $[201,300]$ & $[301,400]$  \\
\hline
\begin{tabular}{l}
LP solver\\maximal\\
number of\\
iterations at $t$
\end{tabular}
&
$\left \lceil(0.55+ 0.45\frac{(t-2)}{T-2})I_{\max} \right \rceil$
& $\left \lceil(0.6+ 0.4\frac{(t-2)}{T-2})I_{\max} \right \rceil$ & 
$\left \lceil (0.65+ 0.35\frac{(t-2)}{T-2})I_{\max} \right \rceil$ \\
\hline
\end{tabular}
\\
\begin{tabular}{|c|c|c|c|}
\hline
\begin{tabular}{c}ISDDP\\iteration\end{tabular}  & $[401,500]$  &  $[501,600]$ & $[601,700]$  \\
\hline
\begin{tabular}{l}
LP solver\\maximal\\
number of\\
iterations at $t$
\end{tabular}
&
$\left \lceil(0.7+ 0.3\frac{(t-2)}{T-2})I_{\max} \right \rceil$
& $\left \lceil(0.75+ 0.25\frac{(t-2)}{T-2})I_{\max} \right \rceil$ & 
$\left \lceil (0.8+ 0.2\frac{(t-2)}{T-2})I_{\max} \right \rceil$ \\
\hline
\end{tabular}
\\
\begin{tabular}{|c|c|c|c|}
\hline
\begin{tabular}{c}ISDDP\\iteration\end{tabular}  & $[701,800]$  &  $[801,900]$ & $>900$  \\
\hline
\begin{tabular}{l}
LP solver\\maximal\\
number of\\
iterations at $t$
\end{tabular}
&
$\left \lceil(0.85+ 0.15\frac{(t-2)}{T-2})I_{\max} \right \rceil$
& $\left \lceil(0.9+ 0.1\frac{(t-2)}{T-2})I_{\max} \right \rceil$ & 
$I_{\max}$ \\
\hline
\end{tabular}
\end{tabular}
\caption{Maximal number of iterations for Mosek LP solver for solving backward and forward passes 
subproblems
as a function of stage $t \geq 2$, ISDDP iteration, and the number $I_{\max}$ of iterations 
used to solve subproblems with SDDP with high accuracy.
In this table, $\left \lceil x \right \rceil$ is the smallest integer 
larger than or equal to $x$.}\label{tablenumberiter0}
\end{table}

\bibliographystyle{siamplain}
\bibliography{Bibliography}

\end{document}